\documentclass[EJP, preprint]{ejpecp} 

\usepackage[dvipsnames]{xcolor}
\usepackage{tikz}
\usetikzlibrary{intersections,shapes,trees,arrows,spy,positioning,decorations,arrows.meta,decorations.pathreplacing,patterns,calc}
\usepackage{dutchcal} 
\usepackage{relsize}
\usepackage{dsfont}
\usepackage{subcaption} 
\usepackage{enumitem}
\usepackage{makecell}
\usepackage{setspace}
\usepackage{accents} 
\usepackage{soul}
\usepackage[normalem]{ulem}
\usepackage[low-sup]{subdepth} 
\usepackage[most]{tcolorbox}
\usepackage{hyperref}
\usepackage{pifont}
\usepackage{forest}
\usetikzlibrary{arrows.meta, positioning, matrix, decorations.pathreplacing, calc}

\SHORTTITLE{Genealogical structures under interactive neutral reproduction}

\TITLE{Genealogical structures under interactive neutral reproduction: factorial moment duality via a Frankenstein process}

\AUTHORS{%
  Ellen~Baake\footnote{Bielefeld University, Germany.
    \EMAIL{ebaake@techfak.uni-bielefeld.de}}\orcid{0000-0001-8970-2102}
  \and 
  Fernando~Cordero\footnote{BOKU University, Austria. \EMAIL{fernando.cordero@boku.ac.at}}\orcid{0000-0001-5959-0787}
  \and
    Hannah~Dopmeyer\footnote{Bielefeld University, Germany.
    \EMAIL{hannah.dopmeyer@uni-bielefeld.de}}\orcid{0009-0001-5505-6699}}
    

\AMSSUBJ{60K35; 92D15; 60J25; 60J27; 60J70} 

\ABSTRACT{
We establish a genealogical framework for an existing analytical moment duality between a Wright–Fisher type SDE and a counting process with interaction. To achieve this, we construct a finite-population Moran model featuring interactive neutral reproduction as a novel mechanism. In the corresponding events, an individual, regardless of its own type, can only reproduce if a randomly encountered partner is of the ``fit'' type. This Moran model has a relatively simple counting process as its factorial moment dual, whose genealogical meaning appears to be cryptic: after all, the line-counting process of the natural genealogical process of the model, namely the ancestral influence graph (AIG), exhibits a complex hierarchical structure not reflected in the factorial moment dual. Since moment duality is a property in expectation, we are allowed to systematically remove information from the AIG and merge different realizations of the ancestry. We call the result the \emph{Frankenstein process}. Based on this, we establish the factorial moment duality from a genealogical perspective. The moment duality in the diffusion limit follows in a natural way.
}

\newcommand{\longtilde}[1]{%
  \raisebox{1.4ex}{$\sim$}\hspace{-0.9em}#1%
  }

\newcommand{\smalllongtilde}[1]{%
  \raisebox{1ex}{\scalebox{0.7}{$\sim$}}\hspace{-0.7em}#1%
}

\newcommand{\cupdot}{\mathbin{\dot{\cup}}}

\renewcommand{\widehat}[1]{#1}

\newcommand{\bs}[1]{\boldsymbol{#1}}

\newcommand{\Eb} {{\mathbb E}}
\newcommand{\Nb} {{\mathbb N}}
\newcommand{\Pb} {{\mathbb P}}
\newcommand{\Rb} {{\mathbb R}}

\newcommand{\As} {{\mathcal A}}
\newcommand{\as} {{\mathcal a}}
\newcommand{\Bs} {{\mathcal B}}
\newcommand{\Gs} {{\mathcal G}}
\newcommand{\gs} {{\mathcal g}}
\newcommand{\Ls} {{\mathcal L}}
\newcommand{\Ps} {{\mathcal P}}
\newcommand{\Ss} {{\mathcal S}}
\newcommand{\Ts} {{\mathcal T}}

\renewcommand{\r}{{r}}
\newcommand{\sP}{\reflectbox{$\mathcal{P}$}}
\newcommand{\smallsP}{\scalebox{0.75}{\reflectbox{$\mathcal{P}$}}}
\newcommand{\mirrorp}{\reflectbox{$p$}}
\newcommand{\sT}{\reflectbox{$\mathcal{T}$}}

\newcommand{\bX} {{\overline X}}
\newcommand{\dd} {{\rm d}}
\newcommand{\1} {{\mathds{1}}}

\definecolor{mypurple}{RGB}{105, 65, 125}
\definecolor{mygreen}{RGB}{150, 180, 80}

\newcommand{\narrowab}
{\begin{tikzpicture}[baseline={(0,-0.1)}]
\draw[-{Latex[length=1.8mm,fill=black]}] (0.5,0.15) -- (0,0.15);
\node[] at (0.25,-0.15) {\footnotesize {$\alpha, \beta$}};
\end{tikzpicture}}

\newcommand{\snarrowab}
{\scalebox{0.7}{
\begin{tikzpicture}[baseline={(0,-0.1)}]
\draw[-{Latex[length=1.8mm,fill=black]}] (0.5,0.15) -- (0,0.15);
\node[] at (0.25,-0.15) {\footnotesize {$\alpha, \beta$}};
\end{tikzpicture}}
}

\newcommand{\sarrowab}
{\begin{tikzpicture}[baseline={(0,-0.1)}]
\draw[-{Latex[length=1.8mm,fill=white]}] (0.5,0.15) -- (0,0.15);
\node[] at (0.25,-0.15) {\footnotesize{$ \alpha, \mathcal{B}$}};
\end{tikzpicture}}

\newcommand{\ssarrowab}
{\scalebox{0.7}{\begin{tikzpicture}[baseline={(0,-0.1)}]
\draw[-{Latex[length=1.8mm,fill=white]}] (0.5,0.15) -- (0,0.15);
\node[] at (0.25,-0.15) {\footnotesize{$ \alpha ,  \mathcal{B}$}};
\end{tikzpicture}}}

\newcommand{\carrowcba}{%
  \begin{tikzpicture}[baseline={(0,-0.1)}]
    \draw[-{Latex[length=1.8mm,fill=black]}] (0,0.15) -- (-0.4,0.15);
    \draw[dashed, -{Diamond[open,length=2.5mm]}] (0,0.15) -- (0.8,0.15);
    \node[] at (0.2,-0.15) {\footnotesize$\alpha, \, \, \beta, \, \, \gamma$};
  \end{tikzpicture}%
}

\newcommand{\scarrowcba}
 {\scalebox{0.7} {\begin{tikzpicture}[baseline={(0,-0.1)}]
    \draw[-{Latex[length=1.8mm,fill=black]}] (0,0.15) -- (-0.4,0.15);
    \draw[dashed, -{Diamond[open,length=2.5mm]}] (0,0.15) -- (0.8,0.15);
    \node[] at (0.2,-0.15) {\footnotesize$\alpha, \, \, \beta, \, \, \gamma$};
  \end{tikzpicture}}}

\newcommand{\carrowcbasetb}{%
  \begin{tikzpicture}[baseline={(0,-0.1)}]
    \draw[-{Latex[length=1.8mm,fill=black]}] (0,0.15) -- (-0.4,0.15);
    \draw[dashed, -{Diamond[open,length=2.5mm]}] (0,0.15) -- (0.8,0.15);
    \node[] at (0.2,-0.15) {\footnotesize $\alpha, \, (\beta, \gamma)$};
  \end{tikzpicture}%
  }

\newcommand{\varDiamond}{%
  \mathord{\tikz[baseline=-0.6ex] 
    \node[diamond, draw, line width=0.4pt,
          minimum width=2mm, minimum height=3mm,
          inner sep=0pt] {};
  }%
}

\newcommand{\ddownarrow}{\Downarrow}
\newcommand{\uuparrow}{\Uparrow}

\newcommand{\ineminus}[3]{\scalebox{0.8}{
  \begin{tikzpicture}[baseline={(0,-0.1)}]
    \draw[-{Latex[length=1.8mm,fill=black]}] (0,0.15) -- (-0.4,0.15);
    \draw[dashed, -{Diamond[open,length=2.5mm]}] (0,0.15) -- (0.8,0.15);
    \node[] at (0.2,-0.15) {\footnotesize ${#1}, ( {#2},{#3})$};
  \end{tikzpicture}}
}
\newcommand{\ine}[3]{\scalebox{0.8}{
  \begin{tikzpicture}[baseline={(0,-0.1)}]
    \draw[-{Latex[length=1.8mm,fill=black]}] (0,0.15) -- (-0.4,0.15);
    \draw[dashed, -{Diamond[open,length=2.5mm]}] (0,0.15) -- (0.8,0.15);
    \node[] at (0.2,-0.15) {\footnotesize ${#1}, {#2},{#3}$};
  \end{tikzpicture}}
}

\newcommand{\ftw}[2]
{\scalebox{0.8}{\begin{tikzpicture}[baseline={(0,-0.1)}]
\draw[-{Latex[length=1.8mm,fill=white]}] (0.5,0.15) -- (0,0.15);
\node[] at (0.25,-0.15) {\footnotesize{${#1},{#2}$}};
\end{tikzpicture}}}

\newcommand{\coal}[2]
{\scalebox{0.8}{\begin{tikzpicture}[baseline={(0,-0.1)}]
\draw[-{Latex[length=1.8mm,fill=black]}] (0.5,0.15) -- (0,0.15);
\node[] at (0.25,-0.15) {\footnotesize {${#1},{#2}$}};
\end{tikzpicture}}}

\newcommand{\ocoal}[2]{\mathrm{C}_{#1,#2}}
\newcommand{\oftw}[2]{\mathrm{F}_{#1,#2}}
\newcommand{\onint}[3]{\mathrm{I}_{#1,#2,#3}}
\newcommand{\onintminus}[3]{\mathrm{I}_{(#1,(#2, #3))}}
\newcommand{\omutr}[1]{\textrm{M}_{#1}^\times}
\newcommand{\omutb}[1]{\textrm{M}_{#1}^\circ}
\newcommand{\aigt}{\mathcal{A}_{[0,t]}}
\newcommand{\qaigt}[1]{\mathcal{A}_{[0,t]}^{-}}

\newcommand{\circa}
{\begin{tikzpicture}[baseline={(0,-0.08)}]
\node at (0.25, 0.15){\scriptsize $\circ$};
\node at (0.25,-0.15) {\footnotesize $\alpha$};
\end{tikzpicture}}

\newcommand{\scirca}
{\scalebox{0.7}{\begin{tikzpicture}[baseline={(0,-0.08)}]
\node at (0.25, 0.15){\scriptsize $\circ$};
\node at (0.25,-0.05) {\footnotesize $\alpha$};
\end{tikzpicture}}}

\newcommand{\timesa}
{\begin{tikzpicture}[baseline={(0,-0.08)}]
\node[] at (0.25, 0.15){\scriptsize $\times$};
\node[] at (0.25,-0.15) {\footnotesize $\alpha$};
\end{tikzpicture}}

\newcommand{\stimesa}
{\scalebox{0.7}{\begin{tikzpicture}[baseline={(0,-0.08)}]
\node[] at (0.25, 0.15){\scriptsize $\times$};
\node[] at (0.25,-0.05) {\footnotesize $\alpha$};
\end{tikzpicture}}}

\begin{document}

\section{Introduction}
Connecting the forward time evolution with ancestral processes backward in time 
is a central theme in mathematical population genetics. It is typically approached via various notions of duality; see \cite[Chaps.~3.4.4, 4.3, and 4.4]{Liggett10} and \cite{JK14} for the mathematical background, and \cite{Moehle99} for an overview of applications to population genetics. A particularly important notion is \emph{moment duality}, which links two real-valued continuous-time Markov processes $X=(X_t)_{t \geq 0}$ and $Z=(Z_t)_{t \geq 0}$ through their moments:
\begin{equation}\label{eq.moment_duality}
\Eb[X_t^z\mid X_0=x]=\Eb[x^{Z_t} \mid Z_0=z]
\end{equation}
for all $t\geq 0$ and for all possible initial states $x$ and $z$ of $X$ and $Z$, respectively.
In the context of population genetics, $X$ typically represents a frequency process in the forward model, while $Z$ is the line-counting process of an appropriate ancestral process obtained by tracing back the genealogy of a sample taken from the present population; in the simplest case, $Z_t$ is the number of ancestors of the sample at time $t$ before the present.

Major challenges arise when the reproduction of individuals depends on their types or on the composition of the population; the genealogies of (samples of) individuals then interact with the types as well. For a model with two types and selection (that is, one, which is termed fit, reproduces faster than the other, which is termed unfit), the Gordian knot was cut by Krone and Neuhauser \cite{KN97} with the introduction of the
ancestral selection graph (ASG), which contains all individuals that, as long as the types are unknown, may \emph{potentially} be ancestors of the given sample; $Z$ then is the line-counting process of this graph. 

In recent years, several variants of the ASG and the corresponding dualities have been developed and applied to specific questions, see, for example, \cite{BCH22,Boenkostetal2021,Cordero17,CorderoMoehle19,CV23,Grevenetal2016,PfaffelhuberPokalyuk2013}; in all these cases, the moment dual (or a closely-related factorial moment dual) retains the meaning of the line-counting process of some genealogical construction that consists of the potential ancestors of a sample. 

However, this interpretation is not available for all moment dualities arising in population-genetic processes. Consider, for example, the solution $X$ to the stochastic differential equation (SDE)
\begin{equation}\label{eq.SDE}
\dd X_t = \sqrt{X_t(1-X_t)(1-\kappa X_t)}\,  \dd B_t, \quad (B_t)_{t \geq 0} \; \text{standard Brownian motion},
\end{equation}
which arises as the diffusion limit of various two-type population models: a Wright--Fisher model with within-generation variance in offspring number \cite{Gillespie74}, a logistic branching process at carrying capacity \cite{DPK25}, the Wright--Fisher model with efficiency \cite{GMP20}, as well as its generalization, the two-size Wright--Fisher model \cite{ACD25}. As shown in \cite{GPP21} by purely analytical means, this SDE admits a moment dual whose transitions can be interpreted as coalescence events, where two lines merge into one, and pairwise branching events, where a pair of lines gives rise to a single new one. As pointed out by the authors, the conceptual connection between the dual and an ancestral process related to the forward model \eqref{eq.SDE} and, in particular, the interpretation of pairwise branching, remain unclear.

There are two standard approaches to establishing moment duality. The first relies on a generator argument and is purely analytic; this is the method used in \cite{GPP21}. The second approach is more involved, but it renders the origin of the dual process more transparent. It proceeds by first constructing duals for an approximating sequence of finite-population models (for instance, Moran or Wright--Fisher models), and then recovering the moment duality by passing to the diffusion limit. At the finite-population level, the duality is obtained by introducing a suitable graphical representation of the forward model, constructing an ancestral process starting from a sample from the population at present, and then reversing time in the graphical picture. The resulting finite-level duality is usually a \emph{factorial} moment duality, see \cite{Cannings74,Gladstien77a,Gladstien77b,Gladstien78}, and arises from an explicit coupling between the forward process and its backward (ancestral) counterpart, as established in \cite{Moehle99}. 

Although the second approach is more involved, it provides deeper insight into the duality relation and endows the dual process with a clear genealogical interpretation. In contrast to the generator method -- which typically works best once a good candidate for the dual is already known --, the use of graphical representations often allows one to identify the appropriate ancestral process even when the interaction mechanisms are complex. In this sense, it illuminates the “black box” underlying generator calculations. Prominent examples include the original construction of the ASG \cite{KN97}, its extension to frequency-dependent selection \cite{K99}, the treatment of generalized coalescence mechanisms \cite{CorderoMoehle19}, the extension of the ASG to Cannings models \cite{Boenkostetal2021}, and recent work in random environments \cite{CV23}. When (deleterious and beneficial) mutations are present, the ASG can be reduced to its informative components by pruning certain lines, thereby yielding a (factorial) moment dual once again; see, for instance, \cite{BCH18,BEH23,BLW16,BW18,CV23,Lenz+15}.

Our goal is to recover the moment duality using this second approach. Our graphical representation is based on a Moran model with interactive neutral reproduction, which has \eqref{eq.SDE} as its diffusion limit. In an interactive neutral reproduction event, an individual initiates an event by sampling (uniformly with replacement) two individuals from the population. The type of the first, called the \emph{checking} individual, determines whether the second reproduces: if the checking individual is fit, the second individual reproduces and its offspring replaces the individual that initiated the event; otherwise, nothing happens.

The standard (and, in a sense, natural) tool to handle the genealogy of this kind of model is the so-called ancestral influence graph (AIG) of Donnelly and Kurtz \cite{DK99a}, a far-reaching generalization of the ASG. In contrast to the ASG, it does not only trace back potential \emph{ancestors} of a sample, but all potential \emph{influencers}, that is, all individuals that may affect the type of the sample; in particular, it may include lines that will never be ancestral themselves, such as the checking lines in our model. In addition, the AIG contains information about the role(s) that every line plays in each of the events in the graph. Specifically, the AIG induced by our model with interactive neutral reproduction contains ternary branchings, which add the two individuals sampled at a potential reproduction event, along with the information which of the two is the checking line and the potential parent, respectively. It is already clear at this point that the line-counting process of the AIG cannot agree with the moment dual: ternary branching is present in the AIG, but absent in the moment dual. Also, to the best of our knowledge, no alternative particle picture matching the forward and backward evolution is known.

The complexity of ancestral graphs may often be reduced by carefully pruning (using the information contained in mutation events) and reordering (using exchangeability) certain lines; in some cases, a moment dual may be obtained in this way. Notably, this happens in the killed ASG \cite{BCH18,BEH23,CV23} (with its predecessor approach \cite{ShUch86}); see also the review \cite{BW18}. In other cases, where the interactions are more complex, the pruning does simplify the graph to some extent, but the roles of the potential influencers form a hierarchy that cannot be completely removed. Specifically, this was described in \cite{BCH22}, where selective reproduction with pairwise interaction was considered; there is no moment dual, but a more complicated tree-valued dual process. Our situation here is different: the results of \cite{GPP21} imply the existence of a moment dual, and we aim to recover and conceptually understand it from the point of view of the genealogical structure.

It will turn out as a key feature that, unlike in the examples mentioned so far, our pre-limit factorial moment duality does not hold pathwise; so there is no obvious conversion of the forward picture into a backward one. This is supported by the observation that an unfit sample may admit fit potential influencers, which contradicts the common intuition behind (pathwise) factorial moment duality. Instead, our method heavily relies on duality as a property in expectation, which allows more freedom in the construction. 
We make use of this by removing the information about the roles of parental and checking lines from the AIG and treat the two different events simultaneously. This unites different realizations, which, apart from the roles, share the same information. The resulting reduced version of the AIG is called a \emph{quasi-AIG}. Our approach is then based on the notion of \emph{configurations}, ordered and typed samples, which arise in the course of the backward type propagation of an initial sample according to the rules of our Moran model and its quasi AIG. These configurations will then be permuted and merged, even those coming from different realizations (``worlds'') with respect to the roles of parental and checking lines. We refer to this as the \emph{Frankenstein method} and envisage that it provides a promising framework for the understanding of dualities in models with interactions.

The model we are going to treat will, beyond interactive and non-interactive neutral reproduction, also contain frequency-dependent selection and mutation; in the diffusion limit, it will lead to the SDE \eqref{eq.full_SDE}. We have, so far, concentrated on interactive neutral reproduction because the other evolutionary forces do not require the specific Frankenstein treatment. Due to the additive structure of duality, the existence of a moment dual for our general version of the Moran model follows by combining \cite{BEH23} and \cite{GPP21}.

This paper is organized as follows. Section~\ref{sec:MoMo} defines the Moran model with interactive neutral reproduction, frequency-dependent selection, and mutation, and states the factorial moment duality as our main result. Section \ref{sec:natural_ancestry} investigates the natural ancestry; it begins with the graphical representation and the ancestral influence graph, establishes what we call a pre-duality, and concludes with the backward type propagation. Section~\ref{sec:our_approach} forms the core of our paper and is centered around the Frankenstein process as the connection between the natural ancestry and the factorial moment dual. The section starts with graphical heuristics motivating our approach, which are then rigorously transformed into a proof of the factorial moment duality. Finally, Section~\ref{sec:moment_duality} considers the diffusion scaling limits of the Moran model and the factorial moment dual to deduce the corresponding moment duality.

\section{Model and main results}\label{sec:MoMo}
Consider a panmictic haploid population consisting of $N\in \Nb$ individuals, which can be either of type $0$ or type $1$. We classify type $0$ as \emph{fit} or \emph{beneficial} and type $1$ as \emph{unfit} or \emph{deleterious}. The population evolves in continuous time and is governed by (non-interactive) neutral reproduction, interactive neutral reproduction, (frequency-dependent) selection, and mutation; that is, at any given time, every individual can independently perform any of the following events:
\begin{itemize}
    \item (Non-interactive) neutral reproduction: At rate $\r/2$, every individual, independently of its type, is replaced by the offspring of a randomly chosen individual from the population. 
    \item Interactive neutral reproduction: At rate $\kappa/2$, every individual, again independently of its type, initiates an interactive neutral reproduction by choosing two individuals from the population uniformly with replacement. The first individual of this pair is the \emph{checking individual}, which decides whether the reproduction takes place: if it is fit, the second individual of the pair reproduces, and its offspring replaces the initiator of the event. If the checking individual is unfit, nothing happens.
    \item Selection is understood as fittest type wins (FTW) selection, where an FTW selection event of order $m$ consists of a single individual that initiates the event and a group of $m\geq 1$ individuals chosen from the population with replacement. If at least one of the group members is fit, a fit offspring is created, which replaces the individual that initiated the event; otherwise, nothing happens. At rate $s_m$, every individual independently initiates an FTW event of order $m$. We assume $0 < \sum_{m=1}^\infty m \, s_m < \infty$, see \cite{GS18} and \cite{BEH23} for other models with FTW selection. The case $s_1>0, s_m=0$ for all $m>1$ is known as genic selection in biology.
    \item Mutation: at rate $u\nu_j$, every individual undergoes a mutation to type $j \in \{0,1\}$, where $\nu_0+\nu_1=1$. 
\end{itemize}
We call this model a \emph{Moran model (MoMo) with interactive neutral reproduction, FTW selection, and mutation}. We denote by $X^{(N)}_t$ the number of unfit individuals
alive at time $t \geq 0$. The process $X^{(N)}=(X^{(N)}_t)_{t \geq 0}$ is a birth-death process on $[N]_0\coloneqq [N] \cup \{0\}$ (where $[N] \coloneqq \{1,2,...,N\}$) and is characterized by its transition rates
\begin{equation*}
 n \in [N]_{0} \! \to \!
 \begin{cases}
     n+1 \quad & \text{at rate } \frac{r}{2}\, (N-n) \,\frac{n}{N}+\frac{\kappa}{2}\,(N-n)\,\frac{n}{N} \frac{N-n}{N}+(N-n) u \nu_1, \\[3pt]
    n-1 \quad & \text{at rate } \frac{r}{2} \,n \,\frac{N-n}{N}+\frac{\kappa}{2} \, n \,(\frac{N-n}{N})^2 +   n u \nu_0 +\sum_{m=1}^\infty s_m n \Big( 1-\big( \frac{n}{N}\big)^m\Big).
 \end{cases}   
\end{equation*}

For $\r=1$ and $\kappa=0$ (that is, without interactive neutral reproduction), the model reduces to the one tackled in \cite{BEH23}. Throughout the manuscript, we will ignore potential reproduction events that are \emph{silent}, that is, the potential parent agrees with the potential offspring individual; this will be made explicit
in more detail later. Furthermore, we will sometimes translate sampling \emph{with} replacement, as assumed in the FTW selection events above, into sampling \emph{without} replacement. Recall that, with rate $s_m$ ($m \in [N]$), we draw a random tuple, that is, a random element of $[N]^m$. The equivalent version without replacement reads: at rate $\tilde s_j, j \in [N]$, we draw a random set $\Bs \subseteq [N]$ with $\lvert \Bs \rvert = j$, where 
\begin{equation}\label{eq.stilde}
\tilde s_j = \sum_{m=j}^{\infty} s_m p^0_{mj},
\end{equation}
where $p_{mj}^n$ is the probability that, when sampling $m$ lines with replacement from the set $[N]$, exactly $j$ distinct lines are not in the set $[n]$.
While we do not need the explicit form of $p_{m j}^n$, we provide it here for completeness (and derive it in the appendix) as 
\begin{equation}\label{eq.pnmj}
p_{m j}^n
= \frac{C_{m j}^n}{N^m},
\qquad
C_{m j}^n
= (N-n)^{\underline{j}}
\sum_{\ell=j}^m \binom{m}{\ell}
\genfrac\{\}{0pt}{}{\ell}{j}\, n^{m-\ell},
\end{equation}
where $x^{\underline{j}}=x(x-1)\cdots(x-j+1)$ is the falling factorial, and $\genfrac\{\}{0pt}{}{\ell}{j}$ denotes the Stirling numbers of the second kind. 
The special case required for \eqref{eq.stilde} is $p^0_{mj}=\frac{N^{\underline{j}}}{N^m}\genfrac\{\}{0pt}{}{\ell}{j}$; the case $p^{n}_{mj}$ for $n>0$ will appear in the next definition.

As is often the case in the Moran setting, the type composition process can be related to a Markov chain on $[N]_{0}^\Delta\coloneqq [N]_0\cup\{\Delta\}$ via duality. Before stating this precisely, 
we introduce the corresponding Markov chain. Our configuration approach will directly lead to this definition and its interpretation as an ancestral process. In contrast, finding the dual via purely algebraic arguments would be difficult and the meaning less transparent.
\begin{definition}\label{def.factorial_dual}
Let $Z^{(N)}=(Z^{(N)}_t)_{t \geq 0}$ be the continuous-time Markov chain on $[N]_{0}^\Delta$ with transitions
\begin{equation*}
    n\in[N]_{0} \to \begin{cases}
        n + j & \text{at rate } n \sum_{m=j}^{\infty} s_m \,p_{mj}^n
        \text{ for }  j\geq 2,\\[5pt]
        n+1 & \text{at rate } \frac{\kappa}{N} \binom{n}{2} \frac{N-n}{N} + n \sum_{m=1}^{\infty}  s_m \, p^n_{m1}, \\[5pt]
        n-1 & \text{at rate } \frac{1}{N} \binom{n}{2} \big( r + \kappa \, \frac{N-(n-1)}{N}\big)  + u\nu_1 n ,\\[5pt]
        \Delta & \text{at rate } u\nu_0 n,
    \end{cases}
\end{equation*}
with $p_{mj}^n$ of \eqref{eq.pnmj}.
\end{definition}

The following theorem is the main result of this paper, both for the factorial moment duality relation it establishes and, even more importantly, for the novel methodology developed for its proof.
As soon as $Z^{(N)}$ is known, the factorial moment duality could be proved via generators in a relatively straightforward way. However, finding the dual is usually more involved than proving it. We address this by using the forward model to systematically derive the dual and establish its genealogical understanding from scratch.

\begin{theorem}[Factorial moment duality]\label{thm.factorial_duality}
Let $X^{(N)}$ be the counting process in the MoMo of size $N$ with interactive neutral reproduction, selection, and mutation, and let $Z^{(N)}$ be as in Def. \ref{def.factorial_dual}.
Then $X^{(N)}$ and $Z^{(N)}$ are dual, that is, for all $t \geq 0$, $n \in [N]_{0}^\Delta$, and $i \in [N]_0$, we have
\begin{equation*}
    \Eb\left[ \frac{ \big( X^{(N)}_t \big)^{\! \vphantom{(N)} \underline{n}}}{N^{\underline{n}}}\, \Big\vert \, X_0^{(N)}=i \right] = \Eb\left [ \frac{i^{\underline{Z^{(N)}_t}}}{N^{\underline{Z^{(N)}_t}}}\, \Big\vert\,  Z_0^{(N)}=n  \right],
\end{equation*}
where $i^{\underline{\Delta}}/N^{\underline{\Delta}}\coloneqq 0$ for all $i$.
We refer to this relation as a \emph{factorial moment duality}.
\end{theorem}

The factorial moment duality states that the probability of obtaining only unfit individuals when sampling $n$ individuals from the population at time $t$ can be computed in two different ways: either in the obvious way by running the \emph{forward} process $X^{(N)}$, starting from the initial frequency $i/N$ and sampling $n$ individuals (without replacement) according to $X^{(N)}_t$ \emph{(left-hand side)}; or by running the \emph{backward process} $Z^{(N)}$ starting from sample size $n$, and sampling $Z^{(N)}_t$ individuals (again without replacement) from the initial population \emph{(right-hand side)}. One is therefore tempted to interpret Theorem~\ref{thm.factorial_duality} as saying that the probability of obtaining an entirely unfit sample at time $t$ when sampling $n$ individuals is equal to the probability of obtaining an entirely unfit sample from the initial population when sampling $Z^{(N)}_t$ potential ancestors. 
This interpretation is correct in the limiting case $\kappa=0$. Indeed, in \cite{BEH23} the authors proved that $Z^{(N)}_t$ is the number of potential ancestors at time~$t$ before the present, and that the sample is unfit if and only if all potential ancestors are unfit. However, an analogous connection is far from obvious when $\kappa \in (0,1)$, once again due to the missing link between the forward and backward perspectives.

In Section~\ref{subsec.difflimit_X}, we will show that the solution 
$X=(X_t)_{t \geq 0}$ of the SDE
\begin{equation}\label{eq.full_SDE}
        \dd X_t = \Big(-\sum_{m=1}^{\infty}\big(\sigma_m X_t(1-X_t^m)\big) -\theta \, \nu_0 X_t + \theta \, \nu_1 (1-X_t)\Big) \dd t + \sqrt{X_t(1-X_t)\big(\r+\kappa(1- X_t)\big)} \dd B_t,
\end{equation}
again with $B$ standard Brownian motion, arises as the diffusion limit of the MoMo. In Section~\ref{subsec.difflimit_R}, we will establish the process $Z=(Z_t)_{t \geq 0}$ on $\Nb_{0,\Delta}\coloneqq \Nb \cup \{0, \Delta\}$, characterized by its transitions\begin{equation}\label{eq.diff_limit_Z}
        n \in \Nb  \to \begin{cases}
            n+m &\text{at rate } \sigma_m n \text{ for }m\geq 2, \\
            n+1 &\text{at rate } \kappa \frac{n}{2}+\sigma_1n, \\
            n-1 &\text{at rate } \frac{n}{2}(r+\kappa) + \theta \nu_1 n,  \\
            \Delta &\text{at rate }  \theta \nu_0 n,
        \end{cases}
    \end{equation}
as the diffusion limit of the process $Z^{(N)}$.
For $\r=1-\kappa$, $u=0$, and $s_m=0$ for all $m>0$, the diffusion limit \eqref{eq.full_SDE} of the forward model reduces to the SDE ~\eqref{eq.SDE} and the diffusion limit of $Z$ to the birth-death process with pairwise branching and coalescence events also mentioned in the introduction. As a consequence of these limits and the factorial moment duality from Theorem \ref{thm.factorial_duality}, we obtain the subsequent corollary.

\begin{corollary}[Moment duality]\label{corollary_moment_duality}
    The processes $X$ and $Z$ are moment dual, that is, for all $t \geq 0$, $x \in [0,1]$, and $n\in \Nb$, we have
    \[
    \Eb[X_t^n \mid X_0=x]=\Eb[x^{Z_t}\mid Z_0=n].
    \]
We refer to this relation as \emph{moment duality}.
\end{corollary}

The corollary is closely related to \cite[Theorem 2]{GPP21}, and for $\r=1-\kappa$, $\theta=0$, and $\sigma_m=0$ for all $m\geq 2$, it reduces to \cite[Lemma 1]{GMP20}. 

\section{Natural ancestry}\label{sec:natural_ancestry}
This section investigates the natural ancestry of the MoMo. Starting from the graphical representation of the model, we define its ancestral influence graph. However, the line-counting process of this graph differs from the factorial moment dual. As an alternative approach we introduce the concept of compatible configurations and establish a pre-duality. A detailed analysis of the backward type propagation motivates the construction of the configuration process.

\subsection{Graphical representation}\label{sec:graphical}

\begin{figure}[htbp]
    \centering
    \scalebox{0.7}{
    \begin{tikzpicture}
		\draw[dashed] (0,-0.5) --(0,5.5);
		\node [anchor = west] at (0,-0.5) {$0$}; 

        \draw[->, semithick] (5.25,-0.8) -- (9.25,-0.8); 
        \node at (7.25, -1){\large $t$};

        \draw[semithick] (0,5) -- (14.5,5); 
        \draw[semithick] (0,4) -- (14.5,4); 
        \draw[semithick] (0,3) -- (14.5,3);
        \draw[semithick] (0,2) -- (14.5,2); 
        \draw[semithick] (0,1) -- (14.5,1); 
        \draw[semithick] (0,0) -- (14.5,0);  
        
        \node[left] at (-0.2,5) {\large $1$}; 
        \node[left] at (-0.2,3.5) {\scalebox{0.8}{$\bullet$}}; 
		\node[left] at (-0.2, 2.5) {\scalebox{0.8}{$\bullet$}};   
		\node[left] at (-0.2,1.5) {\scalebox{0.8}{$\bullet$}};
		\node[left] at (-0.2,0) {\large $N$};   
		 
       
		\draw[-{triangle 45[scale=50]},semithick] (1.5,3) -- (1.5,1);
		
		\draw[-{open triangle 45[scale=5]},semithick] (2.9,0) -- (2.9,5);

        
        \def\x{5} 

        \coordinate (a) at (\x,4);
        \coordinate (b) at (\x,3);
        \coordinate (c) at (\x,0);

        \coordinate (target) at (\x,1);

        \draw[-{open triangle 45[scale=5]},semithick] (a) .. controls (\x+1, 2.5) .. (target);
        \draw[-{open triangle 45[scale=5]},semithick] (b) -- (target);
        \draw[-{open triangle 45[scale=5]},semithick] (c) -- (target);

        \draw (a) node[shape=rectangle, draw, fill=white!100, inner sep=2pt, yshift=0.3ex] {\normalsize $1$};
        \draw (b) node[shape=rectangle, draw, fill=white!100, inner sep=2pt, yshift=0.3ex] {\normalsize $3$};
        \draw (c) node[shape=rectangle, draw, fill=white!100, inner sep=2pt, yshift=0.3ex] {\normalsize $2$};


        \draw[-{triangle 45[scale=5]},semithick] (7.5,4) -- (7.5,3); 
        \draw[-{Diamond[open, scale = 1.3]}, semithick, dashed] (7.5,4) -- (7.5,5); 

        \draw[-{triangle 45[scale=5]},semithick] (8.5,0) -- (8.5,1);

		\draw[-{triangle 45[scale=5]}, semithick] (10.4,2) -- (10.4,4); 
		\draw[-{Diamond[open, scale = 1.3]}, semithick, dashed] (10.4,2) -- (10.4,1); 

        
        \def\x{13} 

        \coordinate (a) at (\x,5);
        \coordinate (b) at (\x,2);

        \coordinate (target) at (\x,4);

        \draw[-{open triangle 45[scale=5]},semithick] (a) -- (target);
        \draw[-{open triangle 45[scale=5]},semithick] (b) -- (target);

        \draw (a) node[shape=rectangle, draw, fill=white!100, inner sep=2pt, yshift=0.3ex] {\normalsize $2$};
        \draw (b) node[shape=rectangle, draw, fill=white!100, inner sep=2pt, yshift=0.3ex] {\normalsize $1$};
       

        \draw (6,3) [] circle (1mm)[fill=white!100];
        \draw (12,2) [] circle (1mm)[fill=white!100];

        \node[] at (2,4) {{$\times$}};
        \node[] at (8,0) {{$\times$}};
        
    \end{tikzpicture}}
    \caption{An untyped realization of the graphical representation of the Moran model with interactive neutral reproduction, fittest-type-wins selection and mutation, $t$ denotes the forward time increment.}
    \label{fig.MoMo}
\end{figure}
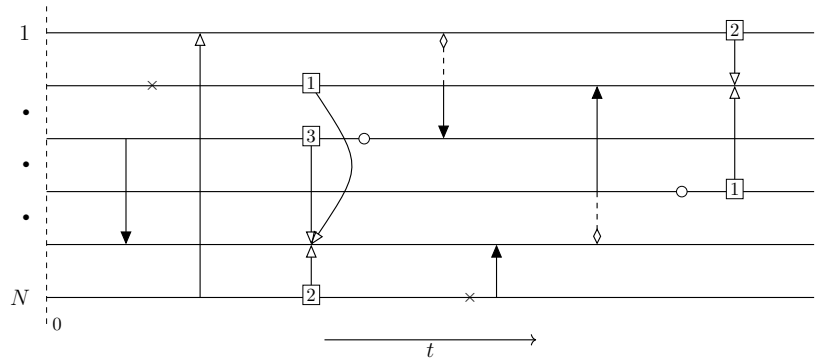

The Moran model admits a graphical representation (GR) as an interactive particle system, in the spirit of Harris \cite{harr78}. The $N$ individuals are represented by horizontal lines, labeled $1$ to $N$, along which time runs from left to right; see Figure~\ref{fig.MoMo}. The GR decouples the randomness inherent in the initial type configuration from the randomness generated by reproduction and mutation events. Because certain reproduction rates depend on the types of the individuals involved, each such event is interpreted as a superposition of two components: an attempted reproduction and, subsequently, a type-dependent decision determining whether the attempt is successful. In the first stage, we work with an untyped GR that contains all reproduction attempts and mutation marks, ignoring any information about the initial types of the lines; see Figure~\ref{fig.MoMo}. We incorporate types into the backward picture in Section \ref{subsec:tracing_back}.

Reproduction attempts are represented by arrows of different styles. We distinguish \emph{reproduction arrows} and \emph{checking arrows}. When a reproduction event is successful, we say that the reproduction arrow is \emph{used.} The tip of a reproduction arrow indicates the individual that will be replaced by the offspring of the individual at the tail (the offspring inherits the type of its parent) provided the attempt is successful. Among the reproduction arrows, we use solid-headed ($\blacktriangle$) arrows to represent neutral events and hollow-headed ($\vartriangle$) arrows to represent selective events. Checking arrows are depicted as dashed arrows with a diamond head ($\varDiamond$). Their tip indicates the individual whose type determines whether the tail individual produces an offspring in an interactive neutral event. 

Non-interactive neutral reproduction attempts are, by definition, always successful, regardless of the types of the involved lines; they occur at rate $\r/(2N)$ per ordered pair of lines.

Interactive neutral reproduction events are depicted by a neutral reproduction arrow and a checking arrow sharing the same tail; such pairs of arrows occur at rate $\kappa/(2N^2)$ per ordered triplet of lines. The reproduction attempt is successful if the checking arrow targets a fit individual.

FTW events of order $j$ (which here means without replacement) are depicted by $j$ selective arrows sharing a common tip; the individual at the tip is said to initiate the event. If all tail individuals are unfit, the reproduction attempt fails; otherwise, one of the fit tail individuals (chosen uniformly at random) succeeds. Each individual initiates FTW events of order $j$ at rate $\tilde s_j$.

Mutations are marked by symbols on the lines: beneficial mutations ($\circ$) occur at rate $u\nu_0$ per line, and deleterious mutations ($\times$) at rate $u\nu_1$.

Given an initial configuration of types and the untyped GR, types are propagated forward in time, thereby resolving each reproduction attempt according to the rules of the model, see Figure~\ref{fig.MoMo_typed}.
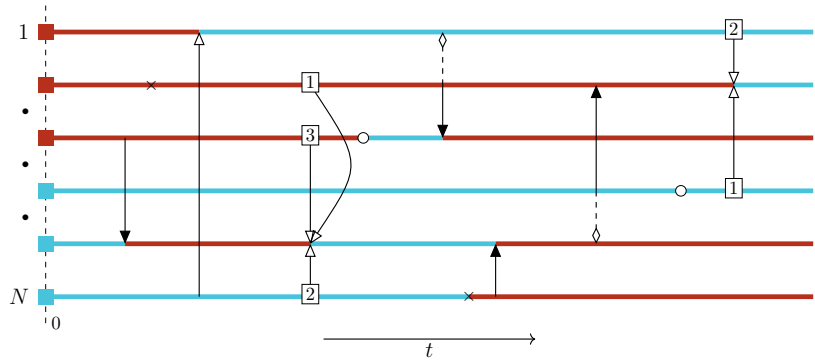
\begin{figure}[htbp]
    \centering
    \scalebox{0.7}{
    \begin{tikzpicture}

		\draw[dashed] (0,-0.5) --(0,5.5); 

		\node [anchor = west] at (0,-0.5) {$0$}; 
        
        \draw[->, semithick] (5.25,-0.8) -- (9.25,-0.8);  
        \node at (7.25, -1){\large $t$};
        
        \node[left] at (-0.2,5) {\large $1$};
		\node[left] at (-0.2,0) {\large $N$};   
		\node[left] at (-0.2,3.5) {\scalebox{0.8}{$\bullet$}};  
		\node[left] at (-0.2,1.5) {\scalebox{0.8}{$\bullet$}}; 
		\node[left] at (-0.2, 2.5) {\scalebox{0.8}{$\bullet$}};


        \fill [BrickRed] (-.15,4.85) rectangle (0.15,5.15);
        \fill [BrickRed] (-.15,3.85) rectangle (0.15,4.15);
        \fill [BrickRed] (-.15,2.85) rectangle (0.15,3.15);
        \fill [SkyBlue] (-.15,1.85) rectangle (0.15,2.15);
        \fill [SkyBlue] (-.15,0.85) rectangle (0.15,1.15);
		\fill [SkyBlue] (-.15,-.15) rectangle (0.15,0.15);

		\draw[BrickRed,line width=0.9mm] (0,5)--(2.9,5);
		\draw[SkyBlue,line width=0.9mm] (2.9,5)--(14.5,5);

        \draw[BrickRed,line width=0.9mm] (0,4)--(13,4);
        \draw[SkyBlue, line width=0.9mm] (13,4)--(14.5,4);

        \draw[BrickRed,line width=0.9mm] (0,3)--(6,3);
		\draw[SkyBlue,line width=0.9mm] (6,3)--(7.5,3);
        \draw[BrickRed,line width=0.9mm] (7.5,3)--(14.5,3);

        \draw[SkyBlue,line width=0.9mm] (0,2)--(14.5,2);

        \draw[SkyBlue,line width=0.9mm] (0,1)--(1.5,1);
		\draw[BrickRed,line width=0.9mm] (1.5,1)--(5,1);
        \draw[SkyBlue,line width=0.9mm] (5,1) --(8.5,1);
        \draw[BrickRed,line width=0.9mm] (8.5,1)--(14.5,1);

		\draw[SkyBlue,line width=0.9mm] (0,0)--(8,0);
		\draw[BrickRed,line width=0.9mm] (8,0)--(14.5,0);

        
		\draw[-{triangle 45[scale=50]},semithick] (1.5,3) -- (1.5,1);
		
		\draw[-{open triangle 45[scale=5]},semithick] (2.9,0) -- (2.9,5);

        
        \def\x{5} 

        \coordinate (a) at (\x,4);
        \coordinate (b) at (\x,3);
        \coordinate (c) at (\x,0);

        \coordinate (target) at (\x,1);

        \draw[-{open triangle 45[scale=5]},semithick] (a) .. controls (\x+1, 2.5) .. (target);
        \draw[-{open triangle 45[scale=5]},semithick] (b) -- (target);
        \draw[-{open triangle 45[scale=5]},semithick] (c) -- (target);

        \draw (a) node[shape=rectangle, draw, fill=white!100, inner sep=2pt, yshift=0.3ex] {\normalsize $1$};
        \draw (b) node[shape=rectangle, draw, fill=white!100, inner sep=2pt, yshift=0.3ex] {\normalsize $3$};
        \draw (c) node[shape=rectangle, draw, fill=white!100, inner sep=2pt, yshift=0.3ex] {\normalsize $2$};


        \draw[-{triangle 45[scale=5]},semithick] (7.5,4) -- (7.5,3); 
        \draw[-{Diamond[open, scale = 1.3]}, semithick, dashed] (7.5,4) -- (7.5,5); 

        \draw[-{triangle 45[scale=5]},semithick] (8.5,0) -- (8.5,1);

		\draw[-{triangle 45[scale=5]}, semithick] (10.4,2) -- (10.4,4); 
		\draw[-{Diamond[open, scale = 1.3]}, semithick, dashed] (10.4,2) -- (10.4,1); 

        
        \def\x{13} 

        \coordinate (a) at (\x,5);
        \coordinate (b) at (\x,2);

        \coordinate (target) at (\x,4);

        \draw[-{open triangle 45[scale=5]},semithick] (a) -- (target);
        \draw[-{open triangle 45[scale=5]},semithick] (b) -- (target);

        \draw (a) node[shape=rectangle, draw, fill=white!100, inner sep=2pt, yshift=0.3ex] {\normalsize $2$};
        \draw (b) node[shape=rectangle, draw, fill=white!100, inner sep=2pt, yshift=0.3ex] {\normalsize $1$};
       

        \draw (6,3)[] circle (1mm)  [fill=white!100];
        \draw (12,2)[] circle (1mm)  [fill=white!100];

        \node[] at (2,4) {{$\times$}};
        \node[] at (8,0) {{$\times$}};		
    \end{tikzpicture}}
    \caption{The typed version of the GR in Figure~\ref{fig.MoMo}, with the unfit type in \textcolor{BrickRed}{red} and the fit in \textcolor{SkyBlue}{blue}.}
    \label{fig.MoMo_typed}
\end{figure}

The graphical representation can be formalized as follows.

\begin{definition}\label{def.P}
Consider the set of graphical elements
\begin{equation*}
    M \coloneqq \Big\{ \narrowab, \sarrowab, \carrowcba, \circa, \timesa: \alpha, \beta, \gamma \in [N], \beta \neq \alpha , \{\alpha\} \neq \Bs  \subseteq[N]\Big\}.
\end{equation*}
Here, $\narrowab$ denotes a neutral arrow from line $\beta$ to line $\alpha\neq \beta$, and $\sarrowab$ represents a set of selective arrows with shared tip at line $\alpha$ and tails at lines $\beta \in \mathcal{B}$. The conditions $\beta \neq \alpha$ and $\Bs \neq \{\alpha\} $, respectively, correspond to our convention of neglecting silent reproduction events.
 The symbol $\carrowcba$ stands for an interactive neutral event with the tip of the checking arrow at line $\gamma$, the tail of the neutral arrow at line $\beta$, and its tip at line $\alpha\neq \beta$. Note that, for all events, line $\alpha$ is initiating, as indicated by the arrows pointing to the left. Finally, $\circa$ (resp.\ $\timesa$) indicates a beneficial (resp.\ deleterious) mutation on line $\alpha$. 

The graphical representation of the Moran model is encoded by the Poisson point process $\Ps$ on $\Rb \times M$ with intensity measure $\dd t\times \gamma_M(\dd m)$, where
\begin{equation}\label{eq.Poisson_measure}
\gamma_M=\frac{\r}{2 N}\sum_{\alpha,\beta}\delta_{\snarrowab}^{}+\frac{\kappa}{2N^{2}}\sum_{\alpha,\beta, \gamma}\delta_{\scarrowcba}^{}+\sum_{\alpha,\mathcal{B}}\frac{\tilde s_{\vert \mathcal{B}\vert}}{ {N \choose \lvert \Bs \rvert}}\delta_{\ssarrowab}+ u\sum_{\alpha}(\nu_0\delta_{\scirca}+\nu_1\delta_{\stimesa}),
\end{equation}
the summation indices are as defined above,
and $\delta_m$ denotes the Dirac mass at $m$. For every $m \in M$, we denote by $\Pi(m)=\Pi m$ the Poisson point process on $\mathbb{R}$ obtained by projecting the restriction of $\Ps$ to $\mathbb{R} \times \{m\}$ onto its first coordinate. The process $\Pi m$ records the events of type $m$. We further denote by $\mathcal{T}$ the Poisson point process obtained by projecting $\Ps$ onto its first component.
\end{definition}

\subsection{Ancestral Influence Graph}\label{sec:AIG}
As mentioned in the introduction, the ASG, which uses the untyped GR to trace back all potential ancestors of a sample taken at a fixed forward time $\tau \geq 0$, does not suffice in our setting. We rather need the more general AIG, which traces back all potential influencers, that is, all individuals that may influence the sample’s types when mutations are ignored, and also keeps track of all graphical elements that affect the lines in the graph.

At backward time $0$, the AIG starts with the set $g\subseteq [N]$ of individuals in the sample. At any backward time $r = \tau - t$, it consists of a subset of lines $\mathcal{G}_r^\tau \subseteq [N]$ and possibly a graphical element $m \in M$ (whenever $\tau - r \in \Pi m$). The set of lines is updated whenever one of the lines in $\mathcal{G}_{r-}^\tau$ is hit by the head of a reproductive arrow. Mutations do not affect the set of lines, but we record them as marks ($\times$ and $\circ$) decorating the lines of the AIG. We now describe how the set of lines is updated depending on the type of reproductive arrow hitting a line currently present in the AIG.

If a line in the AIG is hit by the tip of a reproductive arrow in a non-interactive neutral event, its type is fully determined by the line at the tail. If the latter already belongs to the AIG, we merge the two lines and continue tracing back only the line at the tail. If the line at the tail does not belong to the AIG, we relocate the line at the tip to the position of the tail; see Figure~\ref{fig.ASG_coal}.

\begin{figure}[htbp]
\begin{subfigure}[t]{0.485\textwidth}
\centering
\scalebox{0.8}{
 \begin{tikzpicture}
            \draw[semithick] (3.5,2) -- (6,2); 
            \draw[semithick] (0.2,1) -- (3.5,1);
            \draw[semithick, dotted] (3.5,1) -- (6,1); 
            \draw[-{triangle 45[scale=5]}, semithick] (3.5,1) -- (3.5,2);
            \node[anchor=east] at (3.5,1.3) {parent};
            \node[anchor=west] at (3.5,2.3) {descendant};  
    \end{tikzpicture}}
    \caption{Coalescence and relocation}
    \label{fig.ASG_coal}
\end{subfigure}
\hfill
\begin{subfigure}[t]{0.485\textwidth}
    \centering
        \scalebox{0.8}{
        \begin{tikzpicture}
            \draw[semithick] (0.2,2) -- (6,2); 
            \draw[semithick] (0.2,1) -- (3.5,1); 
            \draw[semithick, dotted] (3.5,1)-- (6,1); 
            \draw[-{open triangle 45[scale=5]}, semithick] (3.5,1) -- (3.5,2);
            \node[anchor=east] at (0.2,2) {continuing};
            \node[anchor=east] at (0.2,1) {incoming};
            \node[anchor=east] at (3.5,2.3) {potential parent};
            \node[anchor=east] at (3.5,1.3) {potential parent};
            \node[anchor=west] at (3.5,2.3) {descendant};
        \end{tikzpicture}}
\caption{Selective branching}
\label{fig.notation_genic}
\end{subfigure}
\\[15pt]
\centering
\begin{subfigure}[t]{0.485\textwidth}
    \centering
        \scalebox{0.8}{
        \begin{tikzpicture}
            \draw[semithick] (0.2,2) -- (6,2); 
            \draw[semithick] (0.2,1) -- (3.5,1); 
            \draw[semithick, dotted] (3.5,1)-- (6,1); 
            \draw[semithick] (0.2,0) -- (3.5,0); 
            \draw[semithick, dotted] (3.5,0)-- (6,0); 
            \draw[-{triangle 45[scale=5]}, semithick] (3.5,1) -- (3.5,2);
            \draw[-{Diamond[open, scale = 1.3]}, semithick, dashed] (3.5,1) -- (3.5,0);
            \node[anchor=east] at (0.2,2) {continuing};
            \node[anchor=east] at (0.2,1) {incoming};
            \node[anchor=east] at (0.2,0) {checking};
            \node[anchor=east] at (3.5,2.3) {potential parent};
            \node[anchor=east] at (3.5,1.3) {potential parent};
            \node[anchor=east] at (3.5,0.3) {potential influencer};
            \node[anchor=west] at (3.5,2.3) {descendant};
        \end{tikzpicture}}
    \caption{Interactive branching}
    \label{fig.notation_checking}
\end{subfigure}
\caption{Notation and roles of lines in the different events. Dotted lines indicate lines that may or may not be present in the AIG before the event. Backward time runs from right to left.}
\end{figure}
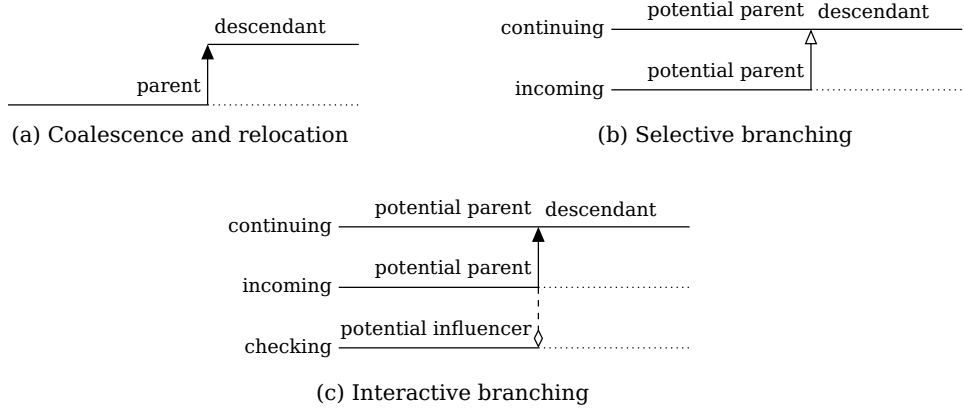

If a line in the AIG is hit by a set $\mathcal B$ of selective arrows, the lines at their tails may influence its type. Consequently, we add all tail lines of these selective arrows to the AIG (if they are not already present). We refer to the tail lines as \emph{incoming}, while the line at the tip is called \emph{continuing} at time $r$ and \emph{descendant} at time $r-$; see Figure~\ref{fig.notation_genic}. Once types are assigned, the parent of the descendant line is the continuing line if all incoming lines are unfit; otherwise, it is chosen uniformly at random from the fit lines.

If a line in the AIG is hit by the tip of the reproductive arrow in an interactive neutral event, the lines at the tail of the reproduction arrow and at the tip of the checking arrow may influence the type of the line that is hit and must be added to the AIG if not already present. Thus, in addition to the continuing and incoming lines, there is now also a checking line; see Figure~\ref{fig.notation_checking}. This interactive branching event may result in ternary branching (if neither the incoming nor the checking line was previously in the AIG, see Figure~\ref{ternary}), pairwise branching (if exactly one of them was already present, see Figures~\ref{pairwise1} and~\ref{pairwise2}), binary branching (if the checking line coincides with either the incoming or the continuing line and the incoming line was not in the AIG, see Figure~\ref{binary1} and~\ref{binary2}), or a collision (if both were already in the AIG, see Figure~\ref{collision}). Recall that if the checking line is unfit (is fit), then the continuing (the incoming) line is the true parent. 
See Figure~\ref{fig.MoMo_backwards} for an illustration of how to construct the AIG from the GR of the Moran model.

\begin{figure}[htbp]
\newlength{\tikzheight} 
\setlength{\tikzheight}{2.5cm}  
\begin{subfigure}[t]{0.15\textwidth}  
    \centering
    \begin{minipage}[t][\tikzheight][t]{\linewidth}
        \centering
        \begin{tikzpicture}[baseline=(current bounding box.north)]  
            \draw[semithick] (0.4,2) -- (2.25,2); 
            \draw[semithick] (0.4,1) -- (1.25,1); 
            \draw[semithick] (0.4,0) -- (1.25,0); 
            \draw[-{triangle 45[scale=5]}, semithick] (1.25,1) -- (1.25,2); 
            \draw[-{Diamond[open, scale =1.3]}, semithick, dashed] (1.25,1)-- (1.25,0); 
        \end{tikzpicture}
    \end{minipage}
\caption{}
\label{ternary}
\end{subfigure}
\begin{subfigure}[t]{0.15\textwidth}  
    \centering
    \begin{minipage}[t][\tikzheight][t]{\linewidth}
        \centering
        \begin{tikzpicture}[baseline=(current bounding box.north)]  

            \draw[semithick] (0.4,2) -- (2.25,2); 
            \draw[semithick] (0.4,1) -- (2.25,1); 
            \draw[semithick] (0.4,0) -- (1.25,0); 
            
            \draw[-{triangle 45[scale=5]}, semithick] (1.25,1) -- (1.25,2); 
            \draw[-{Diamond[open, scale =1.3]}, semithick, dashed] (1.25,1)-- (1.25,0); 
        \end{tikzpicture}
    \end{minipage}
\caption{}
\label{pairwise1}
\end{subfigure}
\hfill
\begin{subfigure}[t]{0.15\textwidth}
    \centering
    \begin{minipage}[t][\tikzheight][t]{\linewidth}
        \centering
        \begin{tikzpicture}[baseline=(current bounding box.north)]
            
            \draw[semithick] (0.4,2) -- (2.25,2); 
            \draw[semithick] (0.4,1) -- (1.25,1); 
            \draw[semithick] (0.4,0) -- (2.25,0); 
            \draw[-{triangle 45[scale=5]}, semithick] (1.25,1) -- (1.25,2); 
            \draw[-{Diamond[open, scale = 1.3]}, semithick, dashed ] (1.25,1) -- (1.25,0); 
        \end{tikzpicture}
    \end{minipage}
\caption{}
\label{pairwise2}
\end{subfigure}
\hfill
\begin{subfigure}[t]{0.15\textwidth}
    \centering
    \begin{minipage}[t][\tikzheight][t]{\linewidth}
        \centering
        \begin{tikzpicture}[baseline=(current bounding box.north)]
            \draw[semithick] (0.4,2) -- (2.25,2); 
            \draw[semithick] (0.4,1) -- (1.25,1); 
            \draw[-{Triangle[scale=1.5]}, semithick] (1.25,1) -- (1.25,2); 
            \draw[-{Diamond[open, scale = 1.3]}, semithick, dashed,] (1.25,1)  to[out=45, in=315,looseness=30] (1.25,0.9); 
        \end{tikzpicture}
    \end{minipage}
\caption{}
\label{binary1}
\end{subfigure}
\hfill
\begin{subfigure}[t]{0.15\textwidth}
    \centering
    \begin{minipage}[t][\tikzheight][t]{\linewidth}
        \centering
        \begin{tikzpicture}[baseline=(current bounding box.north)]
        \draw[semithick] (0.4,2) -- (2.25,2); 
        \draw[semithick] (0.4,1) -- (1.25,1); 
        \draw[-{triangle 45[scale=5]}, semithick] (1.25,1) -- (1.25,2); 
        \draw[-{Diamond[open, scale = 1.3]}, semithick, dashed, shorten >= 2pt] (1.25,1) .. controls (1.7,1.2) and (1.7,1.8) .. (1.25,2); 
        \end{tikzpicture}
    \end{minipage}
\caption{}
\label{binary2}
\end{subfigure}
\hfill
\begin{subfigure}[t]{0.15\textwidth}  
    \centering
    \begin{minipage}[t][\tikzheight][t]{\linewidth}
        \centering
        \begin{tikzpicture}[baseline=(current bounding box.north)]  
            \draw[semithick] (0.4,2) -- (2.25,2); 
            \draw[semithick] (0.4,1) -- (2.25,1); 
            \draw[semithick] (0.4,0) -- (2.25,0); 
            \draw[-{triangle 45[scale=5]}, semithick] (1.25,1) -- (1.25,2); 
            \draw[-{Diamond[open, scale =1.3]}, semithick, dashed] (1.25,1)-- (1.25,0); 
        \end{tikzpicture}
    \end{minipage}
\caption{}
\label{collision}
\end{subfigure}
\caption{Different AIG cutouts that may arise from a ternary (\ref{ternary}), pairwise (\ref{pairwise1},~\ref{pairwise2}) or binary (\ref{binary1},~\ref{binary2}) branching event or a collision (\ref{collision}), induced by an interactive branching event.}
\label{fig.transitions} 
\end{figure}
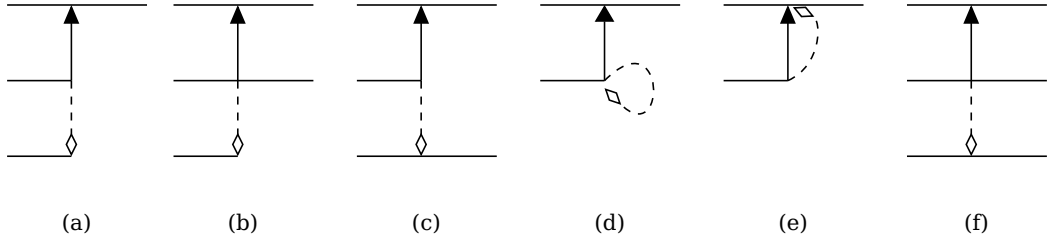

We now provide a more formal definition of the ancestral influence graph on the basis of the time-reversal of the Poisson point process $\Ps$, which we denote by its mirror image $\sP$. In this construction, the first component of the AIG is the set of lines of potential influencers, and its second component consists of the relevant elements of the process $\sP$, which determine the roles of the lines in the various events and are required for the forward propagation of types.
\begin{definition}[Ancestral influence graph]\label{def.AIG}
Consider the processes $\Ps$, $(\Pi m)_{m\in M}$, and $\Ts$ from Def.~\ref{def.P}. Fix $\tau\in \Rb$ and $\gs_0\subseteq [N]$. Let $\sP$ and $\sT$ be the time-reversals of $\Ps$ and $\Ts$, obtained by the time reversal $t \to \tau-t$. The process $\Gs = (\mathcal{G}_r^\tau)_{r\geq 0}$ starts at $\Gs_0=\gs_0$ and is recursively updated at the arrival times $\rho \in \sT \cap \Rb_+$ as follows.

\begin{enumerate}
\item If $\rho \in \Pi\! \narrowab $ for some $\alpha \in \mathcal{G}_{\rho-}^{\tau}$, $\beta \in [N]\setminus \{\alpha\}$, then \[\mathcal{G}_{\rho}^\tau=(\mathcal{G}_{\rho-}^\tau\setminus \{\alpha\}) \cup \{\beta\}.\] 
If $\beta \in \mathcal{G}_{\rho-}^{\tau}$, we speak of a coalescence event; otherwise, we speak of a relocation.
\item If $\rho \in \Pi\! \sarrowab $ for some $\alpha \in \mathcal{G}_{\rho-}^\tau$ and $\{\alpha\}\neq \mathcal{B}\subseteq [N]$, then 
\[\mathcal{G}_{\rho}^\tau=\mathcal{G}_{\rho-}^\tau \cup \Bs;\] 
we call this a selective branching event of order $\vert \mathcal{B}\setminus \mathcal{G}_{\rho-}^{\tau} \vert$.
\item If $\rho \in \Pi \carrowcba$ for some $\alpha \in \mathcal{G}_{\rho-}^\tau$ and $\beta \in [N]\setminus\{\alpha\}, \gamma \in [N]$, then \[\mathcal{G}_{\rho}^\tau=\mathcal{G}_{\rho-}^\tau \cup \{\beta,\gamma\}.\] We call this an interactive branching event. It results in ternary branching if $\{\beta,\gamma\}\cap \mathcal{G}_{\rho-}^\tau=\varnothing$, in pairwise branching if $\lvert \{\beta,\gamma\}\cap \mathcal{G}_{\rho-}^\tau\rvert =1$, in binary branching if $\beta \notin \mathcal{G}_{\rho-}^\tau$ and $\gamma \in \{\alpha,\beta\}$, and in a collision if $\{\alpha, \beta, \gamma \}\in \mathcal{G}_{\rho-}^\tau$.
\end{enumerate}

We refer to the set $\gs_0$ as the initial sample at (forward) time $\tau$, and to its elements as \emph{roots}. The lines of $\mathcal{G}_r^{\tau}$ present at time $r$ are called \emph{leaves}.
The \emph{ancestral influence graph (AIG)} with initial sample $\gs_0$ taken at time $\tau$ is then defined as
\[
\As^\tau=\As^\tau(\gs_0) \coloneqq \big(\Gs^\tau(\gs_0), \sP(\gs_0)\big),
\]
where $\sP(\gs_0)$ consists of the points $(r, m)$ of $\sP$ for which $m$ is either a mutation mark on a line in $\mathcal{G}_r^\tau(\gs_0)$ or encodes a reproduction event whose reproduction arrow has its tip in $\mathcal{G}_r^\tau(\gs_0)$. We write $\As^\tau_{[0,r]}$ for the restriction of $\As^\tau$ to $[\tau-r,\tau]$. 
\end{definition}

The true ancestry on the interval $[\tau-t, \tau]$ is determined by assigning types to the active lines at backward time $t$ and propagating them forward in time according to the rules of the Moran model. In doing so, the interactive neutral events and the selective events are resolved (in the sense that the true parent and the type of the descendant are determined), and the type of the sample is established. We will write $\As$ and $\Gs$ instead of $\As^\tau$ and $\Gs^\tau$, as their distributions do not depend on $\tau$.

\begin{figure}[htbp]
    \centering
    \scalebox{0.7}{
    \begin{tikzpicture}

		\draw[dashed] (0,-0.5) --(0,5.5);
		\draw[dashed] (14.5,-0.5) --(14.5,5.5); 

        \draw[->, semithick] (5.25,-0.8) -- (9.25,-0.8); 
        \node at (7.25, -1){\large forward time}; 

        \node [anchor = west] at (0,-0.5) {$0$}; 
		\node [anchor = west] at (14.5,-0.5) {\large $\tau$};

        \draw[->, semithick, mygreen] (9.25,5.7) -- (5.25,5.7);
        \node at (7.25, 5.9) [text=mygreen]{\large backward time}; 

        \node [anchor = west, mygreen] at (14.5,5.7) {\large $0$}; 
		\node [anchor = west, mygreen] at (0,5.7) {\large $\tau$};

		\node[left] at (-0.2,0) {\large $N$};   
		\node[left] at (-0.2,5) {\large $1$};  
		\node[left] at (-0.2,1.5) {\scalebox{0.8}{$\bullet$}}; 
		\node[left] at (-0.2, 2.5) {\scalebox{0.8}{$\bullet$}};   
		\node[left] at (-0.2,3.5) {\scalebox{0.8}{$\bullet$}};
        
        \draw[mygreen,opacity=1,line width=1.3mm] (0,5)--(13.065,5);
        \draw[mygreen,opacity=1,line width=1.3mm] (0,4)--(14.5,4);
        \draw[mygreen,opacity=1,line width=1.3mm] (0,2)--(13.065,2);
        \draw[mygreen,opacity=1,line width=1.3mm] (8.435,1)--(10.465,1);
        \draw[mygreen,opacity=1,line width=1.3mm] (0,0)--(8.565,0);

        \draw[mygreen,opacity=1,line width=1.3mm] (13,2)--(13,5);
        \draw[mygreen,opacity=1,line width=1.3mm] (10.4,1)--(10.4,4);
        \draw[mygreen,opacity=1,line width=1.3mm] (8.5,0)--(8.5,1);
        \draw[mygreen,opacity=1,line width=1.3mm] (2.9,0)--(2.9,5);

        \draw[semithick] (0,5) -- (14.5,5); 
        \draw[semithick] (0,4) -- (14.5,4); 
        \draw[semithick] (0,3) -- (14.5,3); 
        \draw[semithick] (0,2) -- (14.5,2); 
        \draw[semithick] (0,1) -- (14.5,1); 
        \draw[semithick] (0,0) -- (14.5,0);  
        
        \fill [mygreen, opacity=1] (14.35,3.85) rectangle (14.65,4.15);

        \fill [mygreen] (-.15,4.85) rectangle (.15,5.15);
        \fill [mygreen] (-.15,3.85) rectangle (.15,4.15);
        \fill [mygreen] (-.15,1.85) rectangle (.15,2.15);
        \fill [mygreen] (-.15,-.15) rectangle (.15,.15);

        
		\draw[-{triangle 45[scale=50]},semithick] (1.5,3) -- (1.5,1);
		
		\draw[-{open triangle 45[scale=5]},semithick] (2.9,0) -- (2.9,5);

        
        \def\x{5} 

        \coordinate (a) at (\x,4);
        \coordinate (b) at (\x,3);
        \coordinate (c) at (\x,0);

        \coordinate (target) at (\x,1);

        \draw[-{open triangle 45[scale=5]},semithick] (a) .. controls (\x+1, 2.5) .. (target);
        \draw[-{open triangle 45[scale=5]},semithick] (b) -- (target);
        \draw[-{open triangle 45[scale=5]},semithick] (c) -- (target);


        \draw[-{triangle 45[scale=5]},semithick] (7.5,4) -- (7.5,3); 
        \draw[-{Diamond[open, scale = 1.3]}, semithick, dashed] (7.5,4) -- (7.5,5); 

        \draw[-{triangle 45[scale=5]},semithick] (8.5,0) -- (8.5,1);
	
		\draw[-{triangle 45[scale=5]}, semithick] (10.4,2) -- (10.4,4); 
		\draw[-{Diamond[open, scale = 1.3]}, semithick, dashed] (10.4,2) -- (10.4,1); 

        
        \def\x{13} 

        \coordinate (a) at (\x,5);
        \coordinate (b) at (\x,2);

        \coordinate (target) at (\x,4);

        \draw[-{open triangle 45[scale=5]},semithick] (a) -- (target);
        \draw[-{open triangle 45[scale=5]},semithick] (b) -- (target);

        \draw (a) node[shape=rectangle, draw, fill=white!100, inner sep=2pt, yshift=0.3ex] {\normalsize $2$};
        \draw (b) node[shape=rectangle, draw, fill=white!100, inner sep=2pt, yshift=0.3ex] {\normalsize $1$};
       
        
        \draw (6,3)[] circle (1mm)  [fill=white!100];
        \draw (12,2)[] circle (1mm)  [fill=white!100];

        \node[] at (2,4) {\scalebox{1.5}{$\times$}};
        \node[] at (8,0) {\scalebox{1.5}{$\times$}};
		
    \end{tikzpicture}}
    \caption{The AIG induced by the untyped graphical representation from Figure~\ref{fig.MoMo} for a sample $\gs_0$ of size $1$ (the green square at the right). The green arrows and the mutation marks on green lines are elements of $\sP(\gs_0)$.}
    \label{fig.MoMo_backwards}
\end{figure}
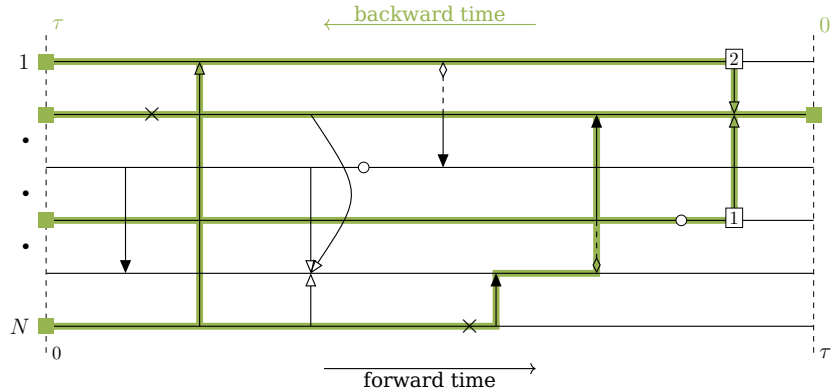
From the above construction and the dynamics of  $\sP$, it follows that the line-counting process of the AIG is a continuous-time Markov chain on $[N]$ with transitions 
    \begin{equation*}
        n \to  \begin{cases}
        n+j & \text{ at rate } n   \sum_{m=j}^{\infty}  s_m \, p_{mj}^n \text{ for } j\geq 3,\\[5pt]
          n+2 & \text{ at rate } \frac{\kappa}{N} \frac{n}{N}  \binom{N-n}{2} + n \sum_{m=2}^{\infty}  s_m \, p_{m2}^n,\\[5pt]
          n+1 & \text{ at rate } 2 \frac{\kappa}{N} \binom{n}{2} \frac{N-n}{N} + \frac{\kappa}{N} n  \frac{N-n}{N} + n 
          \sum_{m=1}^{\infty}  s_m \, p_{m1}^n,\\[5pt]
          n-1 & \text{ at rate } \frac{r}{N} \binom{n}{2}, 
        \end{cases}    
    \end{equation*}
with the probabilities $p_{mj}^n$ as in Def.~\ref{def.factorial_dual}.

In the absence of mutations and interactive neutral events, the line-counting process has the same transition rates and hence the same law as the process $Z^{(N)}$ in Theorem~\ref{thm.factorial_duality}. When mutations are included, $Z^{(N)}$ has additional transitions that result from a dynamical pruning of the lines upon mutation events and turn the ancestral graph into the \emph{killed ancestral graph} with line-counting process $Z^{(N)}$ \cite[Theorem 2.3]{BEH23}. The interactive neutral events, however, stand out: they introduce additional coalescence events into $Z^{(N)}$ that are not contained in the line-counting process, whereas the ternary and binary branchings they cause in the line-counting process are absent in $Z^{(N)}$. This indicates that the two concepts do not fit together without further tuning.

\subsection{Pre-duality}\label{subsec:pre-duality}
As discussed at the end of the previous section, the AIG does not explain the factorial duality. We therefore adopt a different approach, which will, informally, consist of tracing back sets of type configurations that are \emph{compatible}, in a sense to be made precise, with sampling only unfit individuals at time $\tau$, given the ancestral structure encoded by the AIG. To this end, we first introduce some notation. Recall from Figure~\ref{fig.MoMo_typed} that \emph{red} ($\mathrm{r}$) refers to unfit individuals, whereas \emph{blue} ($\mathrm{b}$) refers to fit ones. From this point onward, we write $R = \{\mathrm{r}\}$ and $B = \{\mathrm{b}\}$, and use $\ast = \{\mathrm{r},\mathrm{b}\} = R \cup B$.

\begin{definition}[Cylinders and configurations]
A cylinder $C$ is a set of the form
\[
C = C_I = \underset{i \in I}{\bigtimes} C_i
\]
for some $I \subseteq \Nb$, such that, for all $i \in I$, we have $C_i \in \{R, B, \ast, \varnothing\}$. We imply that $C_I=\varnothing$ if $C_i=\varnothing$ for some $i\in I$. 
If $C_i \in \{R, \ast\}$ for all $i \in I$, we speak of an \emph{$R$-cylinder}. We write $R_I$ for the cylinder $\bigtimes_{i\in I} R\in \varSigma_I^R$, where $\varSigma_I^R$ denotes the set of $R$-cylinders with index set $I$. Similarly, let $\varSigma_I$ be the set of all cylinders with index set $I$, and
\[
\varSigma = \bigcup_{I \subseteq [N]} \varSigma_I
\]
the collection of all cylinders with index set contained in $[N]$; $\varSigma^R$ is defined analogously. For $C \in \varSigma^R$ and $T \in \{R, B, \ast\}$, define $n_T(C) \coloneqq \bigl|\{\, i \in I : C_i = T \,\}\bigr|$, and note that $n_R(C) + n_B(C) + n_\ast(C) = |I|$.

Elements $c_I \in C_I \in \varSigma_I$ are called \emph{configurations}. For $\zeta\in\{\mathrm{r},\mathrm{b}\}$, we define $n_\zeta(c_I) = \lvert \{ i \in I : c_i = \zeta \}\rvert$, and note that $n_\mathrm{r}(c_I) + n_\mathrm{b}(c_I) = |I|$.

When the index set $I$ is clear from the context, we write $C$ and $c$ instead of $C_I$ and $c_I$, respectively. Moreover, in our examples we will usually write $R*BB*$ instead of $R \times * \times B \times B \times *$.
\end{definition}

We now introduce the aforementioned notion of compatibility. 
\begin{definition}[Compatibility]\label{def.compatB}
 Let $\as=\as(\gs_0) = (\mathcal{g}(\gs_0), \mirrorp(\gs_0))$ be a realization of an AIG on a finite time interval with root set $\gs_0$ and leaf set $\ell$. The symbol $\mathcal{g}$ is used for the realization of $\Gs(\gs_0)$, and $\mirrorp(\gs_0)$ for the corresponding realization of $\sP(\gs_0)$, see Def.~\ref{def.AIG}. 
 We refer to 
 \[ c=(c_i)_{i\in \ell}\in \{r,b\}^{\ell}\coloneqq \underset{i\in\ell}{\bigtimes}\{\mathrm{r},\mathrm{b}\} \quad \text {and} \quad \overline{C}=(\overline{C}_i)_{i\in {\gs_0}}\in \{\mathrm{r},\mathrm{b}\}^{\gs_0} \]
as a \emph{leaf-type configuration} and a \emph{root-type configuration}.
Now consider some leaf-type configuration $c\in \{\mathrm{r},\mathrm{b}\}^{\ell}$. By assigning the types $c$ to the leaves of $\as$ and propagating them along the lines of $\as$ according to the rules induced by $\mirrorp(\gs_0)$, one obtains the root-type configuration, denoted
$\overset{\rightarrow}{T}_{\as} (c) \in \{\mathrm{r},\mathrm{b}\}^{\gs_0}$.
We say that a cylinder $C \in \varSigma_{\ell}$ is \emph{compatible} with $\as$ and a set 
$S\in \varSigma_{\gs_0}$ of root-type configurations if, for any leaf-type configuration
$c \in C$, we have
$
\overset{\rightarrow}{T}_{\as}(c) \in S,
$
i.e. if $\overset{\rightarrow}{T}_{\as}(C) \subseteq S$.
We refer to 
\[\mathcal{S}(\as,S)\coloneqq({\overset{\rightarrow}{T}_{\as}})^{-1}(S)
\]
as the set of $(\as, S)$-compatible (leaf) configurations.
\end{definition}

The set of compatible configurations, with respect to different variants of the AIG, will be our main interest in this manuscript. The next result formalizes the connection between the set of compatible configurations and the process $X^{(N)}$, which encodes the evolution of the number of unfit individuals in the Moran model.

Let $\Gamma^{(i)}$ $\in \{\mathrm{r},\mathrm{b}\}^N$ be uniformly distributed over all configurations with exactly $i$ entries equal to $\mathrm{r}$ and $N-i$ equal to $\mathrm{b}$, and assume that $\Gamma^{(i)}$ is independent of $\sP^-$ and of the initial sampling configuration of the AIG.
For a set $B$ of configurations on $I \subset [N]$, we write $\Gamma^{(i)} \in B$ to mean that the restriction of $\Gamma^{(i)}$ to $I$ lies in $B$; equivalently, $\Gamma^{(i)} \in B^*$, where $B^*$ denotes the set of configurations on $[N]$ obtained from elements of $B$ by adding $*$'s on $[N]\setminus I$.
\begin{lemma}[Pre-duality]\label{pre-duality}
Let $\Gamma^{(i)}$ be as above. Then, for all $t \geq 0$, $n \in [N]_{0}^\Delta$, and $i \in [N]_0$, we have
\begin{equation}\label{eq:preduality}
    \Eb\left[ \frac{{(X^{(N)}_t)}^{\underline{n}}}{N^{\underline{n}}}\, \Big\vert \, X^{(N)}_0=i \right] = \Pb\big(\Gamma^{(i)}\in \Ss(\As_{[0,t]},R_{g_n})\big),   
\end{equation} 
where $g_n$ denotes the root set obtained by sampling without replacement $n$ individuals from the population at time $t$.
\end{lemma}
\begin{proof}
Note that the process $\Ps$ and its mirrored version $\sP$ provide a natural coupling of the processes $(X^{(N)}_s)_{s\in[0,t]}$ and $(\Ss(\As_{[0,r]},R_{g_n}))_{r\in[0,t]}$. In addition to the randomness from $\Ps$, the forward process $X^{(N)}$ is subject to the randomness coming from the initial type configuration, chosen uniformly among all configurations with exactly $i$ unfit individuals. Similarly, the randomness inherent in the backward process stems from $\sP$ and the random sample $g_n$. The identity follows because both sides represent the probability that a sample of size $n$, drawn without replacement at time $t$, consists entirely of unfit individuals: the left-hand side is obtained by conditioning on $X^{(N)}_t$, while the right-hand side is obtained by conditioning on the AIG at time $t$ associated with the unfit sample of size $n$ and using the definition of $\Ss(\As_{[0,t]},R_{g_n})$.
\end{proof}

The next result provides a formal expression for the right-hand side of \eqref{eq:preduality} and holds beyond the class of $R$-cylinders.
\begin{lemma}\label{prop_average1}
Let $\mathcal{A}_{[0,t]}$ be an AIG with root set $\gs_0$, let $\Gamma^{(i)}$ be as above, and let $S\in \varSigma_{\gs_0}$. Then, we have 
    \begin{equation*}
       \Pb\big(\Gamma^{(i)}\in \mathcal{S}(\mathcal{A}_{[0,t]},S)\big)=\, \Eb\left[ \sum_{c \in \Ss(\mathcal{A}_{[0,t]},S)}\frac{i^{\underline{n_r(c)}}}{N^{\underline{n_r(c)}}}  \frac{(N-i)^{\underline{n_b(c)}}}{N^{\underline{n_b(c)}}}\right].
    \end{equation*}
\end{lemma}
\begin{proof}
Using additivity of probability and the hypergeometric sampling scheme, we obtain
\begin{align*}
    \Pb\big(\Gamma^{(i)}\in \Ss(\mathcal{A}_{[0,t]}, S)\big)
    &= \Eb\left[\sum_{c \in \Ss(\mathcal{A}_{[0,t]}, S)} \Pb(\Gamma^{(i)}=c)\right] \\
    &= \Eb\left[\sum_{c \in \Ss(\mathcal{A}_{[0,t]}, S)}
       \frac{i^{\underline{n_r(c)}}}{N^{\underline{n_r(c)}}}
       \frac{(N-i)^{\underline{n_b(c)}}}{N^{\underline{n_b(c)}}}\right].
\end{align*}
\end{proof}
In practice, this formula does not yield meaningful insight into the dynamics of the right-hand side of \eqref{eq:preduality}, mainly because we lack sufficient knowledge of the set $\Ss(\mathcal{A}_{[0,t]}, S)$. 
As a next step, we will explore this set at the level of cylinders.

\subsection{Tracing back compatible cylinders; the configuration process}\label{subsec:tracing_back}
So far, we have worked with the probability of the set of compatible configurations $\Ss(\aigt(\gs_0), S)$ without explicitly addressing its dynamics or structure. Figures~\ref{fig.dotted}--\ref{fig.binary} illustrate the local effects of AIG transitions on $\Ss(\aigt(\gs_0),S)$. Starting from a single cylinder, after finitely many transitions we still obtain a finite disjoint union of cylinders. For this reason, we will identify the set $\Ss(\aigt(\gs_0), S)$ with a finite collection of pairwise disjoint cylinders.

To this end, we will describe the action of the various transitions in the AIG on a given cylinder. In what follows, we assume that the AIG consists of the set $I$ of lines before an event and performs a transition after which the set of lines is $\tilde I$ (that is, $I$ and $\tilde I$ are the realizations of $\Gs^{}_{\rho-}$ and $\Gs^{}_{\rho}$ in Def.~\ref{def.AIG} and the set of lines to the right and to the left of the event in our pictures). Our aim is now to trace back the preimage of any given cylinder $C_I \in \varSigma_I$, under the forward type propagation for any given transition.
To prepare for this, recall that the line $\alpha$ that initiates the event is the only one that may change its type in the course of such a forward event. However, backward in time, this is no longer true. To see this, let $D \subseteq \tilde I$ be the set of potential influencers of line $\alpha$ to the left of the event. $D$ is some union of $\{\alpha\}, \{\beta\}, \{\gamma\}$ and $\Bs$ according to Def.~\ref{def.AIG}; in particular, it may or may not include $\alpha$ itself. 
Note that $\tilde I = (I \setminus \{\alpha\}) \cup D$. The preimage $\longtilde{C}_{\tilde{I}}$ may then be determined via the following rules:
\begin{enumerate}
\item[(A)] For the lines in $\tilde I \setminus D$, the type is unchanged by the event, so $\longtilde C_\ell =C_\ell$ for $\ell \in \tilde I \setminus D$.
\item[(B1)] The types of the lines in $D$ to the left of the event must be compatible with $C_\alpha$ under the forward type propagation prescribed by the event. 
\item[(B2)] At the same time, the types of the lines in $D\cap (I\setminus \{\alpha\})$ to the left of the event must be compatible with those to the right of the event. 
\end{enumerate}
In order to avoid the case distinction between $\ell \in D \setminus (I\setminus \{\alpha\})$
(affected by rule (B1)) and $\ell \in D \cap (I\setminus \{\alpha\})$ 
 (affected by rules (B1) and (B2)), we introduce $C^\ast$ as the extension from $C_I$ to
$C_{\tilde I}$ by adding $\ast$'s; that is, $C_\ell^*=C_\ell$ if $\ell \in I$ and $*$ if $\ell \in \tilde I \setminus I$. In many of our pictures, the lines in $D \setminus \{\alpha\}$ will be dotted; they may or may not be contained in $I$. With respect to the lines in $D$, the cylinders that satisfy both (B1) and (B2) are thus the intersections of $C^*_D$ with the cylinders compatible with $C_\alpha$. Together with the coordinates for $\tilde I \setminus D$, they constitute the preimage of $C_I$.

\begin{figure}[htbp]
\centering
\begin{subfigure}{0.24\textwidth}\centering
\begin{tikzpicture}
  \draw[semithick] (0.4,1) -- (1.25,1);
  \draw[semithick] (1.25,1) -- (2.25,1);
  \draw[semithick] (1.25,0) -- (2.25,0);
  \draw[-{triangle 45[scale=1.1]},semithick] (1.25,1) -- (1.25,0);
  \node at (2.5,1) {$C_\beta$};
  \node at (2.45,0) {$T$};
  \node at (-0.2,1) {$T\cap C_\beta$};
\end{tikzpicture}
\caption{$C_\beta^\ast \neq \ast$, $\beta \in I$}
\label{fig.dottedb}
\end{subfigure}
\hfill
\begin{subfigure}{0.18\textwidth}
\centering
\begin{tikzpicture}
  \draw[semithick] (0.4,1) -- (1.25,1);
  \draw[semithick] (1.25,1) -- (2.25,1);
  \draw[semithick] (1.25,0) -- (2.25,0);
  \draw[-{triangle 45[scale=1.1]},semithick] (1.25,1) -- (1.25,0);
  \node at (2.4,1) {$\ast$};
  \node at (2.45,0) {$T$};
  \node at (0.2,1) {$T$};
\end{tikzpicture}
\caption{$C_\beta^{\ast} = \ast$, $\beta\in I$}
\label{fig.dottedc}
\end{subfigure}
\hfill
\begin{subfigure}{0.175\textwidth}\centering
\begin{tikzpicture}
  \draw[semithick] (0.4,1) -- (1.25,1);
  \draw[semithick] (1.25,0) -- (2.25,0);
  \draw[-{triangle 45[scale=1.1]},semithick] (1.25,1) -- (1.25,0);
  \node at (2.45,0) {$T$};
  \node at (0.2,1) {$T$};
\end{tikzpicture}
\caption{$C_\beta^\ast = \ast$, $\beta\notin I$}
\label{fig.dottedd}
\end{subfigure}
\hfill
\begin{subfigure}{0.24\textwidth}\centering
\begin{tikzpicture}
  \draw[semithick] (0.4,1) -- (1.25,1);
  \draw[semithick, dotted] (1.25,1) -- (2.25,1);
  \draw[semithick] (1.25,0) -- (2.25,0);
  \draw[-{triangle 45[scale=1.1]},semithick] (1.25,1) -- (1.25,0); 
  \node[align=left] at (2.45,1) {$C_\beta^\ast$};
  \node at (2.45,0) {$T$};
  \node[align=right] at (-0.2,1) {$T \cap C_\beta^\ast$};
\end{tikzpicture}
\caption{Arbitrary $\beta$ and $C_\beta^\ast$}
\label{fig.dotteda}
\end{subfigure}
\caption{Visualization of  Eq.~\eqref{eq.transition_coalB}, with $T\in \{R, B, \ast\}$. Panel~(d) summarizes (a)--(c).}
\label{fig.dotted}
\end{figure}
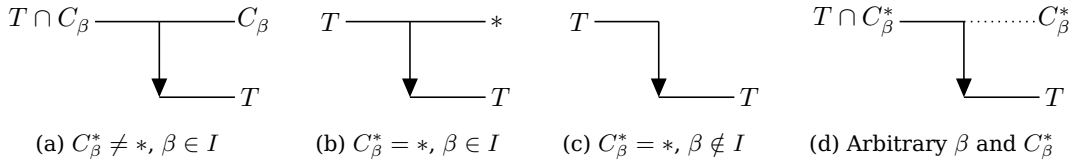

\textbf{Non-interactive neutral events} (see Fig. \ref{fig.dotted}): Assume that an event of type $\coal{\alpha}{\beta}$ occurs between lines $\alpha\in I$ and $\beta\in[N]\setminus \{\alpha\}$. Then $D=\{\beta\}$, $\tilde{I}=(I\,\setminus\{\alpha\})\cup\{\beta\}$, and $C_I \to \ocoal{\alpha}{\beta}(C_I)=\longtilde{C}_{\tilde{I}}$, where, for $\ell\in\tilde{I}$,
\begin{equation}\label{eq.transition_coalB}
   \longtilde{C}_\ell=\begin{cases}
        C_{\ell} &\text{ for } \ell\neq \beta, \\
        C_\alpha \cap C^{*}_{\beta} &\text{ for } \ell= \beta.
    \end{cases}
\end{equation}
The transition reflects the fact that, in the forward type propagation, the event means that line $\alpha$ receives the type from line $\beta$, and nothing happens to the other lines. So the first case in \eqref{eq.transition_coalB} reflects rule (A), while the second reflects (B1) and (B2). Note that the transition maps $C_I$ to the cemetery state $\varnothing$ if the intersection is empty. Note also that \eqref{eq.transition_coalB} covers both coalescence events (with $\beta \in I$) and relocation events (with $\beta \notin I$), see Figure~\ref{fig.dottedc} and~\ref{fig.dottedd}.

\begin{figure}[htbp]
\begin{tikzpicture}[x=1cm, y=1cm]
    \draw[semithick] (0,0) -- (2,0);
    \node at (1,0) {$\times$};
    \node at (-0.3,0) {$\ast$};
    \node at (2.3,0) {$R$};

    \draw[semithick] (3.5,0) -- (5.5,0);
    \node at (4.5,0) {$\times$};
    \node at (3.2,0) {$\varnothing$};
    \node at (5.8,0) {$B$};
    
    \draw[semithick] (7,0) -- (9,0);
    \draw (8,0) circle (1mm) [fill=white];
    \node at (6.7,0) {$\varnothing$};
    \node at (9.3,0) {$R$};

    \draw[semithick] (10.5,0) -- (12.5,0);
    \draw (11.5,0) circle (1mm) [fill=white];
    \node at (10.2,0) {$\ast$};
    \node at (12.8,0) {$B$};
\end{tikzpicture}
\caption{Backward type propagation for mutation events. Deleterious mutations are marked by $\times$ and beneficial mutations by $\circ$.}
\label{fig.mutations}
\end{figure}
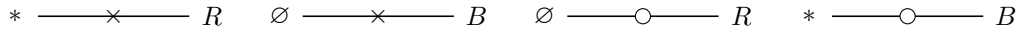

\textbf{Mutation events} (see Fig.~\ref{fig.mutations}): Assume now that a deleterious mutation event $\timesa$ occurs on line $\alpha\in I$. Then $D=\{\alpha\}$, $\tilde I = I$, and $C_I \to \omutr{\alpha}(C_I)=\longtilde{C}_{\tilde I}$, where, for $\ell\in {\tilde I}$,
\begin{equation}\label{eq.outcome_mutation_coalR}
   \longtilde{C}_\ell=\begin{cases}
        C_{\ell} &\text{ for } \ell\neq \alpha, \\
        \varnothing &\text{ for } \ell= \alpha\text{ if $C_\alpha\cap R=\varnothing$},\\
         \ast &\text{ for } \ell=\alpha\text{ if $C_\alpha\cap R\neq\varnothing$}.\\
    \end{cases}
\end{equation}
This is because, forward in time, a deleterious mutation assigns type $\mathrm{r}$ to line $\alpha$, irrespective of the configuration to the left of the event; consequently, rule (1a) leads to case~1 in \eqref{eq.outcome_mutation_coalR}, rule (B1) leads to cases 2 and 3, and rule (B2) is void since $D \cap (I \setminus \alpha)=\varnothing$. Similarly, if a mutation event $\circa$ occurs on line $\alpha\in I$, then again $D=\{\beta\}$, $\tilde{I}=(I\,\setminus\{\alpha\})\cup\{\beta\}$, but now
$C_I \to \omutb{\alpha}(C_I)=\longtilde{C}_{\tilde I}$ where, 
for $\ell\in \tilde I$,
\begin{equation}\label{eq.outcome_mutation_coalB}
   \longtilde{C}_\ell=\begin{cases}
        C_{\ell} &\text{ for } \ell\neq \alpha, \\
        \varnothing &\text{ for } \ell= \text{ if $C_\alpha\cap B=\varnothing$},\\
         \ast &\text{ for } \ell=\alpha\text{ if $C_\alpha\cap B\neq\varnothing$}.\\
    \end{cases}
\end{equation}

\begin{figure}[htbp]
\centering
\begin{subfigure}[c]{0.43 \textwidth}
\begin{tikzpicture}
\draw[semithick] (0.4,6) -- (4.5,6); 
\draw[semithick] (0.4,5) -- (2.5,5); 
\draw[semithick] (0.4,4) -- (2.5,4);
\draw[semithick] (0.4,3) -- (2.5,3);
\draw[semithick] (0.4,2) -- (2.5,2);
\draw[semithick] (0.4,1) -- (2.5,1);

\draw[semithick, dotted] (2.5,1) -- (4.5,1);
\draw[semithick, dotted] (2.5,2) -- (4.5,2);
\draw[semithick, dotted] (2.5,3) -- (4.5,3);
\draw[semithick, dotted] (2.5,4) -- (4.5,4);
\draw[semithick, dotted] (2.5,5) -- (4.5,5);

\def\x{2.5} 

\coordinate (a) at (\x,5);
\coordinate (b) at (\x,4);
\coordinate (c) at (\x,3);
\coordinate (d) at (\x,2);
\coordinate (e) at (\x,1);

\coordinate (target) at (\x,6);

\draw[-{open triangle 45[scale=5]},semithick] (a) -- (target);
\draw[-{open triangle 45[scale=5]},semithick] (b) .. controls (\x+0.5, 4.8) .. (target);
\draw[-{open triangle 45[scale=5]},semithick] (c) .. controls (\x+1, 4.6) .. (target);
\draw[-{open triangle 45[scale=5]},semithick] (d) .. controls (\x+1.5, 4.4) .. (target);
\draw[-{open triangle 45[scale=5]},semithick] (e) .. controls (\x+2, 4.2) .. (target);

\draw (a) node[shape=rectangle, draw, fill=white!100, inner sep=1.5pt, yshift=0.3ex] {\footnotesize $\vert \Bs \vert$};
\draw (d) node[shape=rectangle, draw, fill=white!100, inner sep=1.5pt, yshift=0.3ex] {\footnotesize $2$};
\draw (e) node[shape=rectangle, draw, fill=white!100, inner sep=1.5pt, yshift=0.3ex] {\footnotesize $1$};
\node at (2.5,4.6) {\vdots};
\node at (2.5,3.6) {\vdots};
\node at (2.5,2.6) {\vdots};

\node at (4.9,6) {$R$};

\node at (0,6) {$R$};
\node at (0,5) {$R$};
\node at (0,4.6) {\vdots};
\node at (0,4) {$R$};
\node at (0,3.6) {\vdots};
\node at (0,3) {$R$};
\node at (0,2.6) {\vdots};
\node at (0,2) {$R$};
\node at (0,1) {$R$};

\node[draw, semithick, minimum width=0.5cm, minimum height=5.6cm, anchor=center] at (0,3.5) {};
\end{tikzpicture}
\end{subfigure}
\hfill
\begin{subfigure}[c]{0.55 \textwidth}
\begin{tikzpicture}

\draw[semithick] (0.4,6) -- (4.5,6); 
\draw[semithick] (0.4,5) -- (4.5,5); 
\draw[semithick] (0.4,4) -- (4.5,4);
\draw[semithick] (0.4,3) -- (2.5,3);
\draw[semithick] (0.4,2) -- (2.5,2);
\draw[semithick] (0.4,1) -- (2.5,1);

\def\x{2.5} 

\coordinate (a) at (\x,5);
\coordinate (b) at (\x,4);
\coordinate (c) at (\x,3);
\coordinate (d) at (\x,2);
\coordinate (e) at (\x,1);

\coordinate (target) at (\x,6);

\draw[-{open triangle 45[scale=5]},semithick] (a) -- (target);
\draw[-{open triangle 45[scale=5]},semithick] (b) .. controls (\x+0.5, 4.8) .. (target);
\draw[-{open triangle 45[scale=5]},semithick] (c) .. controls (\x+1, 4.6) .. (target);
\draw[-{open triangle 45[scale=5]},semithick] (d) .. controls (\x+1.5, 4.4) .. (target);
\draw[-{open triangle 45[scale=5]},semithick] (e) .. controls (\x+2, 4.2) .. (target);

\draw (a) node[shape=rectangle, draw, fill=white!100, inner sep=1.5pt, yshift=0.3ex] {\footnotesize $\vert \Bs_0\vert$};
\draw (b) node[shape=rectangle, draw, fill=white!100, inner sep=1.5pt, yshift=0.3ex] {\footnotesize $1$};
\draw (c) node[shape=rectangle, draw, fill=white!100, inner sep=1.5pt, yshift=0.3ex] {\footnotesize $\vert \Bs_1\vert$};
\draw (d) node[shape=rectangle, draw, fill=white!100, inner sep=1.5pt, yshift=0.3ex] {\footnotesize $2$};
\draw (e) node[shape=rectangle, draw, fill=white!100, inner sep=1.5pt, yshift=0.3ex] {\footnotesize $1$};
\node at (2.5,4.6) {\vdots};
\node at (2.5,2.6) {\vdots};

\node at (4.9,6) {$B$};
\node at (4.9,5) {$R$};
\node at (4.9,4) {$R$};

\node at (-2.1,6) {$\ast$};
\node at (-1.4,6) {$\ast$};
\node at (-0.7,6) {$\ast$};
\node at (0,6) {$B$};

\node at (0,5) {$R$};
\node at (0,4.6) {\vdots};
\node at (0,4) {$R$};

\node at (-0.7,5) {$R$};
\node at (-0.7,4.6) {\vdots};
\node at (-0.7,4) {$R$};

\node at (-1.4,5) {$R$};
\node at (-1.4,4.6) {\vdots};
\node at (-1.4,4) {$R$};

\node at (-2.1,5) {$R$};
\node at (-2.1,4.6) {\vdots};
\node at (-2.1,4) {$R$};

\node at (-2.1,3) {$\ast$};
\node at (-1.4,3) {$\ast$};
\node at (-0.7,3) {$B$};
\node at (0,3) {$R$};

\node at (-2.1,2) {$\ast$};
\node at (-1.4,2) {$B$};
\node at (-0.7,2) {$R$};
\node at (0,2) {$R$};

\node at (-2.1,1) {$B$};
\node at (-1.4,1) {$R$};
\node at (-0.7,1) {$R$};
\node at (0,1) {$R$};

\foreach \x in {0,-0.7, -1.4, -2.1}  
{\node[draw, semithick, minimum width=0.5cm, minimum height=5.6cm, anchor=center] at (\x,3.5) {};}
\end{tikzpicture}
\end{subfigure}
\caption{Backward type propagation for an $\tilde s_{5}$ branching event. For the purpose of illustration, the continuing line is always displayed at the top, followed by the set $\Bs_0$ of incoming lines from the set $I$ and the set $\Bs_1$ of incoming lines from $[N]\setminus I$. An unfit descendant corresponds to all incoming lines being unfit \textit{(left)} and a fit descendant requires at least one fit potential parent, either the continuing line itself, a member of $\Bs_0$, or a member of $\Bs_1$ \textit{(right)}.} 
\label{fig.FTW_new}
\end{figure}
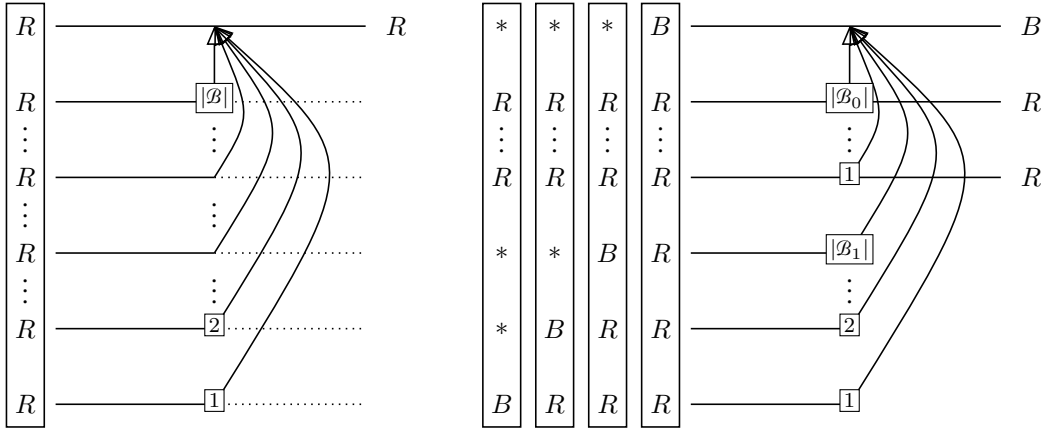

\textbf{Selective events} (see Fig.~\ref{fig.FTW_new}): Assume now a selective event $\ftw{\alpha}{\Bs}$ involving the line $\alpha\in I$ and the lines in $\Bs$. Further, let $\Bs=\Bs_0\cup\Bs_1$ with $\Bs_0\subseteq I$ and $\Bs_1\cap I=\varnothing$. We then have $D=\Bs$ and $\tilde{I}=I\cup\Bs=I\cup\Bs_1$. The mapping now is $C_I\mapsto \oftw{\alpha}{\Bs}(C_I)$, where the latter is defined as follows.
If $C_\alpha\neq B$, then $\oftw{\alpha}{\Bs}(C_I)=\{\longtilde{C}_{\tilde{I}}\}$, where
\begin{equation}\label{eq.selR}
   \longtilde{C}_\ell=\begin{cases}
   C_\ell &\text{ for } \ell\notin \Bs \cup \{\alpha\}, \\
        C_\alpha \cap C^*_\ell &\text{ for } \ell \in \Bs \cup \{\alpha\}.\\
    \end{cases}
\end{equation}
By now, the role of the rules (A), (B1), (B2) should have become clear, so we can concentrate on the essentials and see that the transition reflects that, forward in time, the descendant at line $\alpha$ is unfit if and only if all potential parents (that is, the elements of $\Bs \cup \{\alpha\}$) are unfit; there is no such condition if $C_\alpha=*$.
If $C_\alpha= B$ and there is $\beta_0\in\Bs_0 \setminus \{\alpha\}$ with $C_{\beta_0}=B$, then $\oftw{\alpha}{\Bs}(C_I)=\{\longtilde{C}_{\tilde{I}}\}$,
where
\begin{equation}\label{eq.selB}
   \longtilde{C}_\ell=\begin{cases}
        C_{\ell} &\text{ for } \ell\notin \Bs_1\cup \{\alpha\}, \\
        \ast &\text{ for } \ell\in\Bs_1\cup\{\alpha\}.\\
    \end{cases}
\end{equation}
This transition comes from the fact that, under FTW, the descendant is fit precisely when at least one of the potential parents is fit. If such a fit potential parent is in $\Bs_0\setminus \{\alpha\}$, then $C_\alpha=B$, regardless of the types of the lines in $\Bs_1\cup\{\alpha\}$ to the left of the event; all other lines continue unchanged.
If $C_\alpha= B$ and $C_{\beta_0}=R$ for all $\beta_0\in\Bs_0 \setminus \{\alpha\}$, then $C_I\mapsto \oftw{\alpha}{\Bs}(C_I)
=\{\longtilde{C}_{\tilde{I}}^b\}_{b\in \Bs_1 \cup \{\alpha\}}$, where

\begin{equation}\label{eq.selB1}
\longtilde{C}_\ell^b=
   \begin{cases}
        C_{\ell} &\text{ for } \ell\notin \Bs_1\cup\{\alpha\}, \\
         R &\text{ for } b>\ell\in\Bs_1 \cup \{\alpha\},\\
       B &\text{ for } \ell=b,\\
        \ast&\text{ for } b<\ell\in\Bs_1 \cup \{\alpha\}.
    \end{cases} 
 \quad\text{for } \, b \in \Bs_1, \text{ and} \;    
\longtilde C_\ell^\alpha=\begin{cases}
        C_{\ell} &\text{ for } \ell\notin \Bs_1\cup\{\alpha\}, \\
        B &\text{ for } \ell=\alpha,\\
        R&\text{ for } \ell\in\Bs_1.
    \end{cases} 
\end{equation}
Note that in contrast to the previous cases, $\oftw{\alpha}{\Bs}$ is not a cylinder, but a collection of (disjoint) cylinders, indexed according to where the first $B$ appears (for $b \in \Bs_{1}$, it appears on line $b$ when looking at the lines in $\Bs_1$ in increasing order). 

\begin{figure}[htbp]
\begin{subfigure}[c]{0.3\textwidth}
\centering
    \scalebox{0.8}{
    \begin{tikzpicture}
    \draw[semithick] (0.4,2) -- (4.5,2); 
    \draw[semithick] (0.4,1) -- (2.5,1); 
    \draw[semithick] (0.4,0) -- (2.5,0); 
    
    \draw[-{triangle 45[scale=5]}, semithick] (2.5,1) -- (2.5,2); 
    \draw[-{Diamond[open, scale = 1.3]}, semithick, dashed ] (2.5,1)-- (2.5,0); 

    \node at (4.7,2) {$R$};

    \node at (0.1,2) {$\ast$};
    \node at (-0.5,2) {$R$};

    \node at (0.1,1) {$R$};
    \node at (-0.5,1) {$\ast$};

    \node at (0.1,0) {$B$};
    \node at (-0.5,0) {$R$};

    \foreach \x in {0.1,-0.5} 
    {\node[draw, semithick, minimum width=0.5cm, minimum height=2.6cm, anchor=center] at (\x,1) {};}        
\end{tikzpicture}}
\end{subfigure}
\hfill
\begin{subfigure} [c]{0.3\textwidth}
\centering
\scalebox{0.8}{
\begin{tikzpicture}

    \draw[semithick] (0.4,2) -- (4.5,2); 
    \draw[semithick] (0.4,1) -- (4.5,1); 
    \draw[semithick] (0.4,0) -- (2.5,0); 
    
    \draw[-{triangle 45[scale=5]}, semithick] (2.5,1) -- (2.5,2); 
    \draw[-{Diamond[open, scale = 1.3]}, semithick, dashed] (2.5,1) -- (2.5,0); 
  
    \node at (4.7,1) {$R$};
    \node at (4.7,2) {$R$};

    \node at (-0.5,2) {$\ast$};
    \node at (0.1,2) {$R$};

    \node at (-0.5,1) {$R$};
    \node at (0.1,1) {$R$};

    \node at (-0.5,0) {$B$};
    \node at (0.1,0) {$R$};

    \foreach \x in {0.1,-0.5} 
    {\node[draw, semithick, minimum width=0.5cm, minimum height=2.6cm, anchor=center] at (\x,1) {};} 
\end{tikzpicture}}
\end{subfigure}
\hfill
\begin{subfigure} [c]{0.3\textwidth}
\centering
\scalebox{0.8}{
\begin{tikzpicture}

    \draw[semithick] (0.4,2) -- (4.5,2); 
    \draw[semithick] (0.4,1) -- (2.5,1); 
    \draw[semithick] (0.4,0) -- (4.5,0); 
    
    \draw[-{triangle 45[scale=5]}, semithick] (2.5,1) -- (2.5,2); 
    \draw[-{Diamond[open, scale = 1.3]}, semithick, dashed] (2.5,1) -- (2.5,0); 

    \node at (4.7,0) {$R$};
    \node at (4.7,2) {$R$};

    \node at (0.1,2) {$R$};

    \node at (0.1,1) {$\ast$};

    \node at (0.1,0) {$R$};

    \foreach \x in {0.1} 
    {\node[draw, semithick, minimum width=0.5cm, minimum height=2.6cm, anchor=center] at (\x,1) {};}
    
\end{tikzpicture}}
\end{subfigure}
\caption{Backward type propagation for the cylinder $C_I=R_I$ under the event $\protect\ine{\alpha}{\beta}{\gamma}$ for $\gamma \notin \{\alpha, \beta \}$. The leftmost picture displays ternary branching, and the other two pairwise branching. Not depicted is the silent event $\{\alpha, \beta, \gamma\} \subseteq I$.}
\label{Fig.alphabetagamma} 
\end{figure}
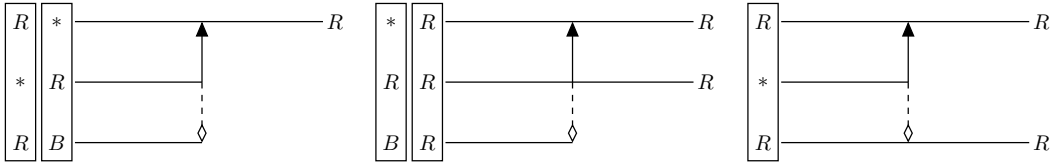

\begin{figure}[htbp]%
\noindent%
\begin{subfigure}[t]{0.21\textwidth}
    \centering
    \vtop{%
    \null
    \hbox{
    \scalebox{0.8}{
   \begin{tikzpicture}[baseline=(current bounding box.north)]
   \path[use as bounding box] (-0.5,0.8) rectangle (2.3,2.2);
        \draw[semithick] (0.4,2) -- (2.25,2); 
        \draw[semithick] (0.4,1) -- (1.25,1); 
        \draw[-{triangle 45[scale=5]}, semithick] (1.25,1) -- (1.25,2); 
        \draw[-{Diamond[open, scale = 1.3]}, semithick, dashed, shorten >= 2pt] (1.25,1) .. controls (1.7,1.2) and (1.7,1.8) .. (1.25,2); 
        \node at (2.45,2) {$R$};
        \node at (-0.5,2) {$B$};
        \node at (0.1,2) {$R$};
        \node at (-0.5,1) {$R$};
        \node at (0.1,1) {$\ast$};
        \foreach \x in {0.1,-0.5} 
        {\node[draw, semithick, minimum width=0.5cm, minimum height=1.6cm, anchor=center] at (\x,1.5) {};}
        \end{tikzpicture}}}}
\end{subfigure}
\hfill
\begin{subfigure}[t]{0.19\textwidth}
    \centering
    \vtop{%
    \null
    \hbox{
    \scalebox{0.8}{
   \begin{tikzpicture}[baseline=(current bounding box.north)]
   \path[use as bounding box] (0,0.8) rectangle (2.3,2.2);
        \draw[semithick] (0.4,2) -- (2.25,2); 
        \draw[semithick] (0.4,1) -- (1.25,1); 
        \draw[-{triangle 45[scale=5]}, semithick] (1.25,1) -- (1.25,2); 
        \draw[-{Diamond[open, scale = 1.3]}, semithick, dashed, shorten >= 2pt] (1.25,1) .. controls (1.7,1.2) and (1.7,1.8) .. (1.25,2); 
        \node at (2.45,2) {$B$};
        \node at (0.1,2) {$B$};
        \node at (0.1,1) {$B$};
        \foreach \x in {0.1} 
        {\node[draw, semithick, minimum width=0.5cm, minimum height=1.6cm, anchor=center] at (\x,1.5) {};}
        \end{tikzpicture}}}}
\end{subfigure}
\hfill
\begin{subfigure}[t]{0.19\textwidth}  
    \centering  
    \vtop{%
    \null
    \hbox{
    \scalebox{0.8}{
     \begin{tikzpicture}[baseline=(current bounding box.north)]
     \path[use as bounding box] (0,0.8) rectangle (2.3,2.2);
            \draw[semithick] (0.4,2) -- (2.25,2); 
            \draw[semithick] (0.4,1) -- (1.25,1); 
            \draw[-{Triangle[scale=1.5]}, semithick] (1.25,1) -- (1.25,2); 
            \draw[-{Diamond[open, scale = 1.3]}, semithick, dashed,] (1.25,1)  to[out=45, in=315,looseness=30] (1.25,0.9); 
            \node at (2.45,2) {$R$};
            \node at (0.1,2) {$R$};
            \node at (0.1,1) {$R$};
            \node[draw, semithick, minimum width=0.5cm, minimum height=1.6cm, anchor=center] at (0.1,1.5) {};
        \end{tikzpicture}}}}
\end{subfigure}
\hfill
\begin{subfigure}[t]{0.21\textwidth}  
    \centering  
    \vtop{%
    \null
    \hbox{
    \scalebox{0.8}{
     \begin{tikzpicture}[baseline=(current bounding box.north)]
     \path[use as bounding box] (-0.5,0.8) rectangle (2.3,2.2);
            \draw[semithick] (0.4,2) -- (2.25,2); 
            \draw[semithick] (0.4,1) -- (1.25,1); 
            \draw[-{Triangle[scale=1.5]}, semithick] (1.25,1) -- (1.25,2); 
            \draw[-{Diamond[open, scale = 1.3]}, semithick, dashed,] (1.25,1)  to[out=45, in=315,looseness=30] (1.25,0.9); 
            \node at (2.45,2) {$B$};
            \node at (0.1,2) {$B$};
            \node at (-0.5,2) {$\ast$};           
            \node at (0.1,1) {$R$};
            \node at (-0.5,1) {$B$};
            \foreach \x in {0.1,-0.5} 
            \node[draw, semithick, minimum width=0.5cm, minimum height=1.6cm, anchor=center] at (\x,1.5) {};
       \end{tikzpicture}}}}
\end{subfigure}
\hfill
\begin{subfigure}[t]{0.19\textwidth}
    \centering
    \vtop{%
    \null
    \hbox{
    \scalebox{0.8}{
     \begin{tikzpicture}[baseline=(current bounding box.north)]
     \path[use as bounding box] (0,0.8) rectangle (2.3,2.2);
            \draw[semithick] (0.4,2) -- (2.25,2); 
            \draw[semithick] (0.4,1) -- (2.25,1); 
            \draw[-{Triangle[scale=1.5]}, semithick] (1.25,1) -- (1.25,2); 
            \draw[-{Diamond[open, scale = 1.3]}, semithick, dashed,] (1.25,1)  to[out=45, in=315,looseness=30] (1.25,0.9); 
            \node at (2.45,2) {$B$};
            \node at (2.45,1) {$B$};
            \node at (0.1,2) {$\ast$};        
            \node at (0.1,1) {$B$};
            \foreach \x in {0.1} 
            \node[draw, semithick, minimum width=0.5cm, minimum height=1.6cm, anchor=center] at (\x,1.5) {};
       \end{tikzpicture}}}}
\end{subfigure}
\caption{Backward type propagation for the cylinders $C_I=R_I$ and $C_I=B_I$ under the event $\protect\ine{\alpha}{\beta}{\gamma}$ for $\gamma\in \{\alpha, \beta\}$. For $\beta \notin I$, this results in binary branching and for $\beta \in I$ in a collision. For the cylinder $R_I$, this collision is silent and not depicted, whereas for the cylinder $B_I$  it is silent if $\gamma=\alpha$ and non-silent for $\gamma=\beta$.}
\label{fig.binary} 
\end{figure}
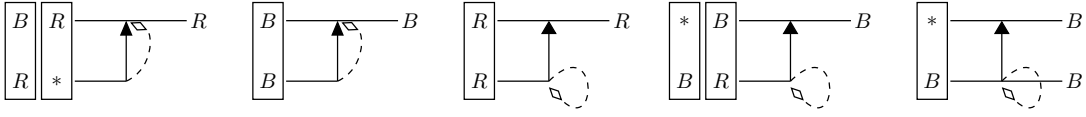

\textbf{Interactive neutral events} (see Figs.~\ref{Fig.alphabetagamma} and~\ref{fig.binary}): Assume now an interactive neutral event $\ine{\alpha}{\beta}{\gamma}$ involving the line $\alpha\in I$ and the lines $\beta \in [N]\setminus \{\alpha\},\, \gamma\in[N]$. Here, $D=\{\alpha,\beta,\gamma\}$ and $\tilde{I}=I\cup\{\gamma,\beta\}$. Then $C_I\mapsto  \{\onint{\alpha}{\beta}{\gamma}^{1}(C_I),\onint{\alpha}{\beta}{\gamma}^{2}(C_I)\}$, where $\onint{\alpha}{\beta}{\gamma}^{1}(C_I) = \longtilde{C}_{\tilde{I}}^1$ and $\onint{\alpha}{\beta}{\gamma}^{2}(C_I) =\longtilde{C}_{\tilde{I}}^2$ are defined as follows. If $\gamma \notin \{\alpha, \beta\}$ (see Fig.~\ref{Fig.alphabetagamma}),
\begin{equation}\label{eq.nint1}
   \longtilde{C}_\ell^1=\begin{cases}
        C_{\ell} &\text{ for } \ell\notin\{\gamma,\beta,\alpha\}, \\
        C_\alpha &\text{ for } \ell=\alpha,\\
        C^\ast_\beta &\text{ for } \ell=\beta,\\
        R\cap C^\ast_\gamma&\text{ for } \ell=\gamma,
    \end{cases}    
    \text{ and }\quad  
   \longtilde{C}_\ell^2=\begin{cases}
        C_{\ell} &\text{ for } \ell\notin\{\gamma,\beta,\alpha\}, \\
       \ast &\text{ for } \ell=\alpha,\\
        C_\alpha\cap C^\ast_\beta &\text{ for } \ell=\beta,\\
        B\cap C^\ast_\gamma&\text{ for } \ell=\gamma.
    \end{cases}  
\end{equation}
This reflects the case distinction whether 1) line $\gamma$ is unfit, whence nothing happens, or 2) line $\gamma$ is fit and hence, forward in time, line $\alpha$ receives the type from line $\beta$ -- regardless of the type of the former to the left of the event. 
If $\alpha=\gamma\neq\beta$, we have\begin{equation}\label{eq.nint2}
   \longtilde{C}_\ell^1=\begin{cases}
        C_{\ell} &\text{ for } \ell\notin\{\beta,\alpha\}, \\
       B &\text{ for } \ell=\alpha \,(=\gamma),\\
        B \cap C^\ast_\beta &\text{ for } \ell=\beta,\\
    \end{cases}    
    \text{ and }\quad  
   \longtilde{C}_\ell^2=\begin{cases}
        C_{\ell} &\text{ for } \ell\notin\{\beta,\alpha\}, \\
        R\cap C_\alpha &\text{ for } \ell=\alpha\,(=\gamma),\\
        C^\ast_\beta &\text{ for } \ell=\beta.\\
    \end{cases}  
\end{equation}
Once more, this is due to the case distinction of whether 1) line $\gamma$ is fit or 2) unfit. But since it now agrees with line $\alpha$, this means that 1) line $\alpha$ inherits its (fit) type from line $\beta$, or 2) nothing happens. 
Finally, if $\alpha\neq \beta=\gamma$, we have
\begin{equation}\label{eq.nint3}
   \longtilde{C}_\ell^1=
   \begin{cases}
        C_{\ell} &\text{ for } \ell\notin\{\beta,\alpha\}, \\
        \psi_R(C_\alpha) &\text{ for } \ell=\alpha,\\
   B\cap C^\ast_\beta &\text{ for } \ell=\beta\, (=\gamma),\\
    \end{cases}   
    \text{ and }\quad  
   \longtilde{C}_\ell^2= 
      \begin{cases}
        C_{\ell} &\text{ for } \ell\notin\{\beta,\alpha\}, \\
        C_\alpha &\text{ for } \ell=\alpha,\\
        R\cap C^\ast_\beta &\text{ for } \ell=\beta\, (=\gamma),\\
    \end{cases}  
\end{equation}
where $\psi_R(R)=\varnothing$ and $\psi_R(c)=\ast$ for $c\neq R$. 
With the same case distinction as before, we have the possibilities that 1) line $\gamma$ is fit; but since it now agrees with line $\beta$, the type of line $\alpha$ will only come from line $\beta$ if the former is compatible with $B$, otherwise the cylinder is empty; and 2) line $\beta$ is unfit and nothing happens.

Since \eqref{eq.nint1}, \eqref{eq.nint2}, and \eqref{eq.nint3} do not require the information which of the lines $\beta$ and $\gamma$ belong to the set $I$, we need no longer distinguish between ternary branching, pairwise branching, and collisions. The only distinction we have to make is whether the checking line is different from the incoming and the continuing lines (see \eqref{eq.nint1}), equal to the continuing line (see \eqref{eq.nint2}), or equal to the incoming line (see \eqref{eq.nint3}). These cases are summarized in Figure~\ref{fig.Rcilinders}. 

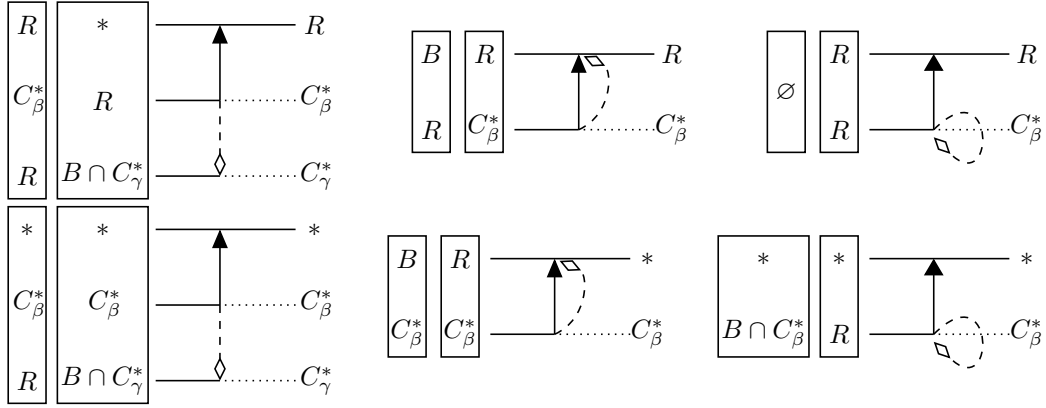
\begin{figure}[htbp]
\begin{subfigure}{0.21\textwidth}
    \centering  
    \vtop{%
    \null
    \hbox{
\begin{tikzpicture}
    \draw[semithick] (0.4,2) -- (2.25,2); 
    \draw[semithick] (0.4,1) -- (1.25,1); 
    \draw[semithick] (0.4,0) -- (1.25,0); 
    \draw[semithick,dotted] (2.25,1) -- (1.25,1); 
    \draw[semithick,dotted] (2.25,0) -- (1.25,0); 
    
    \draw[-{triangle 45[scale=5]}, semithick] (1.25,1) -- (1.25,2); 
    \draw[-{Diamond[open, scale = 1.3]}, semithick, dashed] (1.25,1) -- (1.25,0); 
  
    \node at (2.5,2) {$R$};
    \node at (2.5,1) {$\,C_\beta^\ast$};
    \node at (2.5,0) {$\,C_\gamma^\ast$};

    \node at (-1.3,2) {$R$};
    \node at (-0.3,2) {$\ast$};
    \node at (-1.3,1) {$\,C_\beta^\ast$};
    \node at (-0.3,1) {$R$};
    \node at (-1.3,0) {$R$};
  \node at (-0.3,0) {$B\cap C_\gamma^\ast$};
    \foreach \x in {-0.3} 
    {\node[draw, semithick, minimum width=1.2cm, minimum height=2.6cm, anchor=center] at (\x,1) {};}
    \node[draw, semithick, minimum width=0.5cm, minimum height=2.6cm, anchor=center] at (-1.3,1) {};
\end{tikzpicture}}}
\end{subfigure}
\hfill
\begin{subfigure}[t]{0.21\textwidth}
    \vtop{%
    \null
    \hbox{
   \begin{tikzpicture}[baseline=(current bounding box.north)]
        \draw[semithick] (0.4,2) -- (2.25,2); 
        \draw[semithick] (0.4,1) -- (1.25,1); 
           \draw[semithick,dotted] (2.25,1) -- (1.25,1); 
        \draw[-{triangle 45[scale=5]}, semithick] (1.25,1) -- (1.25,2); 
        \draw[-{Diamond[open, scale = 1.3]}, semithick, dashed, shorten >= 2pt] (1.25,1) .. controls (1.7,1.2) and (1.7,1.8) .. (1.25,2); 
        \node at (2.45,2) {$\,R$};
        \node at (2.45,1) {$\,C_\beta^\ast$};
        \node at (-0.7,2) {$B$};
        \node at (0,2) {$R$};
    
        \node at (-0.7,1) {$R$};
        \node at (0,1) {$C_\beta^\ast$};

        \foreach \x in {0,-0.7} 
        {\node[draw, semithick, minimum width=0.5cm, minimum height=1.6cm, anchor=center] at (\x,1.5) {};}
        \end{tikzpicture}}}
\end{subfigure}
\hfill
\begin{subfigure}[t]{0.21\textwidth}  
    \centering  
    \vtop{%
    \null
    \hbox{
     \begin{tikzpicture}[baseline=(current bounding box.north)]
            \draw[semithick] (0.4,2) -- (2.25,2); 
            \draw[semithick] (0.4,1) -- (1.25,1); 
            \draw[semithick,dotted] (2.25,1) -- (1.25,1); 
            \draw[-{Triangle[scale=1.5]}, semithick] (1.25,1) -- (1.25,2); 
            \draw[-{Diamond[open, scale = 1.3]}, semithick, dashed,] (1.25,1)  to[out=45, in=315,looseness=30] (1.25,0.9); 

            \node at (2.45,2) {$\,R$};
             \node at (2.45,1) {$\,C_\beta^\ast$};
            \node at (0,2) {$R$};
            \node at (-0.7,1.5) {$\varnothing$};

            \node at (0,1) {$R$};

            \foreach \x in {0,-0.7} 
         \node[draw, semithick, minimum width=0.5cm, minimum height=1.6cm, anchor=center] at (\x,1.5) {};
       \end{tikzpicture}}}
\end{subfigure}
\begin{subfigure}{0.21\textwidth}
    \centering  
    \vtop{%
    \null
    \hbox{
\begin{tikzpicture}
    \draw[semithick] (0.4,2) -- (2.25,2); 
    \draw[semithick] (0.4,1) -- (1.25,1); 
    \draw[semithick] (0.4,0) -- (1.25,0); 
    \draw[semithick,dotted] (2.25,1) -- (1.25,1); 
    \draw[semithick,dotted] (2.25,0) -- (1.25,0); 
    
    \draw[-{triangle 45[scale=5]}, semithick] (1.25,1) -- (1.25,2); 
    \draw[-{Diamond[open, scale = 1.3]}, semithick, dashed] (1.25,1) -- (1.25,0); 
  
    \node at (2.5,2) {$\ast$};
    \node at (2.5,1) {$\,C_\beta^\ast$};
    \node at (2.5,0) {$\,C_\gamma^\ast$};

    \node at (-0.3,2) {$\ast$};
    \node at (-1.3,2) {$\ast$};
    \node at (-0.3,1) {$\,C_\beta^\ast$};
    \node at (-1.3,1) {$\,C_\beta^\ast$};
    \node at (-0.3,0) {$B\cap C_\gamma^\ast$};
    \node at (-1.3,0) {$R$};
    \foreach \x in {-0.3} 
    {\node[draw, semithick, minimum width=1.2cm, minimum height=2.6cm, anchor=center] at (\x,1) {};}
    {\node[draw, semithick, minimum width=0.5cm, minimum height=2.6cm, anchor=center] at (-1.3,1) {};}
\end{tikzpicture}}}
\end{subfigure}
\hfill
\begin{subfigure}[t]{0.21\textwidth}
    \vtop{%
    \null
    \hbox{
   \begin{tikzpicture}[baseline=(current bounding box.north)]
        \draw[semithick] (0.4,2) -- (2.25,2); 
        \draw[semithick] (0.4,1) -- (1.25,1); 
        \draw[semithick,dotted] (2.25,1) -- (1.25,1); 
        \draw[-{triangle 45[scale=5]}, semithick] (1.25,1) -- (1.25,2); 
        \draw[-{Diamond[open, scale = 1.3]}, semithick, dashed, shorten >= 2pt] (1.25,1) .. controls (1.7,1.2) and (1.7,1.8) .. (1.25,2); 

        \node at (2.45,2) {$\,\ast$};
        \node at (2.45,1) {$\,C_\beta^\ast$};
        \node at (-0.7,2) {$B$};
        \node at (0,2) {$R$};
    
        \node at (-0.7,1) {$C_\beta^\ast$};
        \node at (0,1) {$C_\beta^\ast$};

        \foreach \x in {0,-0.7} 
        {\node[draw, semithick, minimum width=0.5cm, minimum height=1.6cm, anchor=center] at (\x,1.5) {};}
        \end{tikzpicture}}}
\end{subfigure}
\hfill
\begin{subfigure}[t]{0.21\textwidth}  
    \centering  
    \vtop{%
    \null
    \hbox{
     \begin{tikzpicture}[baseline=(current bounding box.north)]
            \draw[semithick] (0.4,2) -- (2.25,2); 
            \draw[semithick] (0.4,1) -- (1.25,1); 
            \draw[semithick,dotted] (2.25,1) -- (1.25,1); 
            \draw[-{Triangle[scale=1.5]}, semithick] (1.25,1) -- (1.25,2); 
            \draw[-{Diamond[open, scale = 1.3]}, semithick, dashed,] (1.25,1)  to[out=45, in=315,looseness=30] (1.25,0.9); 

            \node at (2.45,2) {$\,\ast$};
             \node at (2.45,1) {$\,C_\beta^\ast$};
            \node at (0,2) {$\ast$};
            \node at (-1,2) {$\ast$};
            \node at (0,1) {$R$};
            \node at (-1,1) {{$B\cap C_\beta^\ast$}};

            \node[draw, semithick, minimum width=0.5cm, minimum height=1.6cm, anchor=center] at (0,1.5) {};
            \node[draw, semithick, minimum width=1.2cm, minimum height=1.6cm, anchor=center] at (-1,1.5) {};
       \end{tikzpicture}}}
\end{subfigure}
\caption{Examples illustrating \eqref{eq.nint1} \textit{(left)}, \eqref{eq.nint2} \textit{(middle)}, and \eqref{eq.nint3} \textit{(right)} for \(R\)-cylinders with \(C_\alpha = R\) (top) and \(C_\alpha=\ast\) (bottom). Lines that may or may not belong to the sample $I$ on the right are dotted; their types are set to \(\ast\) when they do not belong to $C_I$. Note that, for $R$-cylinders, $C_\beta^*\in \{R, \ast\}$, and we read from left to right.}
\label{fig.Rcilinders} 
\end{figure}

Note that, depending on the type configuration, some cylinders may be empty. Although one could in principle discard them, it is convenient to always produce two cylinders in order to avoid case distinctions. Moreover, for a majority of cylinders we have 
\begin{equation}\label{eq.identity}
\longtilde{C}_{\tilde{I}}^1 \cup \longtilde{C}_{\tilde{I}}^2=C^{\ast}_{\tilde I}    
\end{equation}
and therefore could identify $\onint{\alpha}{\beta}{\gamma}$ with the identity. However, since we are aiming for a unified framework and \eqref{eq.identity} does not hold for all cylinders (see Table~\ref{tab:cylinders}), we keep $\longtilde{C}_{\tilde{I}}^1$ and $\longtilde{C}_{\tilde{I}}^2$ separate.

\begin{table}[htbp]
\centering
\caption{The cylinders $C_D$ for which \eqref{eq.identity} does not hold, along with the respective cylinders $\longtilde{C}_{\tilde{I}}^1$ and $\longtilde{C}_{\tilde{I}}^2$ at the relevant positions. On the left we have $\alpha=1,  \beta=2$ and $\gamma=3$ and on the right $\alpha=1,  \beta=2$ and $\gamma$ as specified. We use $T \in \{R, B\}$ for an arbitrary type.}
\label{tab:cylinders}
\begin{minipage}[t]{0.45\textwidth}
\vspace{0pt}
\centering
\begin{tabular}{c|c|c}
  $C_D$ & $\longtilde{C}^{1}_{\tilde{I}}$ 
 & $\longtilde{C}^{2}_{\tilde{I}}$  \\[2pt]
\hline\rule{0pt}{2.5ex} 
$(T,\ast, \ast)$& $(T,\ast,R)$ & $(\ast,T,B)$ \\[2pt]
$ (T,T, \ast)$& $(T,T,R)$ & $(\ast,T,B)$ \\[2pt]
$(B,\ast, B)$& $\varnothing$ &$(\ast,B,B)$ \\[2pt]
$(B,B, B)$& $\varnothing$ &$(\ast, B, B)$\\[2pt]
$(B,R, B)$& $\varnothing$ &$\varnothing$ \\[2pt]
$(B,R, \ast)$& $(B, R, R)$ &$\varnothing$ \\[2pt]
$(R,B, \ast)$& $(R, B, R)$ &$\varnothing$ \\[2pt]
$(R,\ast, B)$& $\varnothing$ &$(\ast, R, B)$ \\[2pt]
$(R,R, B)$& $\varnothing$ &$(\ast, R, B)$ \\[2pt]
$(R,B, B)$&$\varnothing$ &$\varnothing$ 
\end{tabular}
\end{minipage}
\hfill
\begin{minipage}[t]{0.45\textwidth}
\vspace{0pt}
\centering
\begin{tabular}{c|c|c}
$C_D, \gamma=\alpha$ & $\longtilde{C}^1_{\tilde{I}}$ 
 & $\longtilde{C}^2_{\tilde{I}}$   \\[2pt]
\hline\rule{0pt}{2.5ex} 
$(R,R)$& $(B,R)$ & $(R,R)$ \\[2pt]
$(R,\ast)$& $(B,R)$ & $(R,\ast)$ \\[2pt]
$(B,\ast)$& $(B,B)$ & $\varnothing$ \\[2pt]
$(B, R)$& $\varnothing$ & $\varnothing$ \\[2pt]
\end{tabular} 
\par\vspace{0.4cm}
\begin{tabular}{c|c|c}
$C_D, \gamma=\beta$ & $\longtilde{C}^1_{\tilde{I}}$ 
 & $\longtilde{C}^2_{\tilde{I}}$  \\[2pt]
\hline\rule{0pt}{2.5ex} 
$(R,B)$& $\varnothing$ & $\varnothing$ \\[2pt]
$(B,B)$& $(\ast,B)$ & $\varnothing$ \\[2pt]
$(R,\ast)$& $\varnothing$ & $(R,R)$ \\[2pt]
$(B,\ast)$& $(\ast,B)$ & $(B,R)$ \\[2pt]
\end{tabular}
\end{minipage}
\end{table}

Having described how the different events in the AIG act at the level of cylinders, we now return to the set of compatible configurations and introduce a dynamical version of $\Ss(\aigt, S)$.

\begin{definition}[The configuration process]\label{def.ordered_full_info}
Let $\gs_0\subseteq[N]$ and fix a root-type configuration $S\in\varSigma_{\gs_0}^R$. The \emph{configuration process}, denoted by
$(\Theta_t(S))_{t\geq 0}$, is constructed on the basis of the transitions of $(\As_{[0,t]}(\gs_0))_{t\geq 0}$ as follows.
The process starts at
$\Theta_0(S)= S$. Suppose that at some time the state of the process is the finite collection of cylinders
$\{C^{(i)}\}_{i=1}^r$ for some $r>0$.
\begin{enumerate}
\item At the times of a coalescence or a relocation event $\coal{\alpha}{\beta}$, every cylinder $C^{(i)}$, $i \in [r]$, performs the transition $$C^{(i)}\mapsto\ocoal{\alpha}{\beta}(C^{(i)}),$$
and the new state of the process is the collection $\{\ocoal{\alpha}{\beta}(C^{(i)})\}_{i=1}^r$.
\item At the times of a deleterious mutation $\timesa$, every cylinder $C^{(i)}$, $i \in [r]$, performs the transition $$C^{(i)}\mapsto\omutr{\alpha}(C^{(i)}),$$
and the new state of the process is the collection$\{ \omutr{\alpha}(C^{(i)})\}_{i=1}^r$. 
\item At the times of a beneficial mutation $\circa$, every cylinder $C^{(i)}$, $i \in [r]$, performs the transition $$C^{(i)}\mapsto \omutb{\alpha}(C^{(i)}),$$
and the new state of the process is the collection $\{\omutb{\alpha}(C^{(i)})\}_{i=1}^r$. 
\item At the times of an FTW event $\ftw{\alpha}{\Bs}$, every cylinder $C^{(i)}$, $i \in [r]$, performs the transition $$C^{(i)}\mapsto\oftw{\alpha}{\Bs}(C^{(i)}),$$
and the new state of the process is the collection $ \bigcup_{i=1 }^r \oftw{\alpha}{\Bs}(C^{(i)})$.
\item At the times of an interactive neutral event $\ine{\alpha}{\beta}{\gamma}$, every cylinder $C^{(i)}$, $i \in [r]$, performs the transition
$$C^{(i)}\mapsto\onint{\alpha}{\beta}{\gamma}(C^{(i)}),$$
and the new state of the process is the collection $\bigcup_{i=1 }^r  \onint{\alpha}{\beta}{\gamma}(C^{(i)})$.
\end{enumerate}
By construction, we have
\begin{equation}\label{eq.compatible_union}
\Ss\big(\As_{[0,t]}(\gs_0), S\big)=\bigcup_{C\in\Theta_t(S)} C.    
\end{equation}
\end{definition}

With the configuration process and with $\Gamma^{(i)}$ as in Lemma~\ref{prop_average1}, we can express a main object of our interest, namely the probability for the random configuration $\Gamma^{(i)}$ to be compatible with an AIG $\As_{[0,t]}(\gs_0)$ and an initial sample $S$, as 
\begin{equation}\label{eq.connection_config_compat}
\Pb\big(\Gamma^{(i)}\in \Ss(\As_{[0,t]}(\gs_0), S)\big)=\Pb\big(\Gamma^{(i)}\in \Theta_t(S)\big).   
\end{equation}

In the above definition, we introduced the configuration process via the AIG and hence as function of $\sP, \, \gs_0$, and $S$. We will sometimes also make the dependence on $\sP$ explicit and write $\Theta_t(S, \sP)$. Moreover, it is clear from the construction that it can equivalently be defined as an autonomous continuous-time Markov chain on finite collections of cylinders. The transition rates can be read off directly from the rates of the corresponding transitions in the AIG.

While the transitions at the level of cylinders have been established, their connection to the factorial dual remains elusive. Indeed, we have seen that an $R$-cylinder can transform into a cylinder containing a $B$, which implies that unfit potential influencers can be compatible with an entirely unfit sample. As is standard for models of this type, one might be tempted to conclude that the forward evolution does not admit a factorial dual; however, as already stated in Theorem~\ref{thm.factorial_duality}, this conclusion is not correct. The next section explains 
how the configuration process is indeed connected to the factorial moment duality.

\section{The path to the factorial duality}\label{sec:our_approach}
This section starts with a heuristic motivation concerning the connection between the natural ancestry and the factorial moment dual. In several steps, these heuristics are made rigorous and the factorial moment duality is proven via the Frankenstein process. 

\subsection{Heuristics}
Recall from the introduction (and from \cite{BEH23}) that, in the case $\kappa=0$, the sample at the root has an entirely unfit configuration if and only if all potential ancestors (at any given time $t$ before the present) are unfit. This allows to express the set of compatible configurations $\Ss(\aigt, R_{\gs_0})$ (and hence the configuration process) in terms of the process $(Z^{(N)}_t)_{t \geq 0}$, where $Z^{(N)}_t$ is the \emph{number} of potential ancestors at time $t$ before the present, and gives rise to the factorial moment duality in Theorem~\ref{thm.factorial_duality}. As soon as there are interactive neutral events, however, 
we observe that an entirely unfit configuration at the root is no longer equivalent to an entirely unfit sample at the leaves; rather, such events may make fit ($\mathrm{b}$) leaves compatible with an unfit ($\mathrm{r}$) root, see Figures~\ref{fig.Rcilinders} and~\ref{fig.heuristics}. To connect these situations to a counting process nevertheless, let us first consider a ternary branching event that happens to a single unfit individual, so $\Theta_0=R$ (Figure~\ref{fig.heuristics}, left). At some time $t$ just after the event, we have $\Theta_t=\{R*R,*R\,B\}$. But, due to exchangeability, the probability to obtain this set of configurations when sampling without replacement from a population whose proportion of unfit individuals at backward time $t$ is $x$ is
\[
\Pb_x(\Theta_t) = \Pb_x(R*R\cup *R\,B) = \Pb_x(R*R \cup R*B) = \Pb_x(R**) = \Pb_x(R),
\]
so by permuting and uniting the configurations, we get back to the probability of sampling a single unfit individual. So, in probability, the event is silent, and, in particular, we got rid of the `disturbing' $B$. 

For pairwise and binary branching, however, permuting and uniting alone do not suffice; rather, we must also combine configurations across different events. Consider, for example, the pairwise branchings in Figure~\ref{fig.heuristics} (middle and right). Here, the sample consists of the continuing line together with the checking line (middle) and the incoming line (right), respectively. Note that these events are different realizations of pairwise branching, but they happen at the same rate. Since our factorial moment duality is a property in expectation rather than a pathwise one, we may move one of the compatible configurations (namely $*RB$) from the rightmost picture over to the middle one and still get the same probability for the all-unfit ($RR$) sample. Namely, averaging over both realizations, we get
\[
   \Pb_x(\Theta_t) = \frac{1}{2} [\Pb_x (RRR) +  \Pb_x(*RB \cup R*R)] = \frac{1}{2} [\Pb_x (RRR) +  \Pb_x(R**)].
\]
This suggests that, instead of considering the two realizations of the AIG, we may consider a process that has, for this pairwise branching event, moved from $RR$ to $RRR$ or to $R**$, each with probability $1/2$.

The conjecture now is that this continues in further branching events, that is, that we may combine the configurations resulting from collections of AIG realizations into collections of configurations that consist of $R$s and $*$s only.
The goal of the remainder of the paper is to show that this is indeed the case; not only for ternary and pairwise, but also for binary branching and collisions (so for all events in the context of interactive neutral reproduction), and also in combination with all other transitions present in the model. This is by no means immediate -- after all, it implies that we may sample after every interactive branching event and then merge the outcomes across multiple AIGs, whereas in the true AIG, the sampling and merging occurs only once, at the end of the observation interval. 

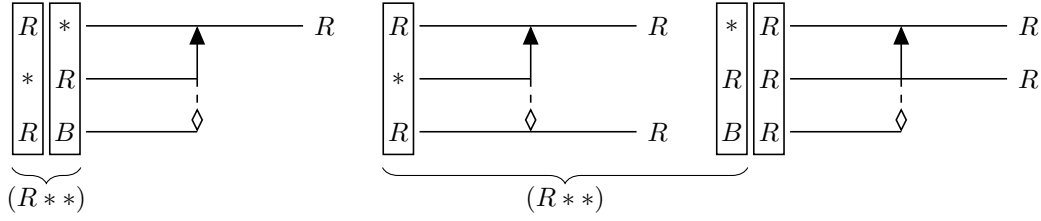
\begin{figure}[htbp]
\begin{tikzpicture}[scale=0.7]
\begin{scope}[shift={(0,0)}]

  \draw[semithick] (0.4,2) -- (4.5,2); 
  \draw[semithick] (0.4,1) -- (2.5,1); 
  \draw[semithick] (0.4,0) -- (2.5,0); 

  \draw[-{triangle 45[scale=5]}, semithick] (2.5,1) -- (2.5,2); 
  \draw[-{Diamond[open, scale=1.3]}, semithick, dashed] (2.5,1) -- (2.5,0); 

  \node at (4.9,2) { $R$};

  \node at (0,2) {$\ast$};
  \node at (-0.7,2)   { $R$};
  \node at (0,1){ $R$};
  \node at (-0.7,1)   {$\ast$};
  \node at (0,0){$B$};
  \node at (-0.7,0)   { $R$};

  \foreach \x in {0,-0.7} {
    \node[draw, semithick, minimum width=0.4cm, minimum height=2cm, anchor=center] at (\x,1) {};
  }

\end{scope}

\begin{scope}[shift={(6.3,0)}] 
  \draw[semithick] (0.4,2) -- (4.5,2); 
  \draw[semithick] (0.4,1) -- (2.5,1); 
  \draw[semithick] (0.4,0) -- (4.5,0); 

  \draw[-{triangle 45[scale=5]}, semithick] (2.5,1) -- (2.5,2);
  \draw[-{Diamond[open, scale=1.3]}, semithick, dashed] (2.5,1) -- (2.5,0);

  \node at (4.9,0) { $R$};
  \node at (4.9,2) { $R$};

  \node at (0,2)   { $R$};
  \node at (0,1)   { $\ast$};
  \node at (0,0)   { $R$};

  \foreach \x in {0} {
    \node[draw, semithick, minimum width=0.4cm, minimum height=2cm, anchor=center] at (\x,1) {};
  }

\end{scope}

\begin{scope}[shift={(13.3,0)}]  

  \draw[semithick] (0.4,2) -- (4.5,2); 
  \draw[semithick] (0.4,1) -- (4.5,1); 
  \draw[semithick] (0.4,0) -- (2.5,0); 

  \draw[-{triangle 45[scale=5]}, semithick] (2.5,1) -- (2.5,2);
  \draw[-{Diamond[open, scale=1.3]}, semithick, dashed] (2.5,1) -- (2.5,0);

  \node at (4.9,1) { $R$};
  \node at (4.9,2) { $R$};

  \node at (-0.7,2) {$\ast$};
  \node at (0,2)   { $R$};
  \node at (-0.7,1){ $R$};
  \node at (0,1)   { $R$};
  \node at (-0.7,0){$B$};
  \node at (0,0)   { $R$};

  \foreach \x in {0,-0.7} {
    \node[draw, semithick, minimum width=0.4cm, minimum height=2cm, anchor=center] at (\x,1) {};
  } 
\end{scope}

\draw[decorate,decoration={brace,amplitude=6pt,mirror},yshift=-5pt]
  (-1,-0.5) -- (0.3,-0.5) node[midway,yshift=-12pt] {$(R \ast \ast)$};

\draw[decorate,decoration={brace,amplitude=6pt,mirror},yshift=-5pt]
  (6.0,-0.5) -- (12.9,-0.5) node[midway,yshift=-12pt] {$(R \ast \ast)$};

\end{tikzpicture}
\caption{Transforming the configuration process into the Frankenstein process.}
\label{fig.heuristics}
\end{figure}

Because these combined configurations come from different worlds (realizations) that would not normally interact, we call the resulting construction the \emph{Frankenstein process}, denoted by $(\Phi_t)_{t \geq 0}$. Crucially, the Frankenstein process consists of a single cylinder and its $R$-counting process has the same distribution as $(Z^{(N)}_t)_{t \geq 0}$, the factorial moment dual.

The main challenge will be to rigorously link the configuration process to the Frankenstein process. We will show that the probability of sampling from the Frankenstein process equals that of sampling from the configuration process at all times: $ \Pb_x(\Phi_t)=\Pb_x(\Theta_t)$ for all $t \geq 0$. To this end, we introduce the \emph{messy process} $(\Psi_t)_{t \geq 0}$, which keeps track of the reshuffling of the configurations between realizations as needed for the combination into the Frankenstein process, while keeping the configurations separate. The messy process can then be connected to the Frankenstein process through exchangeability and to the configuration process through a recursive argument.

\subsection{The quasi-AIG}
This section aims to profit from randomness by considering a process similar to the AIG, but in which some information is removed from $\sP(\gs_0)$. Recall that the interactive event $\ine{\alpha}{\beta}{\gamma}$ involves the continuing line $\alpha$, the incoming line $\beta\neq \alpha$, and the checking line $\gamma$ (since the lines need not be distinct, a single line may play more than one role). We now remove the information about the roles of $\beta$ and $\gamma$ from the AIG. This allows us to couple different realizations of the AIG that differ only in the roles assigned to the lines in interactive neutral events. Such a coupling opens the way to mixing information from different AIGs, as suggested by the heuristics presented in the previous section.
\begin{definition}[quasi-AIG]\label{def.quasiAIG}
Consider the following set of symbols:
\begin{equation*}
    M^- \coloneqq \Big\{
        \narrowab,\ \sarrowab,\ \carrowcbasetb,\ \circa,\ \timesa
        : \alpha, \beta, \gamma \in [N],\beta \neq \alpha, \gamma\leq \beta,\, \{\alpha\} \neq \Bs\subseteq [N]
    \Big\}.
\end{equation*}
Define a function $\rho \colon M \to M^-$ by setting $\rho(m)=m$ for all $m \in M$, except for elements of the form $\ine{\alpha}{\beta}{\gamma}$, for which we set
\begin{equation}\label{eq.rho_map}
    \rho(\ine{\alpha}{\beta}{\gamma})\coloneqq 
    \begin{cases}
     \ineminus{\alpha}{\beta_0}{\gamma_0} \text{ with } \beta_0 = \beta \wedge \gamma \text{ and }
\gamma_0 = \beta \vee \gamma  \quad & \text{for } \gamma \notin \{\alpha, \beta\},  \\
    \ineminus{\alpha}{\beta}{\beta}  & \text{for } \gamma \in \{ \alpha,  \beta\}.
    \end{cases}
\end{equation}

The Poisson point process $\sP^{-}$ is obtained from $\sP$ via the mapping
$(t,m) \mapsto (t,\rho(m))$. For a given sample $\gs_0\subseteq[N]$, the thinning
$\sP^{-}(\gs_0)$ is defined based on $\sP^{-}$ in complete analogy with the construction
of $\sP(\gs_0)$ based on $\sP$. The process $\Gs^{-}(\gs_0)$ is defined analogously to $\Gs(\gs_0)$ from Def.~\ref{def.AIG}. The quasi-AIG is then defined as
\[
\As^{-}(\gs_0) \coloneqq \bigl(\Gs^{-}(\gs_0), \sP^{-}(\gs_0)\bigr).
\]
The quantities $\sT^{-}$ and $\As^{-}_{[0,t]}(\gs_0)$ are then analogous to
$\sT$ and $\As_{[0,t]}(\gs_0)$.
\end{definition}
If we are given a realization of the quasi-AIG, then, in order to recover an AIG,
we must reintroduce the information associated with each interactive neutral event
$\ineminus{\alpha}{\beta}{\gamma}$, $\beta \leq \gamma$, concerning the roles of $\beta$ and
$\gamma$ in an interactive branching event in the AIG. Due to the mapping \eqref{eq.rho_map}, this information was lost; in particular, $\beta$ and $\gamma$ have different meanings in $\ine{\alpha}{\beta}{\gamma}$ and $\ineminus{\alpha}{\beta}{\gamma}$. To this end, we make the following identification. 
\begin{equation}\label{eq.interactive.arrow}
\begin{split}
\text{For } \;   \beta < \gamma, \qquad &(\ineminus{\alpha}{\beta}{\gamma},\downarrow)\equiv \ine{\alpha}{\beta}{\gamma} 
\quad\text{and}\quad
(\ineminus{\alpha}{\beta}{\gamma},\uparrow)\equiv \ine{\alpha}{\gamma}{\beta}, \\[3mm]
\text{and \, for } \; \gamma=\beta, \qquad &
(\ineminus{\alpha}{\beta}{\beta},\downarrow)\equiv \ine{\alpha}{\beta}{\beta}
\quad\text{and}\quad
(\ineminus{\alpha}{\beta}{\beta},\uparrow)\equiv \ine{\alpha}{\beta}{\alpha}.
\end{split}
\end{equation}
Note that the cases $\ineminus{\alpha}{\beta}{\alpha}$ and $\ineminus{\alpha}{\alpha}{\beta}$ do not occur, because, due to our global assumption $\beta \neq \alpha$ in $\ine{\alpha}{\beta}{\gamma}$, the events in question are not in the image of the mapping \eqref{eq.rho_map}.
With the above identification, given $\bs{v}=(v_i)_{i\in\Nb}$, $v_i \in \{\downarrow, \uparrow\}$, we can associate to each realization
$a^{-}$ of a quasi-AIG a realization $a \coloneqq (a^{-},\bs{v})$ of an AIG by using, at the $k$th interactive neutral event $\ineminus{\alpha}{\beta}{\gamma}$ in
$a^{-}$, the interactive event $(\ineminus{\alpha}{\beta}{\gamma},v_k)$ in $a$.

\begin{definition}[The AIG on paths]\label{def.aigonpath}
For any $\bs{v}=(v_i)_{i\in\Nb}\in\{\downarrow,\uparrow\}^\Nb$ and $\gs_0\subseteq[N]$, we refer to $(\As_{[0,t]}^{-}(\gs_0),\bs{v})_{t\geq 0}$ as the AIG process on the path $\bs{v}$ (associated with the sample $\gs_0$). 
\end{definition}
Let $V=(V_i)_{i\in\Nb}$ denote an iid sequence
of Bernoulli-type random variables taking values in $\{\downarrow,\uparrow\}$, each with probability $1/2$. Since, by \eqref{eq.Poisson_measure} and \eqref{eq.rho_map}, the events $\ineminus{\alpha}{\beta}{\gamma}$ ($\alpha \neq \beta \leq \gamma \neq \alpha$) have rates $\kappa/N^2$ each in the quasi-AIG, \eqref{eq.interactive.arrow} implies  the rate $\kappa/(2 N^2)$ for each of the events $\ine{\alpha}{\beta}{\gamma}$ ($\beta \neq \alpha$) in the quasi-AIG on the random path $V$, in agreement with the Poisson measure driving the (original) AIG. So
\begin{equation}\label{random_quasi}
\bigl(\As^{-}_{[0,t]}(\gs_0),V\bigr)_{t\geq 0}
\overset{(d)}{=}
\bigl(\As_{[0,t]}(\gs_0)\bigr)_{t\geq 0},
\end{equation}
and, as a consequence, we obtain the following averaging principle, which relates our main quantity of interest, the set of compatible configurations with respect to the AIG, to the corresponding sets of the quasi-AIG.

\begin{lemma}[Averaging principle 1]\label{av1B}
Let $\gs_0 \subseteq [N]$ and let $S\in \varSigma_{\gs_0}$ be a root-type configuration. Let $M_t$ denote the number of interactive neutral events in the quasi-AIG $\As^{-}_{[0,t]}(\gs_0)$ and let $\Gamma^{(i)}\in \{r,b\}^N$ be chosen uniformly across all configurations with exactly $i$ entries equal to $\mathrm{r}$ and $N-i$ equal to $\mathrm{b}$, as in Lemma \ref{pre-duality}. Then, for any $t \ge 0$, we have
\begin{align*}
\Pb\big(\Gamma^{(i)}\in\Ss(\As_{[0,t]}(\gs_0), S)\big)
=\Eb\bigg[
\frac{1}{2^{M_t}}
\sum_{\bs v \in \{\downarrow,\uparrow\}^{M_t}}
\Pb\Big(
\Gamma^{(i)} \in \Ss\big((\As^{-}_{[0,t]}(\gs_0), \bs v ), S\big)
\,\mid\, \sP^-
\Big)
\bigg].
\end{align*}
\end{lemma}
\begin{proof}
The proof follows by applying~\eqref{random_quasi}, conditioning on
$\As^{-}_{[0,t]}(\gs_0)$, and using the tower property of the conditional expectation.
\end{proof}

Let us now translate the above ideas to the level of cylinders. If we consider transitions of the quasi-AIG, as we will do, their effects on the compatible cylinders is the same as in Section~\ref{subsec:tracing_back}, except for interactive neutral events. In this case, we need to decide the roles of the lines $\beta$ and $\gamma$, which we do, once more, with the help of $v\in\{\downarrow,\uparrow\}$. For this, we consider $v$ and an interactive neutral event $\ineminus{\alpha}{\beta}{\gamma}$ involving the line $\alpha\in I$ and the lines $\beta\in [\gamma],\gamma \in [N]$ and define 
\[
C_I \mapsto \onintminus{\alpha}{\beta}{\gamma}^{v}(C_I)= \big\{\onintminus{\alpha}{\beta}{\gamma}^{(v,i)}(C_I)\big\}_{i\in\{1,2\}}
\]
with
\[
    \onintminus{\alpha}{\beta}{\gamma}^{(\downarrow, i)}(C_I) \coloneqq \begin{cases} \onint{\alpha}{\beta}{\gamma}^{i}(C_I), & \beta<\gamma  \\
    \onint{\alpha}{\beta}{\alpha}^{i}(C_I), & \beta=\gamma \end{cases}
    \quad \text{and} \quad 
     \onintminus{\alpha}{\beta}{\gamma}^{(\uparrow, i)}(C_I) \coloneqq \begin{cases} \onint{\alpha}{\gamma}{\beta}^i(C_I), & \beta < \gamma \\
    \onint{\alpha}{\beta}{\beta}^i(C_I), & \beta=\gamma \end{cases}
\]
for $i \in \{1,2\}$,
thus combining the operators $\textrm{I}$ from \eqref{eq.nint1} -- \eqref{eq.nint3} with the identifications in \eqref{eq.interactive.arrow}. 
Note that, for $\gamma \notin \{\alpha, \beta\}$, the operator $\onint{\alpha}{\gamma}{\beta}^i$ affects the same lines as the operator $\onint{\alpha}{\beta}{\gamma}^i$, but interchanges the roles of the incoming and the checking line; we can explicitly write $\onint{\alpha}{\gamma}{\beta}^i(C_I)=\sigma(\onint{\alpha}{\beta}{\gamma}^i(\sigma(C_I))$, where $\sigma$ is the permutation that interchanges the positions $\beta$ and $\gamma$. With these transitions in mind, we will now, analogously to Def.~\ref{def.ordered_full_info}, define a configuration process on a path $\bs{v}$.

\begin{definition}[The configuration process on paths]\label{def.ordered}
Let $\bs{v}$ and $\sP^-$ be as in Def.~\ref{def.quasiAIG}. Let $(T_k)_{k\geq 1}$ denote the times of interactive neutral events, ordered increasingly. Let $\gs_0\subseteq[N]$ and fix a root-type configuration $S\in\varSigma_{\gs_0}^R$. The \emph{configuration process on the path $\bs{v}$}, denoted by
$(\overline{\Theta}_t(S,\bs{v},\sP^-))_{t\geq 0}$, is defined recursively as follows. The process starts at
$\overline{\Theta}_0(S,\bs{v}) = S$.
Suppose that at some time the state of the process is a finite collection of cylinders
$\{C^{(i)}\}_{i=1}^r$.
\begin{enumerate}
\item At the times of the events $\coal{\alpha}{\beta}$, $\timesa$, $\circa$, and $\ftw{\alpha}{\Bs}$, the configuration process on the path $\bs{v}$ experiences the same transitions as the configuration process from Def.~\ref{def.ordered_full_info}. 
\item If the $k$th interactive neutral event is $\ineminus{\alpha}{\beta}{\gamma}$, then at time $T_k$, every cylinder $C^{(i)}$, $i \in [r]$, performs the transition
$$C^{(i)}\mapsto\onintminus{\alpha}{\beta}{\gamma}^{\,v_k}(C^{(i)}),$$
and the new state of the process is $\bigcup_{i=1 }^r \onintminus{\alpha}{\beta}{\gamma}^{\,v_k}(C^{(i)})$.
\end{enumerate}
By construction, and in analogy with Eq.~\eqref{eq.compatible_union}, we have
\begin{equation*}
    \Ss\big((\As_{[0,t]}^{-}(\gs_0),\bs{v}), S\big) = \bigcup_{C\in\overline{\Theta}_t(S,\bs{v}, \sP^-)}C.
\end{equation*}
\end{definition}
The averaging principle $1$ from Lemma~\ref{av1B} yields
\begin{equation}\label{av1C}
\Pb\big(\Gamma^{(i)}\in \Theta_t(S, \sP)\big) =\Eb \bigg[\frac{1}{2^{M_t}} \sum_{\bs v \in \{\downarrow,\uparrow\}^{M_t}} \Pb\big(\Gamma^{(i)}\in \overline{\Theta}_t(S,\bs{v}, \sP^-)\mid \sP^-\big)\bigg].   
\end{equation}
Notably, \eqref{av1C} can also be proved without Lemma~\ref{av1B}, since for 
$V=(V_i)_{i\in\Nb}$ as above Eq.~\eqref{random_quasi}, Def.~\ref{def.ordered} implies
$$({\Theta}_t(S))_{t\geq 0}\overset{(d)}{=}(\overline{\Theta}_t(V,S))_{t\geq 0}.$$

As for the configuration process, it follows from the above construction that the configuration process on paths admits an equivalent description as an autonomous continuous-time Markov chain on finite collections of cylinders. Its transition rates are directly inherited from the corresponding transitions in the quasi-AIG.

\subsection{The Frankenstein matching}
This section starts with a few definitions, which will pave the way to the formal introduction of the Frankenstein process. From here onward, we will assume $S\in \varSigma^R$
and often work with $S=R_I$. Let $\overline{\varSigma}$ be the set of all finite disjoint unions of elements of $\varSigma$
(with the empty set understood as the empty union of cylinders).
\begin{definition}[Additive and $R$-preserving transitions]
A transition $\pi \colon \overline{\varSigma} \to \overline{\varSigma}$ is called 
\emph{additive} if $\pi(\varnothing) = \varnothing$ and, for any pair $A,B$ of disjoint subsets of $\varSigma$,
\[
\pi(A \cupdot B) = \pi(A) \cupdot \pi(B).
\]
An additive transition $\pi$ is called \emph{$R$-preserving} if
\[
\pi(\varSigma^R) \subseteq \varSigma^R.
\]
\end{definition}

Let us now focus, for a moment, on the case in which interactive neutral events are absent. Figures~\ref{fig.dotted},~\ref{fig.mutations}, and~\ref{fig.FTW_new} illustrate the local effects of coalescence, mutation, and FTW selection on type configurations along the lines of the AIG. In particular, one observes that these transitions are $R$-preserving and single-valued when restricted to $\Sigma^R$ (i.e. one $R$-cylinder is mapped to exactly one $R$-cylinder). It is part of the folklore that this property is closely related to factorial and moment dualities, as will be made explicit in the next result. 
\begin{corollary}\label{cor.R-preserving}
Let $\mathcal{A}_{[0,t]}(\gs_0)$ be an AIG with root set $\gs_0$ and leaf set $\ell$, and without interactive neutral events. Let $\Gamma^{(i)}$ be as in Lemma~\ref{pre-duality}. Then $\mathcal{S}(\mathcal{A}_{[0,t]}(\gs_0),R_{\gs_0})=\{C\}$ for some $C\in\Sigma_{\ell}^R$, and
    \begin{equation*}
       \Pb\big(\Gamma^{(i)}\in \mathcal{S}(\mathcal{A}_{[0,t]}(\gs_0),R_{\gs_0}) \big)=\Pb\big(\Gamma^{(i)}=C \big)=\frac{i^{\underline{n_R(C)}}}{N^{\underline{n_R(C)}}}.
    \end{equation*}
\end{corollary}
\begin{proof}
As mentioned above, the fact that coalescence, mutation, and FTW selection are single-valued (when restricted to $\Sigma^R$) and $R$-preserving follows directly from the description of the underlying transitions in Section~\ref{subsec:tracing_back} (see Figures~\ref{fig.dotted}, \ref{fig.mutations}, and~\ref{fig.FTW_new}). By iterating this fact over the events present in $\mathcal{A}_{[0,t]}(\gs_0)$, we infer that $\mathcal{S}(\mathcal{A}_{[0,t]}(\gs_0),R_{\gs_0})$ consists of a single $R$-cylinder. This establishes the first claim. The second statement follows from the first one as a direct consequence of Lemma~\ref{prop_average1}.
\end{proof}

It is straightforward to verify that, in the absence of interactive neutral events, the $R$-counting process $(n_R(C_t))_{t \ge 0}$ defined by $\Ss(\As_t(\gs_0), R_{\gs_0}) = \{C_t\}$ is itself Markovian. The above result, together with Lemma~\ref{pre-duality}, implies that $X^{(N)}$ and the process $(n_R(C_t))_{t \ge 0}$ are in factorial duality. Moreover, one readily checks that the transition rates of $(n_R(C_t))_{t \ge 0}$ coincide with those of the process $Z^{(N)}$ in Def.~\ref{def.factorial_dual}. 
But as we have seen in the previous section (specifically in Fig.~\ref{fig.Rcilinders}), interactive neutral events are typically not $R$-preserving. That is, even if the configuration process starts from an $R$-cylinder,
after the first interactive neutral event the collection of compatible
cylinders may consist of more than one cylinder and include cylinders with $B$'s. This highlights why the factorial duality established in Theorem~\ref{thm.factorial_duality} is particularly striking, and that the process $(n_R(C_t))_{t \ge 0}$ is not well-defined, and thus, cannot be the factorial moment dual. We will now introduce the Frankenstein property as a weaker version of the $R$-preserving property and analyze why this property is already sufficient for the factorial moment duality.

\begin{definition}[Frankenstein property]
Let $\mathrm{P}_n$ be the set of permutations of $[n]$. A set $\mathcal{C}=\{C^{(1)},\ldots, C^{(m)}\} \subseteq \varSigma_n$ of cylinders has the \emph{Frankenstein property} if either $\mathcal{C}\subseteq \varSigma_n^R$ or there are distinct indices $\tilde{l}_1,\hat{l}_1, \tilde{l}_2,\hat{l}_2,\ldots,\tilde{l}_p,\hat{l}_p\in[m]$ such that: 
\begin{itemize}
    \item For $j \in [m]\setminus\{\tilde{l}_1,\hat{l}_1,\, \tilde{l}_2,\hat{l}_2,...,\tilde{l}_p,\hat{l}_p\}$, we have $C^{(j)} \in \varSigma_n^R$.
    \item For every $i\in[p]$, there is $\sigma_i\in \mathrm{P}_n$ such that $\big(\sigma_i(C^{(\tilde{l}_i)})\big) \cup C^{(\hat{l}_i)})\in \varSigma_n^R $.
\end{itemize}
\end{definition}
In other words, given a collection of cylinders satisfying the Frankenstein property, after permuting coordinates in some of them, one can pair each non--$R$-cylinder with an $R$-cylinder so that each pair forms an $R$-cylinder.

\begin{remark}\label{remark_R_preserving}
Note that, by definition, the outcomes of $R$-preserving transitions have the Frankenstein property.
\end{remark}
 
\begin{example}
The set $E=\{(R\ast B),( R \,R \, \ast)\}$ has the Frankenstein property, because $(R\ast B)\cup (R\ast R)=(R\ast\ast)\in\varSigma^R$.
\end{example}

The next proposition lies at the core of our heuristic argument.

\begin{proposition}[Frankenstein matching]\label{prop_frankenstein_matching}
Let $I \subseteq [N]$. For any $C \in \varSigma_I^R$, any $\alpha,\beta,\gamma \in [N]$ with $\alpha \neq \beta \leq \gamma \neq \alpha$, there exist two permutations 
$\sigma^\downarrow,\sigma^\uparrow$ on $\tilde{I} = I \cup \{\beta,\gamma\}$ and a bijection
$f:\{\downarrow,\uparrow\} \to \{\downarrow,\uparrow\}$ such that, denoting 
\[
\onintminus{\alpha}{\beta}{\gamma}^{(v,1)}(C)=\longtilde{C}_{\tilde{I}}^{(v,1)}, \qquad
   \onintminus{\alpha}{\beta}{\gamma}^{(v,2)}(C)= \longtilde{C}_{\tilde{I}}^{(v,2)},
   \qquad v \in \{\downarrow,\uparrow\},
\]
the sets
\[
\widehat{C}^{\ddownarrow}_{\tilde{I}}
   \coloneqq \longtilde{C}^{(\downarrow,1)}_{\tilde{I}}
   \cup \sigma^\downarrow\!\bigl(\longtilde{C}^{f(\downarrow),2)}_{\tilde{I}}\bigr),
\qquad
\widehat{C}^{\uuparrow}_{\tilde{I}}
   \coloneqq \longtilde{C}^{(\uparrow,1)}_{\tilde{I}}
   \cup \sigma^\uparrow\!\bigl(\longtilde{C}^{(f(\uparrow),2)}_{\tilde{I}}\bigr),
\]
are both $R$-cylinders, and we refer to them as \emph{Frankenstein cylinders}. Moreover, these two cylinders contain the same number of $R$’s as $C$, except in the following cases:

\begin{enumerate}
\item $\beta < \gamma$ and 
$C_D\in \{(R, \, R, \, \ast), (R, \ast, R)\}$ with $D$ from Section~\ref{subsec:tracing_back}. In this case, one of the two cylinders contains one more $R$ than $C$, and the other one contains one less.
\item $\beta = \gamma $ and $C_D = (R,R)$. In this case, one cylinder contains one $R$ less than $C$, while the other contains the same number.
\end{enumerate}
\end{proposition}
\begin{remark}
Although the statement of Prop.~\ref{prop_frankenstein_matching} does not provide an explicit form of the permutations $\sigma^\uparrow$, $\sigma^\downarrow$ and the bijection $f$, the proof is constructive and yields a concrete choice that satisfies all required properties. For the remainder of the paper, the specific choice is irrelevant, and we therefore fix $f$, $\sigma^\downarrow$, and $\sigma^\uparrow$ to be one such admissible triple (for instance, the one produced in the proof of the proposition). Note that these objects depend on $C$, $\alpha$, $\beta$, and $\gamma$; however, in order to avoid overburdening the notation, we do not make this dependence explicit.
\end{remark}
\begin{proof}

We begin by providing a suitable choice for the triple $(f, \sigma^\downarrow, \sigma^\uparrow)$ for fixed $v \in \{\downarrow, \uparrow\}$ and a given $R$-cylinder $C_I$. For this, let $v^c$ be the complement of $v\in \{\downarrow, \uparrow\}$, i.e. $\downarrow^c=\uparrow$ and vice versa. Define the function $f$ as
\begin{equation}\label{eq.virtual_pos}
    f(v)\coloneqq
    \begin{cases}
     v^c &\text{ for } C_\alpha=R \text{ and } n_R(C_D)= \lvert D \rvert-1,\\
     v &\text{ otherwise,}
    \end{cases} 
\end{equation}
and the permutations $\sigma^\downarrow$ and $\sigma^\uparrow$ by 
\begin{equation}\label{eq.permutation}
 \sigma^{v}\coloneqq 
    \begin{cases}
        (\alpha \, \beta ) \quad &\text{ for } \beta <  \gamma , \quad v=\downarrow, \quad \text{ and } \quad \longtilde{C}_{\tilde{I}}^{(v,1)} \cup \longtilde{C}_{\tilde{I}}^{(v,2)} \neq C_{\tilde{I}}^\ast, \\[10pt]
        (\alpha \gamma) \quad &\text{ for } \beta < \gamma, \quad v=\uparrow,\quad \text{and } \quad \longtilde{C}_{\tilde{I}}^{(v,1)} \cup \longtilde{C}_{\tilde{I}}^{(v,2)} \neq C_{\tilde{I}}^\ast, \\[10pt]
         \, \mathrm{id} &\text{ otherwise.}
    \end{cases}
\end{equation}
We now need to verify that, for any $R$-cylinder $C_I$, the sets $\widehat{C}^{\ddownarrow}_{\tilde{I}}$ and $\widehat{C}^{\uuparrow}_{\tilde{I}}$ are also $R$-cylinders and that their $R$-counts change as proposed. Tables \ref{tab.Frankenstein_matching} and \ref{tab.Frankenstein_matching_beta} capture the construction of $\widehat{C}^{\ddownarrow}_{\tilde{I}}$ and $\widehat{C}^{\uuparrow}_{\tilde{I}}$ for all $R$-cylinders, and thus yield the claim.

\begin{table}[htbp]
    \centering
    \caption{Construction of $\widehat{C}^{\ddownarrow}_{\tilde{I}}$ and $\widehat{C}^{\uuparrow}_{\tilde{I}}$ for $C_I\in \varSigma^R$ in an $\mathrm{I}_{\alpha, (\beta, \gamma)}$ event with $\alpha=1, \beta=2,$ and $ \gamma=3$.
    The Frankenstein cylinders $\widehat{C}^{\ddownarrow}_{\tilde{I}}$ and $\widehat{C}^{\uuparrow}_{\tilde{I}}$ are obtained by the pair of cylinders of matching color. If this pair shares a column, 
    this means $f(v)=v$. Filled boxes indicate that this cylinder undergoes a permutation (of positions $\alpha, \beta$ for \textcolor{mypurple}{purple} cylinders; $\alpha, \gamma$ for \textcolor{mygreen}{green} cylinders). Black boxes indicate a change in the number of $R$ individuals during the Frankenstein matching.}
    \label{tab.Frankenstein_matching}
        \label{tab:better_Frankenstein_R}
        \vspace{1em}
        \begin{tabular}{c|c|c|c|c}
            $C_D=(C_1,C_2^\ast,C_3^\ast)$ & $\longtilde{C}_{\tilde{I}}^{(\downarrow,i)}, i \in \{1,2\}$ & $\longtilde{C}_{\tilde{I}}^{(\uparrow,i)}, i \in \{1,2\}$ & $\textcolor{mypurple}{\widehat{C}_{\tilde{I}}^{\ddownarrow}}$ & $\textcolor{mygreen}{\widehat{C}_{\tilde{I}}^{\uuparrow}}$  \\[2pt]
            \hline\rule{0pt}{3ex} 
            $(\ast,\ast,\ast)$ & $\textcolor{mypurple}{(\ast, \ast, R)}, \textcolor{mypurple}{(\ast, \ast, B)}$ & $\textcolor{mygreen}{(\ast, R, \ast)}, \textcolor{mygreen}{(\ast, B, \ast)}$ & $\textcolor{mypurple}{(\ast, \ast, \ast)}$ & $\textcolor{mygreen}{(\ast, \ast, \ast)}$ \\[4pt]
            $(\ast,R,\ast)$ & $\textcolor{mypurple}{(\ast, R, R)}, \textcolor{mypurple}{(\ast, R, B)}$ & $\textcolor{mygreen}{(\ast, R, \ast)}, \textcolor{mygreen}{\varnothing}$ & $\textcolor{mypurple}{(\ast, R, \ast)}$ & $\textcolor{mygreen}{(\ast, R, \ast)}$ \\[4pt]
            $(\ast,\ast,R)$ & $\textcolor{mypurple}{(\ast, \ast, R)}, \textcolor{mypurple}{\varnothing}$ & $\textcolor{mygreen}{(\ast, R, R)}, \textcolor{mygreen}{(\ast, B, R)}$ & $\textcolor{mypurple}{(\ast, \ast, R)}$ & $\textcolor{mygreen}{(\ast, \ast, R)}$\\[4pt]
            $(R,R,R)$ & $\textcolor{mypurple}{(R, R, R)}, \textcolor{mypurple}{\varnothing}$ & $\textcolor{mygreen}{(R, R, R)}, \textcolor{mygreen}{\varnothing}$ & $\textcolor{mypurple}{(R, R, R)}$ & $\textcolor{mygreen}{(R, R, R)}$\\[4pt]
            $(R,\ast,\ast)$ & $\textcolor{mypurple}{(R, \ast, R)}, \fcolorbox{mypurple}{mypurple!15!white}{$\textcolor{mypurple}{(\ast, R, B)}$}$ & $\textcolor{mygreen}{(R, R, \ast)}, \fcolorbox{mygreen}{mygreen!15!white}{$\textcolor{mygreen}{(\ast, B, R)}$}$ & $\textcolor{mypurple}{(R, \ast, \ast)}$ & $\textcolor{mygreen}{(R, \ast, \ast)}$ \\[4pt]
            $(R,R,\ast)$ & $\textcolor{mypurple}{(R, R, R)}, \fcolorbox{mygreen}{mygreen!15!white}{$\textcolor{mygreen}{(\ast, R, B)}$}$ & $\textcolor{mygreen}{(R, R, \ast)}, \fcolorbox{mypurple}{mypurple!15!white}{$\textcolor{mypurple}{\varnothing}$}$ & \fcolorbox{Black}{white}{$\textcolor{mypurple}{(R, R, R)}$} & \fcolorbox{Black}{white}{$\textcolor{mygreen}{(\ast, R, \ast)}$} \\[4pt]
            $(R,\ast,R)$ & $\textcolor{mypurple}{(R, \ast, R)}, \textcolor{mygreen}{\varnothing}$ & $\textcolor{mygreen}{(R, R, R)}, \fcolorbox{mypurple}{mypurple!15!white}{$\textcolor{mypurple}{(\ast, B, R)}$}$ & \fcolorbox{Black}{white}{$\textcolor{mypurple}{(\ast, \ast, R)}$} & \fcolorbox{Black}{white}{$\textcolor{mygreen}{(R, R, R)}$}
        \end{tabular}
\end{table}

\begin{table}[htbp]
    \centering
        \caption{As Table \ref{tab.Frankenstein_matching}, but for $\alpha=1$ and $\beta=\gamma=2$.}
        \label{tab.Frankenstein_matching_beta}
        \begin{tabular}{c|c|c|c|c}
            $C_D=(C_1,C_2^\ast)$ & $\longtilde{C}_{\tilde{I}}^{(\downarrow,i)}, i \in \{1,2\}$ & $\longtilde{C}_{\tilde{I}}^{(\uparrow,i)}, i \in \{1,2\}$ & $\textcolor{mypurple}{\widehat{C}_{\tilde{I}}^{\ddownarrow}}$ & $\textcolor{mygreen}{\widehat{C}_{\tilde{I}}^{\uuparrow}}$ \\[2pt]
            \hline\rule{0pt}{3ex} 
            $(\ast, \ast)$ & $\textcolor{mypurple}{(\ast, B)}, \textcolor{mypurple}{(\ast, R)}$ & $\textcolor{mygreen}{(B, \ast)}, \textcolor{mygreen}{(R, \ast)}$ & $\textcolor{mypurple}{(\ast, \ast)}$ & $\textcolor{mygreen}{(\ast, \ast)}$ \\[4pt]
            $(\ast, R)$ & $\textcolor{mypurple}{\varnothing}, \textcolor{mypurple}{(\ast, R)}$ & $\textcolor{mygreen}{(B,R)}, \textcolor{mygreen}{(R, R)}$ & $\textcolor{mypurple}{(\ast, R)}$ & $\textcolor{mygreen}{(\ast, R)}$ \\[4pt] 
            $(R, \ast)$ & $\textcolor{mypurple}{\varnothing}, \textcolor{mygreen}{(R, R)}$ & $\textcolor{mygreen}{(B, R)}, \textcolor{mypurple}{(R, \ast)}$ & $\textcolor{mypurple}{(R, \ast)}$ & $\textcolor{mygreen}{(\ast, R)}$ \\[4pt]
            $(R,R)$  & $\textcolor{mypurple}{\varnothing}, \textcolor{mypurple}{(R, R)}$ & $\textcolor{mygreen}{(B,R)}, \textcolor{mygreen}{(R, R)}$ & {$\textcolor{mypurple}{(R, R)}$} & \fcolorbox{Black}{white}{$\textcolor{mygreen}{(\ast, R)}$}
        \end{tabular}
\end{table}
\end{proof}

\begin{remark}\label{Frankenstein_step}
Note that, in all cases, two cylinders from
\[
\bigl\{ \longtilde{C}_{\tilde{I}}^{(\downarrow,1)}, \longtilde{C}_{\tilde{I}}^{(\downarrow,2)},
       \longtilde{C}_{\tilde{I}}^{(\uparrow,1)},   \longtilde{C}_{\tilde{I}}^{(\uparrow,2)} \bigr\}
\]
are used to construct the cylinder $\widehat{C}^{\ddownarrow}_{\tilde{I}}$, while the remaining two are used to
construct $\widehat{C}^{\uuparrow}_{\tilde{I}}$. To make this precise, let $\iota$ be a bijection on $\{\downarrow,\uparrow,\ddownarrow,\uuparrow\}$ defined by
\[
\iota(\downarrow)=\,\ddownarrow,\quad \iota(\uparrow)=\,\uuparrow,\quad
\iota(\ddownarrow)=\,\downarrow,\quad \iota(\uuparrow)=\,\uparrow.
\]
Then
$\longtilde{C}_{\tilde{I}}^{(v,1)}$ always contributes to the construction of
$\widehat{C}^{\iota(v)}_{\tilde{I}}$, whereas $\longtilde{C}_{\tilde{I}}^{(v,2)}$ contributes to
$\widehat{C}^{\iota(f(v))}_{\tilde{I}}$. 
Moreover, note that the condition $ {\longtilde{C}_{\tilde{I}}}^{(v,1)} \cup {\longtilde{C}_{\tilde{I}}}^{(v,2)} = C_{\tilde{I}}$ does not imply that these cylinders are paired for the Frankenstein matching, as the case $C_I=(R \ast R)$ shows. 
\end{remark}

As a consequence, we get the following corollary.
\begin{corollary}
Let $\aigt^{-}(\gs_0)$ be a quasi-AIG and let $\overline{\Theta}$ be its configuration process on a path from Def.~\ref{def.ordered}. Fix $S \in \varSigma_{\gs_0}^R$ and let $T_1$ be the first interactive neutral event in the quasi-AIG. Then, the set 
\begin{equation*}
\overline{\Theta}_{T_1}\big(\downarrow,S\big)\cup \overline{\Theta}_{T_1}\big(\uparrow,S\big)
\end{equation*} 
has the Frankenstein property.
\end{corollary}
\begin{proof}
The statement follows by combining Cor.~\ref{cor.R-preserving} with Remark \ref{remark_R_preserving} and the Frankenstein matching from Prop.~\ref{prop_frankenstein_matching}.
\end{proof}

Based on the cylinders $\widehat{C}^{\ddownarrow}$ and $\widehat{C}^{\uuparrow}$, we will now introduce, in analogy to Def.~\ref{def.ordered}, the Frankenstein process on path $\bs w \in \{\ddownarrow, \uuparrow\}^\Nb$. As a next step, we conclude that the $R$-counting process of its randomized version is well defined and coincides with the factorial moment dual from Def.~\ref{def.factorial_dual}. 
We will then close the gap between the configuration process on paths and the Frankenstein process on paths via an intermediate process lying halfway between the two worlds. Finally, we link their randomized versions and obtain the factorial moment duality. 

\subsection{The Frankenstein process}
We now define the Frankenstein process on paths, that is, a process on
$\varSigma^R$ obtained by applying the Frankenstein property from Prop.~\ref{prop_frankenstein_matching} to the set of
compatible cylinders with respect to an entirely unfit sample. In analogy
with the path $\bs{v}$, which determines the
roles of the lines, we introduce a new path $\boldsymbol{w}$, which indicates whether the Frankenstein process
moves to the cylinder $\widehat{C}^{\ddownarrow}$ (down) or to cylinder
$\widehat{C}^{\uuparrow}$ (up). As a consequence, and in contrast to the configuration process, 
all transitions of this process will be single-valued and $R$-preserving:
the process follows the downward or the upward branch, and there
is always exactly one cylinder compatible with the initial, entirely unfit
sample.

\begin{definition}[Frankenstein process on paths]\label{def.Frankenstein_path}
Let $\bs w \in \{\ddownarrow,\uuparrow\}^\Nb$ and let $\sP^-$ be as in
Def.~\ref{def.quasiAIG}. Let $(T_k)_{k \geq 1}$
denote the times of interactive neutral events in the quasi-AIG, ordered increasingly. Let
$\gs_0\subseteq[N]$ and fix $S=R_{\gs_0}$. The \emph{Frankenstein process on the path $\bs w$}, denoted by
$\bigl(\overline{\Phi}_t(R_{\gs_0},\bs w, \sP^-)\bigr)_{t \geq 0}$, takes values in $\varSigma^R$, starts at
$\overline{\Phi}_0(R_{\gs_0},\bs w, \sP^-) = R_{\gs_0}$, and evolves as follows. Suppose that at some time the process is in state $C\in\Sigma^R$.
\begin{enumerate}
\item At the times of a coalescence or a relocation event $\coal{\alpha}{\beta}$, the process performs the transition $$C\mapsto\ocoal{\alpha}{\beta}(C).$$

\item At the times of a deleterious mutation $\timesa$, the process performs the transition $$C\mapsto\omutr{\alpha}(C).$$

\item At the times of a beneficial mutation $\circa$, the process performs the transition $$C\mapsto\omutb{\alpha}(C)=\varnothing.$$

\item At the times of an FTW event $\ftw{\alpha}{\Bs}$, the process performs the transition 
\begin{equation*}
C\mapsto \overline{C}, 
\end{equation*} 
where $\overline C$ is the single cylinder contained in $\oftw{\alpha}{\Bs}(C)$.

\item At the time $T_k$ of the $k$th interactive neutral event, the process performs the transition
$$C\mapsto \widehat{C}^{w_k},$$
\end{enumerate}
where $\widehat{C}^{w_k}$ is as in Prop.~\ref{prop_frankenstein_matching}.
\end{definition}

Note that for coalescence, relocation, and mutations, the Frankenstein process on a path $\bs w$ behaves identically to the configuration process on a path when restricted to $S \in \varSigma^R$. At the same time, this restriction makes the FTW selection events $R$-preserving and single-valued, since they reduce to \eqref{eq.selR}; the cases \eqref{eq.selB} and \eqref{eq.selB1} are now void. The major point, however, is that, due to the Frankenstein matching, the interactive branching events are now also $R$-preserving and single-valued. As a consequence we obtain
\begin{equation}\label{eq.compatible_Frankenstein}
\Pb\big(\Gamma^{(i)}\in \Phi_t(S)\big)=\Eb\Big[\frac{i^{\underline{n_R(\Phi_t(S))}}}{N^{\underline{n_R(\Phi_t(S))}}}\Big].
\end{equation}

We now randomize the information carried by the paths by introducing a
random path \( W = (W_k)_{k \geq 1} \), which yields the following process.

\begin{definition}[Frankenstein process]\label{def.Frankenstein}
Let \(\gs_0 \in [N] \), and let \( W = (W_k)_{k \geq 1} \) be a sequence of iid
random variables taking values in \( \{\ddownarrow, \uuparrow\} \) with
\[
\mathbb{P}(W_k = \ddownarrow) = \mathbb{P}(W_k = \uuparrow) = \tfrac12 .
\]
We refer to the process \( (\overline{\Phi}_t(W))_{t \geq 0} \), started from
\( S\in \varSigma_{\gs_0}^R \), as the \emph{Frankenstein process}, and we denote it by
\( (\Phi_t)_{t \geq 0} \).
\end{definition}

The following result shows that the process tracking the number of \(R\)'s in the Frankenstein process coincides in distribution with the factorial dual of the process \(X^{(N)}\) of Def.~\ref{def.factorial_dual}. Recall that functions of Markov processes are not, in general Markov; but this one is, as we will now see.

\begin{proposition}\label{prop.R-counting_rates}
The $R$-counting process $n_R(\Phi)$ of the Frankenstein process is a continuous-time Markov chain with the same law as $Z^{(N)}$ from Def.~\ref{def.factorial_dual}. 
\end{proposition}
\begin{proof}
Interactive neutral events $\ineminus{\alpha}{\beta}{\gamma}$ decrease the number of $R$ lines in the current cylinder by one 
if $(R, R, \ast)$ is mapped to $\widehat{C}^{\uuparrow}$, 
if $(R, \ast, R)$ is mapped to $\widehat{C}^{\ddownarrow}$,
or if $(R,R)$ is mapped to $\widehat{C}^{\uuparrow}$, see Tables~\ref{tab.Frankenstein_matching} and \ref{tab.Frankenstein_matching_beta}. Every triple $(\alpha,\beta,\gamma)$ of lines with $\alpha \neq \beta \leq \gamma \neq \alpha$ experiences an $\ineminus{\alpha}{\beta}{\gamma}$ event at rate $\kappa/N^2$, as we saw in the argument leading to \eqref{random_quasi}. If there are $n$ lines of type $R$ in the current Frankenstein cylinder, there are $n (n-1) (N-n)/2 $ options for cylinders of type $(R,R,*)$ that may be affected by an $\ineminus{\alpha}{\beta}{\gamma}$ event with $\alpha \neq \gamma < \beta \neq \alpha$ 
(note that there are indeed $N-n$ lines of type $\ast$: the $\ast$s in the sample plus the lines outside the sample). The situation is the same for $(R,*,R)$ cylinders. Likewise, there are $n (n-1)$ possibilities for $(R,R)$ cylinders to experience $\ineminus{\alpha}{\beta}{\gamma}$ events with $\gamma = \beta \neq \alpha$. In any case, the probability for the event to yield $\widehat{C}^{\uuparrow}$ or $\widehat{C}^{\ddownarrow}$ as required is $1/2$. Altogether, therefore, $n$ decreases to $n-1$ due to interactive events at rate
\[
  \frac{\kappa}{2 N^2}  n(n-1)(N-n+1).
\]
Interactive neutral events increase the number $n$ of $R$ lines in the current cylinder by one 
if $(R, R, *)$ or $(R, *, R)$ is transformed to $(R, R, R)$, see Table~\ref{tab.Frankenstein_matching}; this happens at rate
\[
  \frac{\kappa}{2 N^2}  n(n-1)(N-n).
\]

The other events are those described in Section~\ref{subsec:tracing_back}, specialized to cylinders consisting of $R$s and $*$s only. 
So FTW events are initiated by every unfit line and add an additional number $k$ of lines (recall Fig.~\ref{fig.FTW_new}) at rate
\begin{equation}\label{eq.ftw.with_without}
   n \sum_{j=k}^N \tilde s_j q^n_{jk} =   n \sum_{m=k}^\infty s_m p_{mk}^n,
\end{equation}
where $\tilde s_j$ is as in \eqref{eq.stilde}, and 
\begin{equation}\label{eq.qnjk}
  q^n_{jk} = \frac{{N-n \choose k} {n \choose j-k}}{{N \choose j}}
\end{equation}
is the probability that $k$ lines are added to the sample of size $n$ when $j$ lines are sampled without replacement from the total of $N$ lines. The right-hand side of \eqref{eq.ftw.with_without} reflects that we can alternatively think in terms of the original sampling with replacement, with $p^n_{mk}$ from \eqref{eq.pnmj}. In the appendix, we verify explicitly that the two sums in \eqref{eq.ftw.with_without} are equal.

Non-interactive neutral events reduce $n$ by one at rate $r n (n-1)/2,$ deleterious mutation events do so at rate $nu \nu_1$, and beneficial mutation events make the configuration impossible, that is, send it to $\Delta$, at rate $nu \nu_0$ (recall Figs.~\ref{fig.dotted} and \ref{fig.mutations}.) 

Taking all events and rates together results in the law of $Z^{(N)}$.
Note, in particular, that the rates for all events depend on nothing but (the parameters and) $n$, the number of $R$'s in the current cylinder, whence the process is indeed Markov.
\end{proof}

Similar to the Averaging Principle 1 from Lemma~\ref{av1B}, which relates $\Ss(\aigt(\gs_0), S)$ to $\Ss(\qaigt{}(\gs_0), \bs{v}, S)$, there is a corresponding averaging principle for the other direction, namely from the Frankenstein process on paths to the `unconstrained' Frankenstein process just defined. This is the content of the following corollary.

\begin{corollary}[Average 2] \label{av2B}
Let $\gs_0 \subseteq [N]$ and let $S= R_{\gs_0}$ be a set of type
configurations of the lines in the sample $\gs_0$. Let $\overline{\Phi}(S, \bs{w}, \sP^-)$ be the Frankenstein process on paths from Def.~\ref{def.Frankenstein_path} and $\Phi(S)$ be the Frankenstein process from Def.~\ref{def.Frankenstein}. Let $M_t$ denote the number of
interactive neutral events in the quasi-AIG $\As^{-}_{[0,t]}(\gs_0)$ and let $\Gamma^{(i)}\in \{r,b\}^N$ chosen uniformly across all configurations with $i$ entries equal to $\mathrm{r}$ and $N-i$ entries equal to $\mathrm{b}$. Then, for any $t \ge 0$, we have 
\begin{equation}
     \Eb \Big[\frac{1}{2^{M_t}} \sum_{\bs w\in \{\ddownarrow,\uuparrow\}^{M_t}}  \Pb\big(\Gamma^{(i)}\in \overline{\Phi}_t(S,\bs w, \sP^-)\mid \sP^-\big)\Big]= \Pb\big(\Gamma^{(i)}\in \Phi_t(S)\big).
\end{equation}
\end{corollary}
\begin{proof}
Due to construction.
\end{proof}

By Lemma~\ref{pre-duality} and Prop.~\ref{prop.R-counting_rates} the factorial moment duality from Theorem~\ref{thm.factorial_duality} follows once we establish 
\begin{equation}\label{eq.goal}
\Pb\big(\Gamma^{(i)}\in \Theta_t(R_{\gs_0}, \sP)\big)=\Pb\big(\Gamma^{(i)}\in \Phi_t(R_{\gs_0}, \sP)\big).
\end{equation}
We emphasize that we do not claim the identity to hold on paths, as in general
\[
\Pb\big(\Gamma^{(i)}\in \overline{\Theta}_t(R_{\gs_0}, \bs v , \sP^-)\mid \sP^-\big)\neq\Pb\big(\Gamma^{(i)}\in\overline{\Phi}_t(R_{\gs_0}, \iota(\bs v),\sP^-)\mid \sP^-\big).
\]
In view of the two averaging principles (Lemma~\ref{av1B} and Cor.~\ref{av2B}), and for the function $\iota$ from Remark \ref{Frankenstein_step} applied component-wise to $\bs v$, \eqref{eq.goal} is therefore equivalent to 
\begin{align}\label{eq.remaining_av}
&\Eb \Big[\frac{1}{2^{M_t}} \sum_{\bs v \in \{\downarrow,\uparrow\}^{M_t}} \Pb\big(\Gamma^{(i)}\in \overline{\Theta}_t(R_{\gs_0},\bs v,\sP^-)\mid  \sP^-\big)\Big]
\nonumber \\ 
&\qquad =\Eb \Big[\frac{1}{2^{M_t}} \sum_{\iota(\bs v)\in \{\ddownarrow,\uuparrow\}^{M_t}}  \Pb\big(\Gamma^{(i)}\in \overline{\Phi}_t(R_{\gs_0},\bs \iota(v),\sP^-)\mid \sP^-\big)\Big],
\end{align}
which will be established in the subsequent sections.

\subsection{The messy process}
We now construct a process that links the configuration process and the Frankenstein process along their paths and will thus lead to \eqref{eq.remaining_av}. To this end, we fix an initial sample $R_{\gs_0}$ and a realization of the Poisson process $\sP^-$. These ingredients determine the family of configuration processes and Frankenstein processes simultaneously on all paths, that is
\begin{align*}
\bigl(\overline{\Theta}_t(R_{\gs_0},\boldsymbol{v},\sP^-)\bigr)_{t\geq 0} \text{ for all } \boldsymbol{v}\in\{\downarrow,\uparrow\}^{\Nb} \quad \text{and} \quad \bigl(\overline{\Phi}_t(R_{\gs_0}, \boldsymbol{w},\sP^-)\bigr)_{t\geq 0}
\text{ for all } \boldsymbol{w}\in\{\ddownarrow,\uuparrow\}^{\Nb}.
\end{align*}

In addition to the initial sample $R_{\gs_0}$, and the realization of the Poisson process $\sP^-$, we also fix a Frankenstein path $\boldsymbol{w}\in\{\ddownarrow,\uuparrow\}^{\Nb}$. Along these quantities, we construct a process evolving locally (that is, between interactive events) in the same way as the configuration process and consisting of a collection of cylinders whose union, after applying suitable permutations, yields the Frankenstein cylinder. These permutations will be deterministically determined by the initial sample $R_{\gs_0}$, the Frankenstein path $\bs w$ and the Poisson process $\sP^-$.
The process will be updated at every event in $\sP^-$ in the following way:
\begin{enumerate}
    \item Until the first interactive event in $\sP^-$, the process evolves as the configuration process.
    \item At $\ineminus{\alpha}{\beta}{\gamma}$ events in $\sP^-$, every cylinder is mapped to the pair of cylinders $I^{(v,1)}_{(\alpha, (\beta, \gamma))}(C)$ and $\, I^{(f(v),2)}_{(\alpha, (\beta, \gamma))}(C)$ from the Frankenstein matching of Prop. \ref{prop_frankenstein_matching}. Instead of permuting the cylinder $I^{(f(v),2)}_{(\alpha, (\beta, \gamma))}(C)$, a permutation is applied to $\sP^-$. 
    \item Between interactive events, every cylinder evolves according to the configuration process encoded by the permuted Poisson process from the previous step. 
\end{enumerate}
Before formally defining the messy process, we introduce some notation. Let $\sigma$ be a permutation of $[N]$. 
We denote by $\sigma(\sP^-)$ the Poisson process obtained by relabeling the lines in $[N]$ according to $\sigma$. More precisely, for a tuple $J=(J_1,...,J_n)$ with $ J_\ell\subseteq [N]$ for $1\leq \ell \leq n \leq N$, we have
\begin{equation}\label{eq.perm}
\begin{split}
& J \text{ is affected by an event in } \sP^- \\
\Longleftrightarrow \; \, & \sigma(J) \text{ is affected by the corresponding event in } \sigma(\sP^-).
\end{split}
\end{equation}

For $T>0$, we define the shift operator
\[
\zeta_T(\sP^-) \coloneqq \bigl(\sP^-_{T+t}\bigr)_{t\geq 0}.
\]
Clearly, both $\sigma(\sP^-)$ and $\zeta_T(\sP^-)$ have the same distribution as $\sP^-$.
From now on, let $\iota$ be the bijection from Remark \ref{Frankenstein_step} and note that $\iota$ is an involution, i.e.\ $\iota^2=\mathrm{id}$. By abuse of notation, we sometimes apply $\iota$ to vectors, in which case it is understood to act component-wise. 
Finally, for $k\in\Nb$, $\boldsymbol{w}\in\{\ddownarrow,\uuparrow\}^{\Nb}$ and $\boldsymbol{v}\in\{\downarrow,\uparrow\}^{\Nb}$, we set
\[
\bs w_k = (w_{k+i})_{i\in\Nb},
\qquad
\boldsymbol{v}_k = (v_{k+i})_{i\in\Nb}.
\]

Let $(T_k)_{k \ge 1}$ denote the increasing sequence of times at which interactive events occur in $\sP^-$. 
For every $k \ge 1$, we denote by $\sigma_k\coloneqq \sigma^{\iota(w_k)}$ and $f_k$ the permutation and the bijection involved in the Frankenstein matching during the transition 
\[
C=\overline{\Phi}_{T_k-}(R_{\gs_0},\boldsymbol{w}, \sP^- )
\;\longmapsto\;
\overline{\Phi}_{T_k}(R_{\gs_0},\boldsymbol{w}, \sP^-)=\widehat{C}^{w_k}.
\]
Recall that $\sigma_k$ and $f_k$ depend on $C$, on the coordinates involved in the corresponding interactive event, and on the set $I^{\iota(w_k)}_{(\alpha, (\beta, \gamma))}(C)$; for simplicity, we suppress this dependence in the notation.

The messy process $(\Psi_t(R_{\gs_0},\boldsymbol{w},  \sP^-))_{t \ge 0}$ is defined recursively as follows.

\textbf{Initialization: $t \in [0, T_1]$.} For $t \in [0, T_1)$, we trace back the $R$-cylinder compatible with the initial sample $R_{\gs_0}$ and $\sP^-$. That is,
\[
\Psi_t(R_{\gs_0},\boldsymbol{w},  \sP^-)
\coloneqq 
\overline{\Theta}_t(R_{\gs_0},\iota(\boldsymbol{w}),\sP^-)
=
\left\{
\overline{\Phi}_t(R_{\gs_0},\boldsymbol{w}, \sP^-)
\right\}.
\]
Recalling that all transitions before $T_1$ are single-valued and $R$-preserving, we denote by $C$ the single cylinder in $\Psi_{T_1 -}( R_{\gs_0}, \boldsymbol{w}, \sP^-)$.
We then define the messy process at time $T_1$ to be the collection of cylinders that contribute to the Frankenstein cylinder $\widehat{C}^{w_1}$ and set
\[
\Psi_{T_1}(R_{\gs_0}, \boldsymbol{w},\sP^-)
\coloneqq
\left\{
\mathrm{I}_{\alpha,(\beta,\gamma)}^{(\iota(w_1),1)}(C),
\mathrm{I}_{\alpha,(\beta,\gamma)}^{(f_1(\iota(w_1)),2)}(C)
\right\}.
\]
We associate the permutations from the Frankenstein matching to each of these two cylinders as follows:
\[
\sigma_{\smalllongtilde{C}} =
\begin{cases}
\mathrm{id}, & \text{if } 
\longtilde{C} = \mathrm{I}_{\alpha,(\beta,\gamma)}^{(\iota(w_1),1)}(C), \\[6pt]
\sigma_1, & \text{if } 
\longtilde{C} = \mathrm{I}_{\alpha,(\beta,\gamma)}^{(f_1(\iota(w_1)),2)}(C).
\end{cases}
\]
By construction, uniting the two permuted cylinders results in the cylinder from the Frankenstein process:
\[
\bigcup_{\smalllongtilde{C} \in \Psi_{T_1}(R_{\gs_0}, \boldsymbol{w}, \smallsP^-)}
\sigma_{\smalllongtilde{C}}(\longtilde{C})
=
\overline{\Phi}_{T_1}( R_{\gs_0},\boldsymbol{w}, \sP^-).
\]

\textbf{First part of the recursive step: $t \in [T_k, T_{k+1})$.} Assume that
$\Psi_t(R_{\gs_0},\boldsymbol{w},  \sP^-)$ has been constructed for all
$t \in [0,T_k]$, and that for each
$\longtilde{C} \in \Psi_{T_k}(R_{\gs_0}, \boldsymbol{w}, \sP^-)$
a permutation $\sigma_{\smalllongtilde{C}}$ has been specified so that
\begin{equation}\label{attk}
\bigcup_{\smalllongtilde{C} \in \Psi_{T_k}(R_{\gs_0}, \boldsymbol{w}, \smallsP^-)}
\sigma_{\smalllongtilde{C}}(\longtilde{C})
=
\overline{\Phi}_{T_k}(R_{\gs_0},\boldsymbol{w}, \sP^-).
\end{equation}
We will see later that the permutation $\sigma_{\smalllongtilde{C}}$ arises as the composition of the $k$ permutations from the Frankenstein matching; consequently, it is in general not a transposition and depends on the coordinates involved in the corresponding interactive events and on the sets $I^{\iota(w_k)}_{(\alpha, (\beta, \gamma))}(C)$.
For $t \in [T_k, T_{k+1})$, we set
\begin{equation}\label{ktokpo}
\Psi_t(R_{\gs_0},\boldsymbol{w},  \sP^-)
\coloneqq
\bigcup_{\smalllongtilde{C} \in \Psi_{T_k}(R_{\gs_0}, \boldsymbol{w}, \smallsP^-)}
\overline{\Theta}_{t - T_k}
\Big(\longtilde{C},\,
\iota(w_k),\,  \sigma^{-1}_{\smalllongtilde{C}}\big(\zeta_{T_k}(\sP^-)\big)
\Big).
\end{equation}
Here, for every cylinder present at time $T_k$, we trace back the set of compatible cylinders until just before the next interactive event. Since we want to link the messy process to the Frankenstein process on a path, we must account for the fact that, in the Frankenstein process, cylinders are permuted. In our construction, for every cylinder $\longtilde{C}\in \Psi_{T_k}(R_{\gs_0}, \boldsymbol{w}, \sP^-)$, the permutation $\sigma^{-1}_{\smalllongtilde{C}}$ applied to the time-shifted version of $\sP^-$ encodes this effect: line $i$ in the Frankenstein world corresponds to line $\sigma_{\smalllongtilde{C}}^{-1}(i)$ in the non-permuted world. Now, if, in the Frankenstein process, a tuple $J$ of lines is involved in an event of $\sP^-$, 
then this event affects the (permuted) cylinder $\sigma_{\smalllongtilde{C}}(\longtilde{C})$ at the lines $J$. Hence, to mimic this event, the non-permuted cylinder $\longtilde{C}$ needs to be affected at the lines $\sigma_{\smalllongtilde{C}}^{-1}(J)$, cf.~\eqref{eq.perm}. Note that the forward time shift $\zeta_{T_k}(\sP^-)$ compensates the backward time shift in $\bar \Theta_{t - T_k}$. Altogether, this leads us to consider, for every cylinder $\longtilde{C}$, the Poisson process $\sigma^{-1}_{\smalllongtilde{C}}(\zeta_{T_k}(\sP^-))$; this ends the motivation for \eqref{ktokpo}.
 
We now want to show that \eqref{attk} extends to the interval $[T_k,T_{k+1})$. To this end, we assign to every cylinder
\[
\overline{C}\in \overline{\Theta}_{t - T_k}\Big(\longtilde{C},\, \iota(w_k),\,\sigma^{-1}_{\smalllongtilde{C}}\big(\zeta_{T_{k}}(\sP^-)\big)\Big),
\qquad t\in[T_k,T_{k+1}),
\]
the permutation $\sigma_{\overline{C}} \coloneqq \sigma_{\smalllongtilde{C}}$. In the light of the above motivation, this implies for $t \in [T_k, T_{k+1})$ that
\begin{equation}\label{eq.permutation_Poisson}
\sigma_{\smalllongtilde{C}}\bigg(\overline{\Theta}_{t - T_k}
\Big(\longtilde{C},\,
\iota(w_k),\,  \sigma^{-1}_{\smalllongtilde{C}}\big(\zeta_{T_k}(\sP^-)\big)\Big)\bigg) = \overline{\Theta}_{t - T_k}
\big(\sigma_{\smalllongtilde{C}}(\longtilde{C}),\,
\iota(w_k),\, \zeta_{T_k}(\sP^-)\big)
\end{equation}
(this actually holds for any permutation $\sigma$ in the place of $\sigma_{\smalllongtilde{C}}$). This is the key to the extension we are aiming at.
\begin{lemma}\label{rec-lemma}
Assume that \eqref{attk} holds. Then, for all $t\in[T_k,T_{k+1})$, we have
\begin{equation*}
\bigcup_{\overline{C} \in \Psi_{t}( R_{\gs_0}, \boldsymbol{w}, \smallsP^-)}
\sigma_{\overline{C}}(\overline{C})
=
\overline{\Phi}_{t}( R_{\gs_0}, \boldsymbol{w}, \sP^-).
\end{equation*}
\end{lemma}
\begin{proof}
Using \eqref{ktokpo} in the first step, \eqref{eq.permutation_Poisson} in the second, the additivity of the configuration process in the third, \eqref{attk} in the fourth, and the definition of the Frankenstein process on a path in the last, we get
\begin{align*}
\bigcup_{\overline{C} \in \Psi_{t}( R_{\gs_0}, \boldsymbol{w}, \smallsP^-)}
\sigma_{\overline{C}}(\overline{C})&=\bigcup_{\smalllongtilde{C} \in \Psi_{T_k}(R_{\gs_0}, \boldsymbol{w}, \smallsP^-)} \sigma_{\smalllongtilde{C}}\bigg(\overline{\Theta}_{t - T_k}
\Big(\longtilde{C},\,
\iota(w_k),\,  \sigma^{-1}_{\smalllongtilde{C}}\big(\zeta_{T_k}(\sP^-)\big)
\Big)\bigg)\\
&=\bigcup_{\smalllongtilde{C} \in \Psi_{T_k}(R_{\gs_0}, \boldsymbol{w}, \smallsP^-)} \overline{\Theta}_{t - T_k}
\big(\sigma_{\smalllongtilde{C}}(\longtilde{C}),\,
\iota(w_k),\,  \zeta_{T_k}(\sP^-)
\big)\\
&=\overline{\Theta}_{t - T_k}
\Big(\bigcup_{\smalllongtilde{C} \in \Psi_{T_k}( R_{\gs_0}, \boldsymbol{w}, \smallsP^-)}\sigma_{\smalllongtilde{C}}(\longtilde{C}),\, 
\iota(w_k),\, \zeta_{T_k}(\sP^-)
\Big)\\
&=\overline{\Theta}_{t - T_k}
\big(\overline{\Phi}_{T_k}( R_{\gs_0}, \boldsymbol{w}, \sP^-) ,\,
\iota(w_k),\, \zeta_{T_k}(\sP^-)
\big)\\
&=\overline{\Phi}_{t}( R_{\gs_0}, \boldsymbol{w}, \sP^-).
\end{align*}
\end{proof}

\textbf{Second part of the recursive step: $t=T_{k+1}$.} Suppose now that at time $T_{k+1}$, $\sP^{-}$ experiences an event 
$\ineminus{\alpha}{\beta}{\gamma}$. Let us first describe the corresponding event in terms of the permuted Poisson process 
$\sigma^{-1}(\sP^{-})$; so the lines involved are
$\sigma^{-1}(\alpha)$, $\sigma^{-1}(\beta)$, and $\sigma^{-1}(\gamma)$. The resulting event in the Frankenstein world depends on the relative order of the permuted labels 
$\sigma^{-1}(\beta)$ and $\sigma^{-1}(\gamma)$. More precisely, if 
$
\sigma^{-1}(\beta) \leq \sigma^{-1}(\gamma),
$
the corresponding event is
$
\bigl(\ineminus{\sigma^{-1}(\alpha)}{\sigma^{-1}(\beta)}{\sigma^{-1}(\gamma)},\, w_{k+1}\bigr).
$
Conversely, if 
$
\sigma^{-1}(\beta) > \sigma^{-1}(\gamma),
$
the event is
$
\bigl(\ineminus{\sigma^{-1}(\alpha)}{\sigma^{-1}(\gamma)}{\sigma^{-1}(\beta)},\, w_{k+1}^{c}\bigr).
$
Define
\[
\beta(\sigma^{-1}) \coloneq \sigma^{-1}(\beta) \wedge \sigma^{-1}(\gamma),
\qquad
\gamma(\sigma^{-1}) \coloneq \sigma^{-1}(\beta) \vee \sigma^{-1}(\gamma),
\]
and
\[
w_{k+1}(\sigma^{-1}) \coloneq
\begin{cases}
w_{k+1}, & \text{if } \sigma^{-1}(\beta) \le \sigma^{-1}(\gamma),\\
w_{k+1}^{c}, & \text{otherwise}.
\end{cases}
\]
The event in the messy process at time $T_{k+1}$ can now be written compactly as
\[
\bigl(\ineminus{\sigma^{-1}(\alpha)}{\beta(\sigma^{-1})}{\gamma(\sigma^{-1})},\, w_{k+1}(\sigma^{-1})\bigr).
\]
According to the same principle as before, we define the messy process at time $T_{k+1}$ as 
\begin{equation}\label{tkvstkm}
\Psi_{T_{k+1}}( R_{\gs_0}, \boldsymbol{w}, \sP^-)
\coloneqq
\bigcup_{C \in \Psi_{T_{k+1}-}( R_{\gs_0}, \boldsymbol{w}, \smallsP^-)}
\!\!\!\!\!\!\!\!\left\{
\mathrm{I}_{\sigma^{-1}_C(\alpha),(\beta(\sigma^{-1}_C),\gamma(\sigma^{-1}_C))}^{(\iota(w_{k+1}(\sigma^{-1}_C)),1)}(C),\,
\mathrm{I}_{\sigma^{-1}_C(\alpha),(\beta(\sigma^{-1}_C),\gamma(\sigma^{-1}_C))}^{(f_{k+1}(\iota(w_{k+1}(\sigma^{-1}_C))),2)}(C)
\right\}.
\end{equation}
Finally, for every
$\longtilde{C} \in \Psi_{T_{k+1}}( R_{\gs_0}, \boldsymbol{w}, \sP^-)$,
define the associated permutation $\sigma_{\smalllongtilde{C}}$ by
\[
\sigma_{\smalllongtilde{C}} \coloneqq
\begin{cases}
\sigma_C, & \text{if }
\longtilde{C}
=
\mathrm{I}_{\sigma^{-1}_C(\alpha),(\beta(\sigma^{-1}_C),\gamma(\sigma^{-1}_C))}^{(\iota(w_{k+1}(\sigma^{-1}_C)),1)}(C)
\text{ for some }
C \in \Psi_{T_{k+1}-}( R_{\gs_0}, \boldsymbol{w}, \sP^-), \\[6pt]
\sigma_{k+1}\circ \sigma_C, & \text{if }
\longtilde{C}
=
\mathrm{I}_{\sigma^{-1}_C(\alpha),(\beta(\sigma^{-1}_C),\gamma(\sigma^{-1}_C))}^{(f_{k+1}(\iota(w_{k+1}(\sigma^{-1}_C))),2)}(C)
\text{ for some }
C \in \Psi_{T_{k+1}-}( R_{\gs_0}, \boldsymbol{w}, \sP^-).
\end{cases}
\]
The following result shows that these permutations have the desired effect.

\begin{lemma}\label{attkpo}
Assume that \eqref{attk} holds. Then
\[
\bigcup_{\smalllongtilde{C} \in \Psi_{T_{k+1}}( R_{\gs_0}, \boldsymbol{w}, \smallsP^-)}
\sigma_{\smalllongtilde{C}}(\longtilde{C})
=
\overline{\Phi}_{T_{k+1}}( R_{\gs_0}, \boldsymbol{w}, \sP^-).
\]
\end{lemma}
\begin{proof}
Indeed, we have
\begin{align*}
&\bigcup_{\smalllongtilde{C} \in \Psi_{T_{k+1}}( R_{\gs_0}, \boldsymbol{w}, \smallsP^-)}
\sigma_{\smalllongtilde{C}}(\longtilde{C})\\
&\quad=\bigcup_{{C} \in \Psi_{T_{k+1}-}( R_{\gs_0}, \boldsymbol{w}, \smallsP^-)}
\sigma_C\big(\mathrm{I}_{\sigma^{-1}_C(\alpha),(\beta(\sigma^{-1}_C),\gamma(\sigma^{-1}_C))}^{(\iota(w_{k+1}(\sigma^{-1}_C)),1)}(C)\big )\cup\sigma_{k+1}\Big(\sigma_C\big(\mathrm{I}_{\sigma^{-1}_C(\alpha),(\beta(\sigma^{-1}_C),\gamma(\sigma^{-1}_C))}^{(f_{k+1}(\iota(w_{k+1}(\sigma^{-1}_C))),2)}(C)\big)\!\Big)\\
&\quad=\bigcup_{{C} \in \Psi_{T_{k+1}-}( R_{\gs_0}, \boldsymbol{w}, \smallsP^-)}
\mathrm{I}_{\alpha,(\beta,\gamma)}^{(\iota(w_{k+1}),1)}\big(\sigma_C(C)\big)\cup\sigma_{k+1}\Big(\mathrm{I}_{\alpha,(\beta,\gamma)}^{(f_{k+1}(\iota(w_{k+1}),2))}\big(\sigma_C(C)\big)\!\Big)\\
&\quad=\mathrm{I}_{\alpha,(\beta,\gamma)}^{(\iota(w_{k+1}),1)}\Big(\bigcup_{{C} \in \Psi_{T_{k+1}-}( R_{\gs_0}, \boldsymbol{w}, \smallsP^-)}\!\!\!\!\!\!\!\!\sigma_C(C)\Big)\cup\sigma_{k+1}\bigg(\mathrm{I}_{\alpha,(\beta,\gamma)}^{(f_{k+1}(\iota(w_{k+1})),2)}\Big(\bigcup_{{C} \in \Psi_{T_{k+1}-}( R_{\gs_0}, \boldsymbol{w}, \smallsP^-)}\!\!\!\!\!\!\!\!\sigma_C(C)\Big)\!\!\bigg)\\
&\quad=\mathrm{I}_{\alpha,(\beta,\gamma)}^{(\iota(w_{k+1}),1)}\Big(\overline{\Phi}_{T_{k+1}-}( R_{\gs_0}, \boldsymbol{w}, \sP^-)\Big)\cup\sigma_{k+1}\bigg(\mathrm{I}_{\alpha,(\beta,\gamma)}^{(f_{k+1}(\iota(w_{k+1})),2)}\Big(\overline{\Phi}_{T_{k+1}-}( R_{\gs_0}, \boldsymbol{w}, \sP^-)\Big)\!\!\bigg)\\
&\quad=\overline{\Phi}_{T_{k+1}}( R_{\gs_0}, \boldsymbol{w}, \sP^-).
\end{align*}
The first identity follows by the definition of $\Psi_{T_{k+1}}( R_{\gs_0}, \boldsymbol{w}, \sP^-) $ in \eqref{tkvstkm} and the subsequent definition of the permutations $\sigma_{\smalllongtilde{C}}$. The second equality is another instance of the key feature \eqref{eq.permutation_Poisson}: the outcome of an interactive event in the permuted Poisson process $\sigma^{-1}_C(\sP^-)$ affecting the cylinder $C$ is identical to the outcome of an interactive event in $\sP^-$ affecting the permuted cylinder $\sigma_C(C)$. 
The third identity is once more the additivity of the interactive event. In the penultimate step, we used Lemma~\ref{rec-lemma}, and in the last one the definition of the Frankenstein process at interactive events. This completes the proof.
\end{proof}
Summarizing, the messy process $\Psi$ satisfies the following property.

\begin{proposition}[Assembling the Frankenstein process]\label{step1}
For every $t \ge 0$ and every 
$\boldsymbol{w} \in \{\ddownarrow, \uuparrow\}^{\Nb}$, 
there exists, for every 
$C \in \Psi_t(R_{\gs_0},\boldsymbol{w},  \sP^-)$, 
a permutation $\smalllongtilde{\sigma}_C$ such that
\[
\overline{\Phi}_t( R_{\gs_0}, \boldsymbol{w}, \sP^-)
=
\bigcup_{C \in \Psi_t(R_{\gs_0},\boldsymbol{w},  \smallsP^-)}
\smalllongtilde{\sigma}_C(C).
\]
\end{proposition}
\begin{proof}
We have already established the result on $[0,T_1]$. The general case follows by induction, using Lemmas~\ref{attk} and~\ref{attkpo}.
\end{proof}

Figure \ref{fig.messy} gives an example of the construction of the messy process.

\begin{figure}[hbtp]
\scalebox{0.65}{
\begin{tikzpicture}[
    grow=east,  
 edge from parent/.style={
        draw, 
        double, 
        double distance=1.5pt, 
        -{Implies[width=8pt, length=5pt]}, 
        shorten >=-2pt,
        shorten <=3pt
    },
    parent anchor=east,
    child anchor=west, 
    every node/.style={
        anchor=west,
        inner xsep=2pt,
        inner ysep=2pt,
        font=\sffamily,
        align=center
    },
    boxed/.style={line width=1.2pt, rounded corners=2pt, inner sep=0pt},
    level 1/.style={
        level distance=3cm,
        sibling distance=5cm,
        nodes={text width=5cm, minimum width=3.5cm} 
    },
    level 2/.style={
        level distance=4.5cm,
        sibling distance=2.5cm
    }
]
\node {$(RR\ast)$}
    child {
        node {
            \begin{tabular}{cc} $\onintminus{\alpha}{\beta}{\gamma}^{(\downarrow, 1)}(RR\ast)$  & $\onintminus{\alpha}{\beta}{\gamma}^{(\uparrow, 2)}(RR\ast)$  \\[5pt]
              $(R\,R\,R)$ & ${\color{white}{\ast \, \ast}}\varnothing{\color{white}{\ast \, \ast}}$ \\[5pt]
                $\textrm{id}$ & $\textrm{id}$ 
            \end{tabular}
        }
        child {
            node {
            \begin{tabular}{cccc}
            $\onintminus{\alpha}{\beta}{\gamma}^{(\downarrow, 1)}(RRR)$ & $\onintminus{\alpha}{\beta}{\gamma}^{(\downarrow, 2)}(RRR)$ & $\onintminus{\alpha}{\beta}{\gamma}^{(\downarrow, 1)}(\varnothing)$ &  $\quad \ \  \onintminus{\alpha}{\beta}{\gamma}^{(\downarrow, 2)}(\varnothing)$\\[10pt]
            $(RRRR)$ & $\varnothing$ & $\varnothing$ & $\quad  \ \  \varnothing$ \\ [5pt]$\mathrm{id}$&$\mathrm{id}$&$\mathrm{id}$&$\quad \ \  \mathrm{id}$
                \end{tabular}
            }
        edge from parent node [below, sloped] {$(RRRR)$}
        }
        child {
            node {
                \begin{tabular}{cccc}
           $\onintminus{\alpha}{\beta}{\gamma}^{(\uparrow, 1)}(RRR)$ & $\onintminus{\alpha}{\beta}{\gamma}^{(\uparrow, 2)}(RRR)$ & $\onintminus{\alpha}{\beta}{\gamma}^{(\uparrow, 1)}(\varnothing)$ & $\quad \ \ \onintminus{\alpha}{\beta}{\gamma}^{(\uparrow, 2)}(\varnothing)$\\[10pt]
          $(RR\!\ast \!R)$ & $(BR\!\ast\! R)$ & $\varnothing$ & $\quad \ \ \varnothing$  \\[5pt]$\mathrm{id}$&$\mathrm{id}$&$\mathrm{id}$&$\quad \ \ \mathrm{id}$
                \end{tabular}
            }
            edge from parent node [above, sloped] {$(\ast R\!\!\!\ast\!\!\!R)$}
        }
    edge from parent node [below, sloped] {$(RRR)$}
    }
    child {
        node {
            \begin{tabular}{cc}
             $\onintminus{\alpha}{\beta}{\gamma}^{(\uparrow, 1)}(RR\ast)$ & $\onintminus{\alpha}{\beta}{\gamma}^{(\downarrow, 2)}(RR\ast)$ \\[5pt]$(R\,R\,\ast)$ & 
            $(\ast \, R \, B)$ \\[5pt]
            ${\textrm{id}}$ & $(1,3)$
            \end{tabular}
        }
        child {
            node {
                \begin{tabular}{cccc}
                $\onintminus{\alpha}{\beta}{\gamma}^{(\downarrow, 1)}(RR\ast)$ & $\onintminus{\alpha}{\beta}{\gamma}^{(\downarrow, 2)}(RR\ast)$ & $\onintminus{\alpha}{\beta}{\gamma}^{(\downarrow, 1)}(\ast R B)$ & $\onintminus{\alpha}{\beta}{\gamma}^{(\downarrow, 2)}(\ast R B)$\\[10pt]
                $\,(RRR\ast)\,$ &  $\,(\ast R B \ast)\,$ & $(RRB\ast)$ & $\varnothing$ \\[5pt]
$\mathrm{id}$&$\mathrm{id}$&$\mathrm{id}\circ (1,3)$&$\mathrm{id}\circ(1,3)$
                \end{tabular}
            }
        edge from parent node [below, sloped] {$(\ast R\!\!\!\ast\!\!\!\ast)$}
        }
        child {
            node {
                \begin{tabular}{cccc}          $\onintminus{\alpha}{\beta}{\gamma}^{(\uparrow, 1)}(RR\ast)$ & $\onintminus{\alpha}{\beta}{\gamma}^{(\uparrow, 2)}(RR\ast)$ & $\onintminus{\alpha}{\beta}{\gamma}^{(\uparrow, 1)}(\ast R B)$ & $\onintminus{\alpha}{\beta}{\gamma}^{(\uparrow, 2)}(\ast R B)$
                \\[10pt]
                {$(RR\!\ast\!\ast)$} & $\varnothing$ & $(\ast RB \ast)$& $\varnothing$ \\[5pt]
$\mathrm{id}$&$\mathrm{id}$&$\mathrm{id}\circ (1,3)$&$\mathrm{id}\circ (1,3)$
                \end{tabular}
            }
        edge from parent node [above, sloped] {$(\ast R\!\!\!\ast\!\!\!\ast)$}
        }
    edge from parent node [above, sloped] {$(\ast R\ast)$}
    }; 
\end{tikzpicture}}
\caption{Construction of the messy process for two interactive branching events and all four paths $\bs w \in \{\ddownarrow, \uuparrow\}^2$. For simplicity, we assume that no events take place in $(T_1,T_2)$. The first event acts on $C=(C_1, C_2, C_4)=(RR\ast)$, with $\alpha=1, \beta=2, \gamma=4$. The resulting Frankenstein cylinders are depicted on the arrows. The messy process at time $T_1$ has an upper branch (representing $w_1=\uuparrow)$ and a lower branch (representing $w_1=\uuparrow)$, each consisting of the cylinders $\onintminus{\alpha}{\beta}{\gamma}^{\iota(w_1, 1)}(C)$ and $\onintminus{\alpha}{\beta}{\gamma}^{f(\iota(w_1)), 2)}(C)$ (first line), where the function $f$ depends on $C, \onintminus{\alpha}{\beta}{\gamma}^{\iota(w_1, 1)}(C)$ and $\onintminus{\alpha}{\beta}{\gamma}^{\iota(w_1, 2)}(C)$. The explicit cylinders are shown in the second line. The third line keeps track of the permutations needed for the Frankenstein matching, but without applying them to the cylinders. The second interactive event in $\sP^-$ acts on $C=(C_1, C_2, C_4)$ with $\alpha=1, \beta=2, \gamma=3$ and hence uses $C_3^\ast=\ast$. The first event of the time-shifted and permuted process $(1, 3)(\zeta_{T_1}(\sP^-))$ acts on $C=(C_1,C_2,C_4)$ with $\alpha=4$, $\beta=2$, $\gamma=3$. The function $f$ needed for the outcome of these events is determined by the outcome of the operators $\onintminus{\alpha}{\beta}{\gamma}^{\iota(w_2, 1)}$ and $\onintminus{\alpha}{\beta}{\gamma}^{\iota(w_2, 2)}$ acting on the Frankenstein cylinders that would have been established in the previous event.}
\label{fig.messy}
\end{figure}
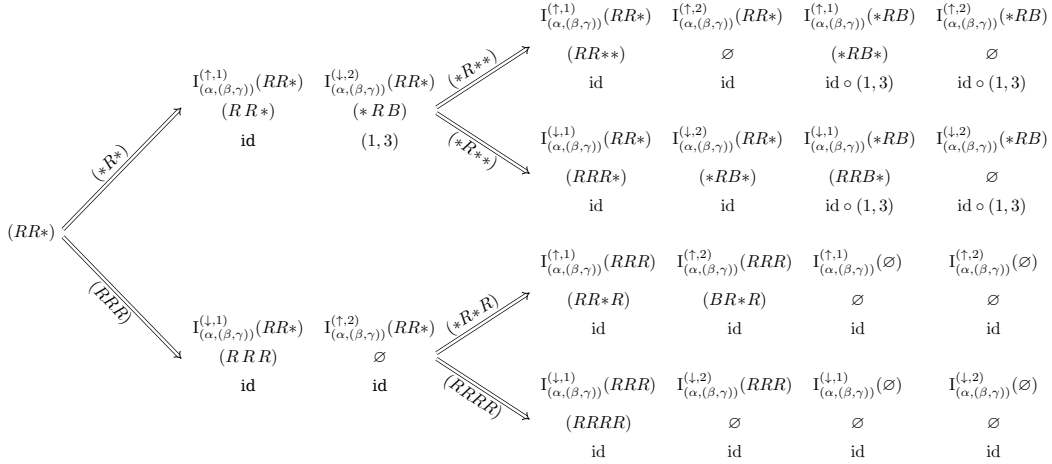

\subsection{Connecting all processes}
Now that we have established the messy process and its key features, it remains to connect it to the Frankenstein process and the configuration process on paths and thus show \eqref{eq.remaining_av}. The link from the Frankenstein process on the path $\bs w$ to the messy process on the same path follows by construction.

\begin{corollary}\label{connection_Frankenstein_messy}
For $\Gamma^{(i)}$ as in Lemma \ref{pre-duality}, we have
$$\Pb\big(\Gamma^{(i)} \in \overline{\Phi}_t( R_{\gs_0}, \boldsymbol{w}, \sP^-)\mid \sP^-\big)=\Pb\big(\Gamma^{(i)} \in\Psi_t(R_{\gs_0},\boldsymbol{w},  \sP^-)\mid \sP^-\big).$$
\end{corollary}
\begin{proof}
The statement follows directly from Prop.~\ref{step1} using exchangeability.
\end{proof}

It remains to relate the messy process to the configuration process. To this end, we interpolate between them as follows.
For $k\in\Nb_0$ and for $T_0=0$, define the process $(\Psi^{(k)}_t( R_{\gs_0}, \boldsymbol{w}, \sP^-))_{t\geq 0}$ by
\begin{equation}\label{def.psiti}
\Psi^{(k)}_t( R_{\gs_0}, \boldsymbol{w}, \sP^-)
=
\begin{cases}
\Psi_t(R_{\gs_0},\boldsymbol{w},  \sP^-), & \text{for } t < T_k,\\[6pt]
\displaystyle
\bigcup_{C \in \Psi_{T_k}( R_{\gs_0}, \boldsymbol{w}, \smallsP^-)}
\overline{\Theta}_{t-T_k}\bigl(C,\iota(\boldsymbol{w}_k), \zeta_{T_k}(\sP^-)\bigr),
& \text{for } t\geq T_k.
\end{cases}
\end{equation}

The desired connection is based on the subsequent result. 

\begin{proposition}[Interpolation]\label{interpolation}
For all $t\geq 0$ and $k\in\Nb_0$, we have
\begin{align*}
&\Eb\Bigg[\frac{1}{2^{M_t}}
\sum_{\boldsymbol{w}\in\{\ddownarrow,\uuparrow\}^{M_t}}
\Pb\bigl(\Gamma^{(i)}\in \Psi_t^{(k)}(R_{\gs_0}, \boldsymbol{w}, \sP^-)\mid \sP^-\bigr)
\Bigg]
\\
&\qquad\qquad\qquad=
\Eb\Bigg[\frac{1}{2^{M_t}}
\sum_{\boldsymbol{w}\in\{{\ddownarrow,\uuparrow\}^{M_t}}}
\Pb\bigl(\Gamma^{(i)}\in \Psi_t^{(k+1)}(R_{\gs_0}, \boldsymbol{w}, \sP^-)\mid \sP^-\bigr)
\Bigg].
\end{align*}
\end{proposition}
\begin{proof}
We decompose both sides of the identity according to the events $\{t<T_k\}$, $\{T_k\leq t<T_{k+1}\}$, and $\{T_{k+1}\leq t\}$, and verify the identity on each of these events.
By definition of $\Psi^{(k)}$ and $\Psi^{(k+1)}$, on the event $\{t<T_k\}$ we have, for all $\boldsymbol{w}$,
\[
\Psi^{(k)}_t(R_{\gs_0},\boldsymbol{w},\sP^-)
=
\Psi_t(R_{\gs_0},\boldsymbol{w},  \sP^-)= 
\Psi^{(k+1)}_t(R_{\gs_0}, \boldsymbol{w},\sP^-).
\]
We now turn to the second event. By construction and \eqref{ktokpo}, we have
\begin{align}\label{eq.second_event}
&\Eb\Big[\1_{\{T_k\leq t< T_{k+1}\}}\frac{1}{2^{M_t}}
\sum_{\boldsymbol{w}\in\{\ddownarrow,\uuparrow\}^{M_t}}
\Pb\big(\Gamma^{(i)}\in\Psi_t^{(k+1)}(R_{\gs_0}, \boldsymbol{w}, \sP^-)\mid \sP^-\big)\Big]\nonumber \\
&=\Eb\Big[\1_{\{T_k\leq t< T_{k+1}\}}\frac{1}{2^{k}}
\sum_{\boldsymbol{w}\in\{\ddownarrow,\uuparrow\}^{k}}
\Pb\big(\Gamma^{(i)}\in\Psi_t(R_{\gs_0},\boldsymbol{w},  \sP^-)\mid \sP^-\big)\Big]\nonumber \\
&=\Eb\Bigg[\frac{\1_{\{T_k\leq t< T_{k+1}\}}}{2^{k}} \!\!\!\!\!\!
\sum_{\boldsymbol{w}\in\{\ddownarrow,\uuparrow\}^{k}} \,\,
\sum_{\smalllongtilde{C}\in\Psi_{T_k}\!(R_{\gs_0}, \boldsymbol{w},\smallsP^-)}
\!\!\!\Pb\bigg(\Gamma^{(i)}\in\overline{\Theta}_{t-T_k}\Big(\longtilde{C},\iota(\boldsymbol{w}_k), \sigma^{-1}_{\smalllongtilde{C}}\big(\zeta_{T_k}(\sP^-)\big)\!\Big) \,\Big|\, \sP^-\bigg)\Bigg].
\end{align}
Since, for $\mathcal{F}_{T_k}$ as the natural filtration of $\sP^-$ up to time $T_k$, the invariance of $\sP^-$ under permutations gives us
\begin{align*}
 &\Eb\bigg[\frac{\1_{\{T_k\leq t< T_{k+1}\}}}{2^{k}}
\Pb\bigg(\Gamma^{(i)}\in\overline{\Theta}_{t-T_k}\Big( \longtilde{C},\iota(\boldsymbol{w}_k), \sigma^{-1}_{\smalllongtilde{C}}\big(\zeta_{T_k}(\sP^-)\big)\Big)\,\Big|\, \sP^-\bigg)
\,\Big|\, \mathcal{F}_{T_k}\bigg] \\
&=\Eb\bigg[\frac{\1_{\{T_k\leq t< T_{k+1}\}}}{2^{k}}
\Pb\Big(\Gamma^{(i)}\in\overline{\Theta}_{t-T_k}\big(\longtilde{C},\iota(\boldsymbol{w}_k), 
\zeta_{T_k}(\sP^-)\big)\mid \sP^-\Big)
\,\Big|\, \mathcal{F}_{T_k}\bigg],
\end{align*}
we can apply the tower property of conditional expectation twice to express \eqref{eq.second_event} as
\begin{align*}
&\Eb\bigg[\frac{\1_{\{T_k\leq t< T_{k+1}\}}}{2^{k}}
\sum_{\boldsymbol{w}\in\{\ddownarrow,\uuparrow\}^{k}}\,\,
\sum_{\smalllongtilde{C}\in\Psi_{T_k}(R_{\gs_0}, \boldsymbol{w},\smallsP^-)}
\Pb\Big(\Gamma^{(i)}\in\overline{\Theta}_{t-T_k}\big(\longtilde{C},\iota(\boldsymbol{w}_k), \zeta_{T_k}(\sP^-)\big)\mid \sP^-\Big)\bigg]\\
&=\Eb\bigg[\frac{\1_{\{T_k\leq t< T_{k+1}\}}}{2^{k}}
\sum_{\boldsymbol{w}\in\{\ddownarrow,\uuparrow\}^{k}}
\Pb\big(\Gamma^{(i)}\in\Psi^{(k)}_t(R_{\gs_0},\boldsymbol{w}, \sP^-)\mid \sP^-\big)\bigg],
\end{align*}
where we have used \eqref{def.psiti} in the last step.
This proves the desired identity on the event $\{T_k\leq t<T_{k+1}\}$.
It remains to prove it on the event $\{T_{k+1}\leq t\}$. This requires some notation. For $t>T_k$, we denote by $M_t^{k}$ the number of interactive events in $\sP^-$ in $(T_k, t]$. For a permutation $\sigma$, we write $\sP^{-}_{k,\sigma^{-1}}$ for the Poisson process that coincides with $\sigma^{-1}(\zeta_{T_k}(\sP^-))$ on $(0,T_{k+1}-T_k]$ and with $\zeta_{T_{k+1}}(\sP^-)$ on $(T_{k+1}-T_k,\infty)$. Moreover, for $\bs w ^{(1)}\in \{\ddownarrow, \uuparrow\}^k$, let $(\bs {w}^{(1)}, \bs w^{(2)})\in \{\ddownarrow, \uuparrow\}^{\Nb}$ be the vector whose first $k$ entries are determined by $\bs w^{(1)}$ and and those from entry $k+1$ onward by $\bs w^{(2)}$.
This allows us to separate interactive events in the interval $(T_k, T_{k+1}]$ from those in $(T_{k+1}, \infty)$: 
\begin{align}\label{stardec}
&\Eb\Big[\frac{\1_{\{t> T_{k+1}\}}}{2^{M_t}}
\sum_{\boldsymbol{w}\in\{\ddownarrow,\uuparrow\}^{M_t}}
\Pb\big(\Gamma^{(i)}\in\Psi_t^{(k+1)}(R_{\gs_0}, \boldsymbol{w}, \sP^-)\mid \sP^-\big)\Big]\nonumber\\
&=\frac{1}{2^{k}}\sum_{\boldsymbol{w^*}\in\{\ddownarrow,\uuparrow\}^{k}}\Eb\Big[
\sum_{\boldsymbol{w}\in\{\ddownarrow,\uuparrow\}^{M_t^k}}
\frac{\1_{\{t> T_{k+1}\}}}{2^{M_t^k}}\Pb\Big(\Gamma^{(i)}\in\Psi_t^{(k+1)}\big(R_{\gs_0},(\boldsymbol{w}^*,\boldsymbol{w}),  \sP^-\big)\mid \sP^-\Big)\Big].
\end{align}
For any $\boldsymbol{w}^*\in\{\ddownarrow,\uuparrow\}^{k}$, we have
\begin{align*}
&\Eb\Big[
\sum_{\boldsymbol{w}\in\{\ddownarrow,\uuparrow\}^{M_t^k}}
\frac{\1_{\{t> T_{k+1}\}}}{2^{M_t^k}}
\Pb\Big(\Gamma^{(i)}\in\Psi_t^{(k+1)}\big(R_{\gs_0},(\boldsymbol{w}^*,\boldsymbol{w}), \sP^-\big)\mid \sP^-\Big)\Big]\\
&=\Eb\Bigg[\sum_{C\in\Psi_{T_{k+1}-}(R_{\gs_0},\boldsymbol{w}^*,\smallsP^-)}
\Eb\bigg[\sum_{\boldsymbol{w}\in\{\ddownarrow,\uuparrow\}^{M_t^k}}
\frac{\1_{\{t> T_{k+1}\}}}{2^{M_t^k}}
\bigg.\Bigg.\\
&\qquad\qquad\times\sum_{j=1}^{2}\bigg.\Bigg.\Pb\Big(\Gamma^{(i)}\in\overline{\Theta}_{t-T_{k+1}}\big( \mathrm{I}_j (w_{k+1},C), \iota(\boldsymbol{w}),\zeta_{T_{k+1}}(\sP^-)\big)\mid \sP^-\Big)
\mid \mathcal{F}_{T_{k+1}-}\bigg]\Bigg].
\end{align*}
where $\mathrm{I}_1(w_{k+1},C)$ and $\mathrm{I}_2(w_{k+1},C)$ are the two cylinders arising from $C$ as in \eqref{tkvstkm}. Since we sum over both possible values of $w_{k+1}$, we may interchange $\mathrm{I}_2(w_{k+1},C)$ with  $\mathrm{I}_2(w_{k+1}^c,C)$ whenever $f_{i+1}(\iota(w_{k+1}))\neq \iota(w_{k+1})$. In this way, we will reverse the Frankenstein matching at time $T_{k+1}$ and come back to the transitions of the configuration process at time $T_{k+1}$. For fixed $C\in \Psi_{T_{k+1}-}(R_{\gs_0},\boldsymbol{w}^*,\sP^-)$, we obtain 
\begin{align*}
&\sum_{\boldsymbol{w}\in\{\ddownarrow,\uuparrow\}^{M_t^k}} \sum_{j=1}^{2} \Pb\Big(\Gamma^{(i)}\in\overline{\Theta}_{t-T_{k+1}}\big( \mathrm{I}_j (w_{k+1},C), \iota(\boldsymbol{w}),\zeta_{T_{k+1}}(\sP^-)\big)\mid \sP^-\Big)\nonumber\\
&\qquad =\sum_{\boldsymbol{w}\in\{\ddownarrow,\uuparrow\}^{M_t^k}} \Big [ \Pb\Big(\Gamma^{(i)}\in\overline{\Theta}_{_{t-T_{k+1}}}\big( \mathrm{I}_1 (w_{k+1},C), \iota(\boldsymbol{w}),\zeta_{T_{k+1}}(\sP^-)\big)\mid \sP^-\Big)\nonumber\\
& \qquad \qquad +\Pb\Big(\Gamma^{(i)}\in\overline{\Theta}_{t-T_{k+1}}\big( \mathrm{I}_2 (f(\iota(w_{k+1})),C), \iota(\boldsymbol{w}),\zeta_{T_{k+1}}(\sP^-)\big)\mid \sP^-\Big) \Big ].
\end{align*}
By \eqref{tkvstkm}, the cylinders $\mathrm{I}_1 (w_{k+1},C)$ and $\mathrm{I}_2 (f(\iota(w_{k+1})),C)$ arise as outcomes $\mathrm{I}^{(\iota(w_k),1)}$ and $\mathrm{I}^{(\iota(w_k),2)}$ of an interactive event at time $T_{k+1}$ of the configuration process. This concludes the reversion of the Frankenstein matching.
As a consequence, and since, on $(T_k,T_{k+1})$, the process $\Psi^{(k)}$ evolves according to the configuration process, there exists a $\hat{C}\in \Psi_{T_{k}}(R_{\gs_0},\boldsymbol{w}^*,\sP^-)$ such that the above expression may be stated in terms of a $T_k$-shifted configuration process as
\[
\sum_{\boldsymbol{w}\in\{\ddownarrow,\uuparrow\}^{M_t^k}}\Pb\Big(\Gamma^{(i)}\in\overline{\Theta}_{t-T_{k}}\big( \hat{C},\iota(\boldsymbol{w}), \sP^-_{k,\sigma^{-1}_C}\big)\mid \sP^-\Big).
\]
Overall we have
\begin{align*}
&\Eb\bigg[\sum_{\hat{C}\in\Psi_{T_{k}}(R_{\gs_0},\boldsymbol{w}^*,\smallsP^-)}
\sum_{\boldsymbol{w}\in\{\ddownarrow,\uuparrow\}^{M_t^k}}
\frac{\1_{\{t> T_{k}\}}}{2^{M_t^k}}\,
\Pb\Big(\Gamma^{(i)}\in\overline{\Theta}_{t-T_{k}}\big( \hat{C},\iota(\boldsymbol{w}), \sP^-_{k,\sigma^{-1}_C}\big)\mid \sP^-\Big)\bigg]
\\
&=\Eb\bigg[\sum_{\hat{C}\in\Psi_{T_{k}}(R_{\gs_0},\boldsymbol{w}^*,\smallsP^-)}
\sum_{\boldsymbol{w}\in\{\ddownarrow,\uuparrow\}^{M_t^k}}
\frac{\1_{\{t> T_{k}\}}}{2^{M_t^k}}\,
\Pb\Big(\Gamma^{(i)}\in\overline{\Theta}_{t-T_{k}}\big(\hat{C},\iota(\boldsymbol{w}),  \zeta_{T_k}(\sP^-)\big)\mid \sP^-\Big)\bigg],
\end{align*}
where, in the last identity, we used the invariance of the law of $\sP^-$ under permutations. Summing over $\boldsymbol{w}^*\in\{\ddownarrow,\uuparrow\}^{k}$, multiplying by $2^{-k}$, and using \eqref{stardec}, the preceding identities yield
\begin{align*}
&\Eb\bigg[\frac{\1_{\{t> T_{k+1}\}}}{2^{M_t}}
\sum_{\boldsymbol{w}\in\{\ddownarrow,\uuparrow\}^{M_t}}
\Pb\bigl(\Gamma^{(i)}\in\Psi_t^{(k+1)}(R_{\gs_0}, \boldsymbol{w}, \sP^-)\mid \sP^-\bigr)\bigg]\\
&=\Eb\bigg[\frac{\1_{\{t> T_{k+1}\}}}{2^{M_t}}
\sum_{\boldsymbol{w}\in\{\ddownarrow,\uuparrow\}^{M_t}}
\Pb\bigl(\Gamma^{(i)}\in\Psi_t^{(k)}(R_{\gs_0}, \boldsymbol{w}, \sP^-)\mid \sP^-\bigr)\bigg],
\end{align*}
which concludes the proof.
\end{proof}
Due to their construction, we can directly link the configuration process $\overline{\Theta}(R_{\gs_0},\bs v, \sP^-)$ to the process $\Psi^{(0)}(R_{\gs_0},\iota(\bs v),\sP^-)$. Iterating the interpolation principle yields the desired connection to the messy process (obtained as $\Psi^{(k)}(\cdot)$ for $k\to \infty$). Overall, we can state
\begin{align}\label{eq.connection_messy_config}
 &\Eb \Big[\frac{1}{2^{M_t}} \sum_{\bs v \in \{\downarrow,\uparrow\}^{M_t}} \Pb\big(\Gamma^{(i)} \in \overline{\Theta}_t(R_{\gs_0},\bs v,\sP^-)\mid  \sP^{-}\big)\Big]\nonumber\\
&\quad =
 \Eb \Big[\frac{1}{2^{M_t}} \sum_{\bs v \in \{\downarrow,\uparrow\}^{M_t}} \Pb\big(\Gamma^{(i)} \in \Psi^{(0)}_t(R_{\gs_0},\iota (\bs v),\sP^-)\mid  \sP^{-}\big)\Big] \nonumber \\
&\quad =  \Eb \Big[\frac{1}{2^{M_t}} \sum_{\bs v \in \{\downarrow,\uparrow\}^{M_t}} \lim_{k \to \infty}\Pb\big(\Gamma^{(i)} \in \Psi^{(k)}_t(R_{\gs_0},\iota (\bs v),\sP^-)\mid  \sP^{-}\big)\Big] \nonumber \nonumber \\
&\quad =
\Eb \Big[\frac{1}{2^{M_t}} \sum_{\bs v\in \{\downarrow,\uparrow\}^{M_t}}  \Pb\big(\Gamma^{(i)} \in {\Psi}_t(R_{\gs_0},\iota(\bs {v}),\sP^-)\mid \sP^{-}\big)\Big].  
\end{align}
We now have everything in hand to establish the factorial moment duality from Theorem \ref{thm.factorial_duality} from a genealogical perspective.

\begin{proof}[Proof of Theorem \ref{thm.factorial_duality}]
Eq.~\eqref{eq.connection_messy_config} and Cor.~\ref{connection_Frankenstein_messy} together yield
\begin{align*}
&\Eb \Big[\frac{1}{2^{M_t}} \sum_{\bs v \in \{\downarrow,\uparrow\}^{M_t}} \Pb\big(\Gamma^{(i)} \in \overline{\Theta}_t(R_{\gs_0},\bs v,\sP^-)\mid  \sP^{-}\big)\Big]\\
&\quad =
\Eb \Big[\frac{1}{2^{M_t}} \sum_{\bs v\in \{\downarrow,\uparrow\}^{M_t}}  \Pb\big(\Gamma^{(i)} \in \overline{\Phi}_t( R_{\gs_0},\iota(\bs {v}),\sP^-)\mid \sP^{-}\big)\Big].
\end{align*}
The analogous statement for the random sample $R_{g_n}$ is obtained by randomizing over all choices of $R_{\gs_0}$.
Using this together with the averaging principles from Eq.~\eqref{av1C} and Cor.~\ref{av2B} as well as Equations \eqref{eq.connection_config_compat} and \eqref{eq.compatible_Frankenstein} we conclude that
\[
\Pb\big(\Gamma^{(i)} \in \Ss(\As_{[0,t]}(g_n), R_{g_n})\big)=\Pb\big(\Gamma^{(i)} \in \Theta_t(R_{g_n})\big)=\Pb\big(\Gamma^{(i)} \in \Phi_t(R_{g_n})\big)=\Eb\Big[\frac{i^{\underline{n_R(\Phi_t(R_{g_n}))}}}{N^{\underline{n_R(\Phi_t(R_{g_n}))}}}\Big].
\] 
By Lemma~\ref{pre-duality}, we thus recover the factorial moment duality
\begin{equation*}
    \Eb\Big[ \frac{\big(X^{(N)}_t\big)^{\! \vphantom{(N)} \underline{n}}}{N^{\underline{n}}}\, \Big  \vert \, X^{(N)}_0=i\Big] = \Eb\Big[ \frac{i^{\underline{n_R(\Phi_t)}}}{N^{\underline{n_R(\Phi_t)}}}\,\Big  \vert \,n_R(\Phi_0)=n  \Big].
\end{equation*} 
\end{proof}
As before, the left-hand side of the duality describes the probability to obtain only unfit individuals when sampling $n$ individuals without replacement from the population at time $t$ that initially had $i$ unfit individuals. The right-hand side tells us that we can as well start with a sample of $n$ unfit individuals at forward time $t$, run the Frankenstein process backward in time, count the number of $R$s at backward time $t$, and sample accordingly from the initial population. 

\section{Recovering the moment duality}\label{sec:moment_duality}
This section proves the diffusion limits of the Moran model and the $R$-counting process of the Frankenstein process, which together imply the moment duality from Cor. \ref{corollary_moment_duality}.

\subsection{Diffusion scaling limit of the Moran model}\label{subsec.difflimit_X}
Let $X^{(N)}_t$ be the number of unfit individuals in the Moran model with interactive neutral reproduction, FTW selection, and mutation parametrized by $r, \kappa, s_m^{(N)}$ and $u^{(N)}$ at time $t\geq 0$. Then $X^{(N)}=(X^{(N)}_t)_{t\geq 0}$ is a birth-death process on $[N]_0$. In line with the model description in Section~\ref{sec:MoMo}, the respective generator $\mathcal{L}_{X^{(N)}}$ acts on functions $f:[N]_0\to \Rb$ and is given by
\begin{equation}\label{eq.gen_Momo}
    \mathcal{L}_{X^{(N)}} =\mathcal{L}_{X^{(N)}}^{\text{n}}+\mathcal{L}_{X^{(N)}}^{u}+\sum_{m=1}^\infty\mathcal{L}_{X^{(N)}}^{s_m},
\end{equation}
where the effects of (interactive and non-interactive) neutral events are captured by
\begin{equation}\label{eq.gen_n}
    \mathcal{L}_{X^{(N)}}^{\text{n}}\big((f(n)\big)= \lambda_n^{(N)} \big(f(n+1)-f(n)\big) + \eta_n^{(N)}\big(f(n-1)-f(n)\big)
\end{equation}
with
\begin{equation}\label{eq.lambda_eta}
\begin{aligned} 
    &\lambda_n ^{(N)}\coloneqq \frac{\r}{2} \, n \frac{N-n}{N} + \frac{\kappa}{2} n \frac{N-n}{N} \frac{N-n}{N}, \\ 
    &\eta_n^{(N)} \coloneqq \frac{\r}{2} (N-n) \frac{n}{N} +  \frac{\kappa}{2} (N-n) \frac{N-n}{N} \frac{n}{N}=\lambda_n ^{(N)} .  
\end{aligned}
\end{equation}
Similarly, for the effects of mutation 
$(\mathcal{L}_{X^{(N)}}^{u})$ and FTW selection $(\mathcal{L}_{X^{(N)}}^{s_m})$, we have
\begin{align*} 
 \mathcal{L}_{X^{(N)}}^{u} \big(f(n)\big) &\coloneqq u^{(N)} \, \nu_{1} (N-n)\big( f(n+1)-f(n)\big) + u^{(N)} \, \nu_{0} n \big( f(n-1)-f(n) \big), \\
\mathcal{L}_{X^{(N)}}^{s_m} \big(f(n)\big) &\coloneqq
s_m^{(N)} \, n \Big( 1-\Big(\frac{n}{N}\Big)^m\Big) \big(f(n-1)-f(n)\big)  
 \end{align*}
with the convention that $f(-1)\coloneqq 0 \eqqcolon f(N+1)$. The latter generators are standard for a Moran model with mutation and FTW selection, see \cite{BEH23} for more details.

We will now establish that the diffusion limit of our Moran model is given by a Wright--Fisher-type SDE. While the drift term captures the effects of FTW selection and mutation, the diffusion term encodes both types of neutral reproduction events. 

\begin{proposition}\label{prop.diffusion_limit}
    Fix $r\geq 0$, $\kappa \geq 0$, and $N \in \Nb$. Assume that $Nu^{(N)}\to \theta\geq 0$, $N s^{(N)}_m \to \sigma_m\in \Rb_+$, and $\sum_{m=1}^\infty N s_m^{(N)} \to \sum_{m=1}^\infty \sigma_m$, with $\sum_{m=1}^\infty m\,\sigma_m<\infty$. Suppose that $X_0^{(N)}/N \to X_0$ in distribution as $N \to \infty$. Then $\Big( X^{(N)}_{Nt}/N\Big)_{t \geq 0}$ converges in distribution to the solution $(X_t)_{t \geq 0}$ of the SDE
    \begin{equation*}
        \dd X_t = \Big(-\sum_{m=1}^{\infty}\big(\sigma_m X_t(1-X_t^m)\big) +\theta \, \nu_0 (1-X_t)-\theta \, \nu_1 X_t \Big) \dd t + \sqrt{X_t(1-X_t)\big(r+\kappa(1-X_t)\big)} \dd B_t,
    \end{equation*}
    with initial value $X_0=x_0$ and $B$ standard Brownian motion.
\end{proposition}
Note that, for $r=1-\kappa$ and $\theta=\sigma_m\equiv0$, the SDE reduces to \eqref{eq.SDE}. 
\begin{proof}
As argued in \cite[Proof of Theorem 2.1]{ACD25} by using \cite[Theorems 8.2.1 and 1.6.1]{EK86}, the process $X$ is Feller and the set $C^\infty\big([0,1]\big)$ forms a core for its generator. It hence suffices to prove uniform convergence of the respective generators applied to polynomials on $[0,1]$. To this end, let $\mathcal{L}$ be the infinitesimal generator of the limiting diffusion $X$, that is,
\begin{equation*}
 \mathcal{L}_X =   \mathcal{L}^{\text{n}}_X + \mathcal{L}^{\theta}_X +\sum_{m=1}^\infty \mathcal{L}^{\sigma_m}_X
\end{equation*}
with
\begin{align}
  \mathcal{L}^{\text{n}}_X \big(f(x)\big) &\coloneqq \frac{1}{2}\, x(1-x)\big(r+\kappa(1-x)\big) f''(x), \label{eq.gen_n_diff} \\
    \mathcal{L}^{\theta}_X \big(f(x)\big) &\coloneqq \theta \nu_{1} (1-x) f'(x)- \theta \nu_{0} x f'(x), \nonumber\\
  \mathcal{L}^{\sigma_m}_X \big(f(x)\big) &\coloneqq - \sigma_m x(1-x^m) f'(x), \nonumber
\end{align}
and $f \in C^3([0,1])$. Furthermore, let 
$\mathcal{L}_{\bX ^{(N)}}$ be the generator of the rescaled process $\bX ^{(N)}\coloneqq (X^{(N)}_{Nt}/{N})_{t \geq 0}$.
For $\mathcal{L}_{X^{(N)}}$ from \eqref{eq.gen_Momo} and for all $n \in [N]$, an elementary calculation shows that
\begin{equation}\label{eq.connection}
    \mathcal{L}_{\bX^{(N)}}\Big( f\Big( \frac{n}{N}\Big) \Big) = N \, \mathcal{L}_{X^{(N)}} \Big( f\Big( \frac{n}{N}\Big) \Big),
\end{equation}
that is, the collection of processes $\bX^{(N)}$ forms a density-dependent family in the sense of \cite[p. 455]{EK86}. To conclude the diffusion limit, we need to show that, for any $f \in C^3([0,1])$, 
\begin{equation*}
    \lim_{N \to \infty}\sup_{n \in [N]} \Big \vert \mathcal{L}_{\bX^{(N)}} f \Big(\frac{n}{N}\Big)-\mathcal{L}^{}_{X} f \Big(\frac{n}{N}\Big)\Big \vert =0.
\end{equation*}
Since the uniform convergence of $\mathcal{L}^{u}_{\bX^{(N)}}$ to $\mathcal{L}^{\theta}_X$ and $\mathcal{L}^{ s_m}_{\bX^{(N)}}$ to $\mathcal{L}^{\sigma_m}_{X}$ can be found in \cite[Section 9]{BEH23}, it remains to verify
the respective claim for $\mathcal{L}_{\overline{X}^{(N)}}^\text{n}$ and $\mathcal{L}_{X}^\text{n}$. We use \eqref{eq.gen_n}, \eqref{eq.lambda_eta}, \eqref{eq.gen_n_diff}, \eqref{eq.connection}, and a second-order Taylor-expansion of $f((n \pm 1)/N)$ around $f(n/N)$ to deduce
\begin{align*}
  \Big \vert \mathcal{L}^{\text{n}}_{\bX^{(N)}} f \Big(\frac{n}{N}\Big)-\mathcal{L}^\text{n}_{X} f \Big(\frac{n}{N}\Big)\Big \vert \!=&\Big \vert N \, \lambda_n^{(N)} \Bigg( f\Big( \frac{n+1}{N} \big) - f\Big( \frac{n}{N} \Big) \!\!\Bigg) 
+ N \, \eta_n^{(N)} \Bigg( f\left( \frac{n-1}{N} \right) - f \Big( \frac{n}{N} \Big) \!\! \Bigg)\\
 & \quad  - \frac{1}{2} \frac{n}{N} \frac{N-n}{N} \Big( r+\kappa \frac{N-n}{N}\Big) f''\Big(\frac{n}{N}\Big)\Big \vert \\
 \quad =& \Big \vert N \big(\lambda^{(N)}-\eta^{(N)}\big)  f'\Big( \frac{n}{N}\Big) + \frac{1}{2} \big(\lambda^{(N)}+\eta^{(N)}\big) f''\Big( \frac{n}{N} \Big)  \\
    & \quad  - \frac{1}{2} \frac{n}{N} \frac{N-n}{N} \Big( r+\kappa \frac{N-n}{N}\Big) f''\Big( \frac{n}{N} \Big) + \mathcal{O}(N^{-3}) \Big \vert \\
 \quad =& \vert \mathcal{O}(N^{-3})\vert, 
\end{align*}
where the last equality follows from $\lambda^{(N)}=\eta^{(N)}$ and $\lambda^{(N)}+\eta^{(N)}=\frac{n}{N}\frac{N-n}{N}(r+\kappa \frac{N-n}{N})$.
\end{proof}

\subsection{Diffusion scaling limit of the factorial moment dual}\label{subsec.difflimit_R}
This section proves that $Z^{(N)}$ from Def.~\ref{def.factorial_dual}, and hence the $R$-counting process $(n_R(\Phi_t))_{t\geq 0}$, has the process $Z$, characterized by \eqref{eq.diff_limit_Z} as its diffusion limit. We start with a helpful proposition.

\begin{proposition}\label{prop.convergence_rates}
  Let $\Ls_{Z^{(N)}}$ and $\Ls_Z$ denote the generators of the processes $Z^{(N)}$ and $Z$. Let $\Nb_{0, \Delta} \coloneqq \Nb \cup \{0, \Delta\}$ and let $f:\Nb_{0, \Delta}\to \Rb$ be a bounded function and denote its restriction to $[N]_{0, \Delta}$ by $f{\vert _N}$.
Then, for $k \in \Nb$ and $N \geq k$, we have
\begin{equation*}
    \lim_{N \to \infty }\sup_{n \in [k]} \big\vert N\Ls_{Z^{(n)}}f\big\vert_N(n)-\Ls_Z f(n)\big\vert= 0.
\end{equation*}
\end{proposition}
\begin{proof}
    The generators follow from Definitions~\ref{def.factorial_dual} and Equation \eqref{eq.diff_limit_Z}, which yield
    \begin{equation*}
        \Ls_{Z^{(N)}} = \Ls_{Z^{(N)}}^{\text{n}} + \Ls^{u}_{Z^{(N)}}+\sum_{m= 1}^\infty \Ls_{Z^{(N)}}^{s_m} 
    \end{equation*}
    with
    \begin{align*}
\Ls_{Z^{(N)}}^{\text{n}}\big(f(n)\big)&= \frac{\kappa}{N} \binom{n}{2} \frac{N-n}{N}\big( f(n+1)-f(n) \big)\\
&\qquad + \frac{1}{N} \binom{n}{2}\Big(r +\kappa \frac{N-(n-1)}{N}\Big)\big( f(n-1)-f(n) \big), \\
        \Ls_{Z^{(N)}}^{s_m}\big(f(n)\big) &= s_m^{(N)} \, n \sum_{j=1}^m \, p_{mj}^n\big( f(n+j)-f(n) \big), \\
        \Ls^{u}_{Z^{(N)}} \big(f(n)\big) &= u^{(N)} \nu_0 n \big( f(\Delta)-f(n)\big) + u^{(N)} \nu_1 n \big( f(n-1)-f(n) \big)
    \end{align*}
    as the generator of the process $Z^{(N)}$; and similarly,
\begin{equation*}
        \Ls_Z = \Ls_Z^{\text{c}} + \Ls_Z^{\text{pb}} + \Ls^\theta_Z+ \sum_{m =1}^\infty\Ls_Z^{\sigma_m} 
    \end{equation*}
    with
    \begin{align*}
        \Ls_Z^{\text{c}}\big(f(n)\big)&= \binom{n}{2}(r+\kappa)\big( f(n-1)-f(n) \big), \\
        \Ls_Z^{\text{pb}}\big(f(n)\big)&= \kappa \binom{n}{2} \big( f(n+1)-f(n) \big), \\
        \Ls_Z^{\sigma_m}\big(f(n)\big) &= \sigma_m n \big( f(n+m)-f(n) \big), \\
        \Ls^\theta_Z \big(f(n)\big) &= \theta \nu_0 n \big( f(\Delta)-f(n)\big) + \theta \nu_1 n \big( f(n-1)-f(n) \big)
    \end{align*}
    as the generator of the process $Z$.
The convergence
\begin{equation*}
   \lim_{N \to \infty} \vert N \Ls_{Z^{(N)}}^{u}f\big\vert_N(n)-\Ls_Z^{\theta} f(n)\vert=0
\end{equation*}
follows from the assumption $Nu^{(N)}\to \theta$ as $N \to \infty$ and the convergence of 
\begin{equation*}
   \lim_{N \to \infty} \sup_{n \in [k]}\vert N \Ls_{Z^{(N)}}^{s_m}f\big\vert_N(n)-\Ls_Z^{\sigma_m} f(n)\vert=0
\end{equation*}
is covered in \cite[Lemma 9.1]{BEH23}. It remains to show the statement for the parts arising from (interactive) neutral events $\Ls^\text{n}_{Z^{(N)}}$ and their translation into coalescence $\Ls^{\text{c}}_Z$ and pairwise branching $\Ls^{\text{pb}}_Z$. For $n \in [k]$, we have
\begin{align*}
&\sup_{n \in [k]}\big\vert N \Ls_{Z^{(N)}}^{\text{n}} f\big\vert_N(n)-\big(\Ls_Z^{\text{c}}+\Ls_Z^{\text{pb}}\big) f(n)\big\vert \\
& \quad = \sup_{n \in [k]}\big \vert -\kappa \frac{n}{N} \binom{n}{2} \big( f(n+1)-f(n)\big)  -\kappa \frac{n-1}{N}\binom{n}{2}\big(f(n-1)-f(n)\big) \big \vert \\
& \quad \leq \, \kappa\, \binom{k}{2}\, C \,\frac{2k-1}{N} \to 0 \quad \text{as }N\to \infty,
\end{align*}
which concludes the proof. 
\end{proof}

We can now state the diffusion limit of the $R$-counting process $(n_r(\Phi_t))_{t\geq 0}$.

\begin{proposition}\label{diffusion_limit_R}
Let $Z^{(N)}$ be the factorial moment dual from Def.~\ref{def.factorial_dual} and $Z$ the process from Eq.~\eqref{eq.diff_limit_Z}. For $n \in \Nb$, define $\bar{Z}^{(N)}\coloneqq (Z_{Nt}^{(N)})_{t \geq 0}$ as the time-rescaled version of $Z^{(N)}$. Assume that $\bar{Z}_0^{(N)}\to Z_0$ in distribution as $N \to \infty$. Then, as $N \to \infty$, $\bar{Z}^{(N)}$ converges in distribution to $Z$.
\end{proposition}
\begin{proof}
The statement is proven analogously to \cite[Prop. 5.3]{Cordero17} and \cite[Prop. 2.15]{BEH23} and relies on Prop.~\ref{subsec.difflimit_R} and the non-explosion of the limit $Z$, see \cite[Theorem 1]{GPP21}. 
\end{proof}

As a direct consequence of Props. \ref{prop.diffusion_limit} and \ref{diffusion_limit_R}, we obtain the moment duality from Cor. \ref{corollary_moment_duality}; the proof is perfectly analogous to \cite[Proof of Thm. 2.16]{BEH23}.

\section{Appendix: sampling with and sampling without replacement}
Let us collect here the arguments leading to the combinatorial factor $p^n_{mj}$ of 
\eqref{eq.pnmj} and its relationship with $q^n_{mj}$ from \eqref{eq.qnjk}. Recall that $p_{mj}^n$ is defined as the probability that, when sampling $m$ lines with replacement from the set $[N]$, exactly $j$ distinct lines are not in the set $[n]$. For convenience of the reader, we recapitulate the argument for \eqref{eq.pnmj} from \cite[Proof of Proposition 2.2]{BEH23}. To obtain~$j\in[m]$ distinct lines that are not in $[n]$, we have to place the $m$ distinguishable marks on the $N$ distinguishable lines such that exactly $j$ out of the $N-n$ lines not in $[n]$ receive at least one mark. There are $ C^n_{mj}$ such possibilities, with $ C^n_{mj}$ as in \eqref{eq.pnmj}. To see this, observe that there are $(N-n)^{\underline{j}}$ possibilities to choose (without replacement) $j$ out of the $N-n$ lines not in $[n]$. For each of these possibilities, we must place some number $\ell\in \{j,\ldots,m\}$ of the $m$ marks on these $j$ lines, and the remaining $m-\ell$ ones on the~lines in $[n]$. We therefore sum over all possibilities to select~$\ell$ out of the~$m$ marks; and for every such $\ell$, there are $\binom{m}{\ell}$ such ways. For every such possibility, in turn, there are $\genfrac\{\}{0pt}{}{\ell}{j}$ ways to partition~$\ell$ marks into the~$j$ selected lines. Finally, for the remaining $m-\ell$ marks, there are $n^{m-\ell}$ ways to place them on the lines in $[n]$ and each configuration has probability $N^{-m}$. Hence, we obtain~\eqref{eq.pnmj}.

Let us now verify that the two sums in \eqref{eq.ftw.with_without} are equal. Inserting $\tilde s_j$ from \eqref{eq.stilde} and $q^n_{jk}$ from \eqref{eq.qnjk} into the sum on the left-hand side of \eqref{eq.ftw.with_without}, we get
\begin{equation}\label{eq.lhs}
\begin{split}
\sum_{j=k}^N \tilde s_j q^n_{jk} & = \sum_{j=k}^N \sum_{m=j}^\infty s_m \frac{N^{\underline j}}{N^m} \genfrac\{\}{0pt}{}{m}{j} \frac{\binom{N-n}{k}\binom{n}{j-k}}{\binom{N}{j}} \\ & = \sum_{m=k}^\infty \frac{s_m}{N^m} \sum_{j=k}^{\min\{m,N\}} N^{\underline j} \genfrac\{\}{0pt}{}{m}{j} \frac{\binom{N-n}{k}\binom{n}{j-k}}{\binom{N}{j}},
\end{split}
\end{equation}
where we have changed summation in the second step. Likewise, inserting $p^n_{mj}$ of \eqref{eq.pnmj} into the second sum in \eqref{eq.ftw.with_without} gives
\begin{equation}\label{eq.rhs}
\sum_{m=k}^\infty s_m p^n_{mk} =  \sum_{m=k}^\infty \frac{s_m}{N^m} C^n_{mk}.
\end{equation}
Comparing coefficients in the right-hand sides of \eqref{eq.lhs} and \eqref{eq.rhs} reveals that we have to show that
\[
 \sum_{j=k}^{\min\{m,N\}} N^{\underline j} \genfrac\{\}{0pt}{}{m}{j} \frac{\binom{N-n}{k}\binom{n}{j-k}}{\binom{N}{j}} = C^n_{mk}.
\]
But this is indeed true: assume that $m$ marks have been sampled without replacement. The left-hand side then counts all possibilities to first select some $j \in \{k, \ldots, \min\{m,N\}\}$ different lines on which to place the $m$ marks (for every $j$, there are $N^{\underline{j}} \genfrac\{\}{0pt}{}{m}{j}$ such possibilities) times the proportion of cases where, out of these $j$ lines, $k$ are not in $[n]$ (this proportion is $\binom{N-n}{k}\binom{n}{j-k}/\binom{N}{j}$). The result is $C^n_{mk}$, the number of possibilities to choose $k \in [m]$ lines that are not in $[n]$, as explained above.

\begin{acks}
All authors gratefully acknowledge financial support by the German Research Foundation (DFG) -- Project-ID 317210226 -- SFB 1283.
\end{acks}

\end{document}